\documentclass{amsart}
\usepackage{amscd}
\usepackage{amssymb}
\usepackage{stmaryrd}
\usepackage{times}
\usepackage{mathrsfs}

\usepackage{color}
\usepackage{booktabs}
\usepackage[all]{xy}
\CompileMatrices

\usepackage[colorlinks=true, pdfstartview=FitV, linkcolor=blue, pagebackref, citecolor=blue, urlcolor=blue]{hyperref}

\makeatletter
\@namedef{subjclassname@2020}{\textup{2020} Mathematics Subject Classification}
\makeatother

\newtheorem{thm}[equation]{Theorem}
\newtheorem{lem}[equation]{Lemma}

\newtheorem{cor}[equation]{Corollary}
\newtheorem{prop}[equation]{Proposition}

\newtheorem*{thm*}{Theorem}
\newtheorem*{prop*}{Proposition}
\newtheorem*{cor*}{Corollary}
\newtheorem*{lem*}{Lemma}
\newtheorem*{MT*}{Main Theorem}

\newtheorem*{ques*}{Question}

\theoremstyle{definition} %
\newtheorem{defn}[equation]{Definition}
\newtheorem*{defn*}{Definition}

\newtheorem{eg}[equation]{Example}

\theoremstyle{remark} %
\newtheorem{rmk}[equation]{Remark}

\newtheorem*{rmk*}{Remark}
\newtheorem*{rmks*}{Remarks}

\DeclareMathOperator{\Tr}{Tr}

\newcommand{\jprod}{\bullet}

\DeclareMathOperator{\car}{char}
\DeclareMathOperator{\Iso}{Iso}

\DeclareMathOperator{\Isom}{Isom}
\DeclareMathOperator{\Lie}{Lie}
\DeclareMathOperator{\Pic}{Pic}
\DeclareMathOperator{\Hom}{Hom}

\DeclareMathOperator{\Max}{Max}

\newcommand{\Pol}{\mathscr{P}}

\newcommand{\eps}{\varepsilon}

\newcommand{\grd}{\nabla}  % directional derivative symbol

\newcommand{\trans}{\top}

\newcommand{\Frd}{Q}  % symbol for our new kind of Freudenthal triple system

\newcommand{\CC}{\mathbb{C}}
\newcommand{\HH}{\mathbb{H}}

\newcommand{\OO}{\mathbb{O}}

\newcommand{\ZZ}{\mathbb{Z}}
\newcommand{\IZ}{\ZZ}

\newcommand{\QQ}{\mathbb{Q}}
\newcommand{\IQ}{\QQ}

\newcommand{\RR}{\mathbb{R}}
\newcommand{\FF}{\mathbb{F}}
\newcommand{\IR}{\RR}

\newcommand{\La}{\Lambda}

\newcommand{\oct}{\mathcal{O}}

\newcommand{\eff}{\mathsf{F}}
\newcommand{\gee}{\mathsf{G}}
\newcommand{\ay}{\mathsf{A}}
\newcommand{\bee}{\mathsf{B}}

\newcommand{\cee}{\mathsf{C}}
\newcommand{\ee}{\mathsf{E}}

\newcommand{\Gm}{\mathbb{G}_{\mathrm{m}}}

\newcommand{\nat}{\sharp}  % the sharp operator

\DeclareMathOperator{\Sq}{Sq}

\newcommand{\Hind}{H^1_{\mathrm{ind}}}

\newcommand{\stbtmat}[4]{\left( \begin{smallmatrix} #1 & #2 \\ #3 & #4 \end{smallmatrix} \right)}  % small 2x2 matrix

\newcommand{\sjordmat}[6]{\left( \begin{smallmatrix} #1 & #6 & \cdot \\ \cdot & #2 & #4 \\ #5 & \cdot & #3 \end{smallmatrix} \right)}
\newcommand{\sbasjord}{\sjordmat{\alpha_1}{\alpha_2}{\alpha_3}{c_1}{c_2}{c_3}}
\newcommand{\sbasGjord}{\left( \begin{smallmatrix} \alpha_1 & \gamma_2 c_3 &  \gamma_3 \bar{c}_2 \\ \gamma_1  \bar{c}_3 & \alpha_2 & \gamma_3 c_1 \\ \gamma_1 c_2  &  \gamma_2 \bar{c}_1 & \alpha_3 \end{smallmatrix} \right)}

\newcommand{\alg}[1]{{#1\text{-\textbf{alg}}}}

\newcommand{\Ralg}{\alg{R}}

\newcommand{\Rx}{R^\times}

\newcommand{\iso}[2]{#1^{(#2)}}
\newcommand{\uiso}[1]{\iso{#1}{u}}

\DeclareMathOperator{\Aut}{Aut}

\DeclareMathOperator{\Mat}{Mat}
\DeclareMathOperator{\Nil}{Nil}
\DeclareMathOperator{\Spec}{Spec}

\newcommand{\bfIsom}{\mathbf{Isom}}
\newcommand{\bfIso}{\mathbf{Iso}}
\newcommand{\bfAut}{\mathbf{Aut}}
\newcommand{\bfG}{\mathbf{G}}
\newcommand{\bfGL}{\mathbf{GL}}
\newcommand{\bfX}{\mathbf{X}}
\newcommand{\bfM}{\mathbf{M}}

\newcommand{\bfm}{\mathbf{m}}
\newcommand{\bfp}{\mathbf{p}}

\newcommand{\idl}{\mathfrak{a}}

\newcommand{\mfm}{\mathfrak{m}}

\newcommand{\mfp}{\mathfrak{p}}
\newcommand{\mfq}{\mathfrak{q}}

\newcommand{\qform}[1]{{\left\langle{#1}\right\rangle}}                   % a quadratic form
\DeclareMathOperator{\PGL}{PGL}
\DeclareMathOperator{\PGSp}{PGSp}
\DeclareMathOperator{\Her}{Her}

\DeclareMathOperator{\End}{End}
\DeclareMathOperator{\Id}{Id}

\DeclareMathOperator{\Zor}{Zor}

\newcommand{\W}{\mathbf{W}}  % module into a sheaf
\DeclareMathOperator{\GL}{GL}
\newcommand{\Sym}{\mathsf{S}}

\numberwithin{equation}{section}

\begin{document}

\title{Albert algebras over $\ZZ$ and other rings}

\author[S. Garibaldi]{Skip Garibaldi}

\author[H.P. Petersson]{Holger P. Petersson}

\author[M.L. Racine]{Michel L. Racine}

% 17C40 = exceptional Jordan structures
% 17C30 = automorphism groups of Jordan algebras
% 20G41 = exceptional groups

\begin{abstract}
Albert algebras, a specific kind of Jordan algebra, are naturally distinguished objects among commutative non-associative algebras and also arise naturally in the context of simple affine group schemes of type $\eff_4$, $\ee_6$, or $\ee_7$.  We study these objects over an arbitrary base ring $R$, with particular attention to the case $R = \ZZ$.
We prove in this generality results previously in the literature in the special case where $R$ is a field of characteristic different from 2 and 3.
\end{abstract}

\maketitle

\setcounter{tocdepth}{1}
\tableofcontents
%
%%%%%%%%%%%%%%%%%%%%%%%%%%%%%%%%%%%%%%%%%%%%%%%%%%%%%%%%%%
\section{Introduction}

Albert algebras, a specific kind of Jordan algebra, are naturally distinguished objects among commutative non-associative algebras and also arise naturally in the context of simple affine group schemes of type $\eff_4$, $\ee_6$, or $\ee_7$.  We study these objects over an arbitrary base ring $R$, with particular attention to the case $R = \ZZ$.
We prove in this generality results previously in the literature in the special case where $R$ is a field of characteristic different from 2 and 3.

\subsection*{Why Albert algebras?}
In the setting of semisimple algebraic groups over a field, a standard technique for computing with elements of a group --- especially an anisotropic group --- is to interpret the group in terms of automorphisms of some algebraic structure, such as viewing an adjoint group of type $\bee_n$ as the special orthogonal group of a quadratic form of dimension $2n+1$, or an adjoint group of type $\ay_n$ as the automorphism group of an Azumaya algebra of rank $(n+1)^2$.  This approach can be seen in many references, from \cite{Weil}, through \cite{KMRT} and \cite{Conrad:Z}.  In this vein, Albert algebras appear as a natural tool for computations related to $\eff_4$, $\ee_6$, and $\ee_7$ groups, as we do below.

In the setting of nonassociative algebras, Albert algebras arise naturally.  Among commutative not-necessarily-associative algebras under additional mild hypotheses (the field has characteristic $\ne 2, 3, 5$ and the algebra is metrized), every algebra satisfying a polynomial identity of degree $\le 4$ is a Jordan algebra, see \cite[Prop.~A.8]{ChayetG}.  Jordan algebras have an analogue of the Wedderburn-Artin theory for associative algebras \cite[p.~201, Cor.~2]{Jac:J}, and one finds that all the simple Jordan algebras are closely related to associative algebras (more precisely, ``are special'') except for one kind, the Albert algebras, see for example \cite[p.~210, Th.~11]{Jac:J} or \cite{McCZ}.

\subsection*{Our contribution}
In the setting of nonassociative algebras, we prove a classification of Albert algebras over $\ZZ$ (Theorem \ref{skip.AlbZ}), which was viewed as an open question in the context of nonassociative algebra; here we see that it is equivalent to the classification of groups of type $\eff_4$, which was known, see \cite{Conrad:Z} which leverages   \cite{Gross:Z} and \cite{ElkiesGross}.  We also prove new results about ideals in Albert algebras (Theorem \ref{IDCUJM}), about isotopy of Albert algebras over semi-local rings (Theorem \ref{local.iso}), and about the number of generators of an Albert algebra (Proposition \ref{generators}).  We have not seen Lemma \ref{round} in the literature, even in the case of a base field of characteristic different from 2 and 3.  

In the setting of
affine group schemes, we present  Albert algebras in a streamlined way in Definition \ref{Frd.def}.  Note that this definition is in the context of what was formerly called a ``quadratic'' Jordan algebra --- because instead of a bilinear multiplication one has a quadratic map, the $U$-operator --- and that it makes sense whether or not 2 is invertible in the base ring.  Applying this definition here allows one to replace, in some proofs, ``global'' computations over $\ZZ$ as one finds in \cite{Conrad:Z} with ``local'' computations over an algebraically closed field that exist in several places in the literature (see, for example, the proof of Lemma \ref{skip.aut.alb}).  We also interpret a clever computation in \cite{ElkiesGross} as an example of a general mechanism known as isotopy, see Definition \ref{EG.def}.  Our classification of groups of type $\ee_7$ over $\ZZ$ in Proposition \ref{E6E7} uses general techniques to reduce the problem to computations over $\RR$.

\subsection*{A different definition}
The definition of Albert algebra over a ring $R$ given here (Definition \ref{Frd.def}) is in the context of para-quadratic algebras as recalled at the beginning of \S\ref{jord.sec} --- such an algebra is an $R$-module $M$ with a distinguished element $1_M$ and a quadratic map $U \!: M \to \End_R(M)$ such that $U_{1_M} = \Id_M$.  There are no other axioms to check.  We then define a specific para-quadratic algebra, $\Her_3(\Zor(R))$, in Example \ref{her.quad}, and define an Albert $R$-algebra $J$ to be a para-quadratic $R$-algebra such that $J \otimes S \cong \Her_3(\Zor(R)) \otimes S$ for some faithfully flat $R$-algebra $S$.

A different approach is taken by references such as \cite[\S6.1]{Ptr:surv} or \cite{Alsaody}.  They define an Albert $R$-algebra to be a cubic Jordan $R$-algebra (Definition \ref{cubjord.def}) $J$ whose underlying $R$-module is projective of rank 27 and $J \otimes F$ is a simple algebra for every homomorphism $R$ to a field $F$.  This definition involves axioms that in principle need to be verified over all $R$-algebras.  The two definitions give the same objects.  Theorem 17 in \cite{Ptr:surv} states that an Albert algebra in the sense of that paper is an Albert algebra in the sense of this paper by producing a faithfully flat $R$-algebra; the proof is not written but closely follows an argument for octonion algebras from \cite{LoosPR}.  For the converse, an Albert algebra in the sense of this paper has a projective underlying module (because $\Her_3(\Zor(R))$ does), is a cubic Jordan algebra (Proposition \ref{Frd.cubic}), and satisfies the simplicity condition (Corollary \ref{Frd.simple}).  

%\subsection*{Comparison to other works}
%From the standpoint of semisimple affine group schemes, we focus our attentions on those of type $\gee_2$, $\eff_4$, $\ee_6$, and $\ee_7$.  Readers who are interested in type $\dee_4$ should may wish to consult \cite{AlsaodyGille}.

\subsection*{Acknowledgments} 
We thank the many kind readers who provided suggestions, comments, or encouragement.  The development of this paper was influenced in various ways by \cite{Conrad:Z}, even when it is not explicitly cited.

\section{Notation}
Rings, by definition, have a 1.  We put $\alg{\ZZ}$ for the category of commutative rings, where $\ZZ$ is an initial object.  For any $R \in \alg{\ZZ}$, we put $\alg{R}$ for the category of pairs $(S, f)$ with $S \in \alg{\ZZ}$ and $f \!: R \to S$, i.e., the coslice category $R \downarrow \alg{\ZZ}$.  Below, $R$ will typically denote an element of $\alg{\ZZ}$.  (The interested reader is invited to mentally replace $R$ by a base scheme $X$, $\alg{R}$ with the category of schemes over $X$, finitely generated projective $R$-modules with vector bundles over $X$, etc., thereby translating results below into a language closer to that in \cite{CalmesFasel}.)

We write $\Mat_n(R)$ for the ring of $n$-by-$n$ matrices with entries from $R$; its invertible elements form the group $\GL_n(R)$.  We write $\qform{\alpha_1, \ldots, \alpha_n} \in \Mat_n(R)$ for the diagonal matrix whose $(i,i)$-entry is $\alpha_i$.

Suppose now that $\bfG$ is a finitely presented group scheme over $R$.  For each fppf $S \in \alg{R}$, we write $H^1(S/R, \bfG)$ for the set of isomorphism classes of $\bfG$-torsors over $R$ trivialized by $S$, see for example \cite{Giraud}, \cite{Wa}, or \cite[\S2.2]{CalmesFasel}.  It does not depend on the choice of structure homomorphism $R \to S$ \cite[Rem.~III.3.6.5]{Giraud}.  The subcategory of fppf elements of $\alg{R}$ has a small skeleton, so the colimit
\[
H^1(R, \bfG) := \varinjlim_{\text{fppf $S \in \alg{R}$}} H^1(S/R, \bfG)
\]
is a set.  It is the non-abelian fppf cohomology of $\bfG$.  
In case $\bfG$ is smooth, it agrees with \'etale $H^1$.  If additionally $R$ is a field, then it agrees with the non-abelian Galois cohomology defined in, for example, \cite{SeCG}.

\subsection*{Unimodular elements}
Let $M$ be an $R$-module.  An element $m \in M$ is said to be \emph{unimodular} if $Rm$ is a free $R$-module of rank 1 and a direct summand of $M$, equivalently, if there is some $\lambda \in M^*$ (the dual of $M$) such that $\lambda(m) = 1$.  When $M$ is finitely generated projective, this is equivalent to: $m \otimes 1$ is not zero in $M \otimes F$ for every field $F \in \Ralg$, see for example \cite[0.3]{Loos:genalg}.  

If $m \in M$ is unimodular, then so is $m \otimes 1 \in M \otimes S$ for every $S \in \Ralg$.  In the opposite direction, 
if $M$ is finitely generated projective, $S$ is a cover of $R$ (i.e., $\Spec S \to \Spec R$ is surjective), and $m \otimes 1$ is unimodular in $M \otimes S$, it follows that $m$ is unimodular as an element of $M$.

\section{Background on polynomial laws} \label{directional.sec}

We may identify an $R$-module $M$ with a functor $\W(M)$ from $\alg{R}$ to the category of sets defined via $S \mapsto M \otimes S$.  For $R$-modules $M$, $N$, 
a polynomial law (in the sense of \cite{Roby}) $f \colon \W(M) \to \W(N)$ is a morphism of functors, i.e., a collection of set maps $f_S \colon M \otimes S \to N \otimes S$ varying functorially with $S$.   We put $\Pol_R(M, N)$ for the collection of polynomial laws $\W(M) \to \W(N)$, and omit the subscript $R$ when it is understood.

A polynomial law is \emph{homogeneous of degree} $d \ge 0$ if $f_S(sx) = s^d f_S(x)$ for every $S \in \alg{R}$, $s \in S$, and $x \in M \otimes S$, see \cite[p.~226]{Roby}.  A \emph{form of degree $d$} on $M$ is a polynomial law $\W(M) \to \W(R)$ that is homogeneous of degree $d$.  The  forms of degree 0 are constants, i.e., given by an element of $R$.
Those of degree 1 are $R$-linear maps $M \to R$.   Those of degree 2 are commonly known as quadratic forms on $M$.    We put $\Pol^d_R(M, N)$ for the submodule of $\Pol_R(M, N)$ of polynomial laws that are homogeneous of degree $d$.

It is often useful to argue that a polynomial law $f$ is zero, which a priori means checking a condition for all $S \in \Ralg$.  However, it suffices to verify that $f_T = 0$ for every \emph{local} ring $T \in \Ralg$.  Indeed, for $m \in M \otimes S$, $f_S(m) = 0$ in $N \otimes S$ if and only if $f_S(m) \otimes 1 = f_{S_\mfm}(m \otimes 1) = 0$ in $N \otimes S_\mfm$ for every maximal ideal $\mfm$ of $S$.

\begin{lem} \label{law.uni}
Let $M$ be a finitely generated projective $R$-module, and suppose $f \in \Pol(M, R)$ is   such that $f(0) = 0$.  If $m \in M$ has $f(m) \in R^\times$, then $m$ is unimodular.
\end{lem}

\begin{proof}
 If $m$ is not unimodular, then there is a field $F \in \Ralg$ such that $m \otimes 1 = 0$ in $M \otimes F$, and $f(m \otimes 1) = 0$, whence $f(m)$ belongs to the kernel of $R \to F$, a contradiction.
\end{proof}

\subsection*{Directional derivatives}
For $f \in \Pol(M, N)$, $v \in M$, and $t$ an indeterminate $n \ge 0$, we define a polynomial law $\grd_v^n f$ as follows.  For $S \in \alg{R}$ and $x \in M \otimes S$, 
$f_{S[t]}(x + v \otimes t)$ is an element of $N \otimes S[t]$, and we define $\grd_v^n f_S(x) \in N \otimes S$ to be the coefficient of $t^n$.   This defines a polynomial law called the \emph{$n$-th directional derivative $\grd_v^n f$ of $f$ in the direction $v$}.  One finds that $\grd^0_v f = f$ regardless of $v$.   We abbreviate $\grd_v f := \grd_v^1 f$; it is linear in $v$.

 If $f$ is homogeneous of degree $d$ and $0 \le n \le d$, then $\grd_v^n f(x)$ is homogeneous of degree $d - n$ in $x$ and degree $n$ in $v$.  The symmetry implicit in the definition of the directional derivative gives $\grd_v^n f(x) = \grd_x^{d-n} f(v)$ for $x \in M$. 
 
\begin{lem} \label{law.dense}
Suppose $M$, $N$ are $R$-modules and $A$ is a unital associative $R$-algebra and $g \in \Pol(M , A)$ is a polynomial law such that there is an element $m \in M$ such that $g(m) \in A$ is invertible.  If $f \in \Pol^d(M, N)$  satisfies
\[
g_S(x) \in A_S^\times \Rightarrow f_S(x) = 0
\]
for all $S \in \alg{R}$ and $x \in M \otimes S$, then $f$ is identically zero.
\end{lem}

\begin{proof}
Since the hypotheses are stable under base change, it suffices to show that $f(v) = 0$ for all $v \in M$. Replacing $g$ by $L \circ g \in \Pol(M, A)$, where $L \in \End_R(A)$ is multiplication in $A$ on the left by the inverse of $g(m)$, we may assume $g(m) = 1_A$.  Set $S := R[\eps]/(\eps^{d+1})$. For $v \in M$, the element
\[
g_S(m + \eps v) = 1_A + \sum_{n=1}^d \eps^n \grd_v^n g(m)
\]
is invertible in $A_S$, so by hypothesis,
\[
0 = f_S(m+\eps v) = \sum_{n = 0}^d \eps^n \grd_v^n f(m).
\]
Focusing on the coefficient of $\eps^d$ in that equation gives
\[
0 = \grd_v^d f(m) = \grd_m^0 f(v) = f(v),
\]
as required.
\end{proof}

\subsection*{The module of polynomial laws}
In the following, we write $\Sym^n M$ for the $n$-th symmetric power of $M$, i.e., the $R$-module $\otimes^n M$ modulo the submodule generated by elements $x - \sigma(x)$ for $x \in \otimes^n M$ and $\sigma$ a permutation of the $n$ factors.

\begin{lem} \label{pol}
Let $M$ and $N$ be finitely generated projective $R$-modules.  Then for each $d \ge 0$:
\begin{enumerate}
\item \label{pol.proj} $\Pol^d(M, N)$ is a finitely generated projective $R$-module.
\item \label{pol.base} If $T \in \Ralg$ is faithfully flat, the natural map $\Pol^d_R(M, N) \otimes T \to \Pol^d_T(M \otimes T, N \otimes T)$ is an isomorphism.
\item \label{pol.S} The natural map $\Sym^d(M^*) \otimes N \to \Pol^d(M, N)$ is an isomorphism.
\item \label{pol.Loos} The natural map $\Pol^d(M, R) \otimes N \to \Pol^d(M, N)$ is an isomorphism.
\end{enumerate}
\end{lem}

\begin{proof}
$\Pol^d_R(M, N)$ is naturally isomorphic to $\Hom_R(\Gamma_d(M), N)$ by \cite[Th.~IV.1]{Roby}, where $\Gamma_d(M)$ denotes the module of degree $d$ divided powers on $M$.   Then $\Pol^d_R(M, N) \otimes T \cong \Hom_R(\Gamma_d(M), N) \otimes T$, which in turn is $\Hom_T(\Gamma_d(M) \otimes T, N \otimes T)$ because $T$ is faithfully flat \cite[p.~33, Prop.~II.2.5]{KO}.  Now $\Gamma_d(M) \otimes T \cong \Gamma_d(M \otimes T)$ by  \cite[\S{IV.5}, Exercise 7]{Bou:alg2}, completing the proof of \eqref{pol.base}.

\eqref{pol.S}: If $M$ and $N$ are free modules, then the map is an isomorphism by \cite[p.~232]{Roby}.  If $M$ and $N$ have constant rank, then there is a faithfully flat $T \in \Ralg$ such that $M \otimes T$ and $N \otimes T$ are free.  Since \eqref{pol.S} holds over $T$ by the free case, \eqref{pol.base} and faithfully flat descent give that \eqref{pol.S} holds.  In the general case, since $M$ and $N$ are finitely generated projective, we may write $R = \prod_{i=0}^n R_i$ for some $n$ such that $M = \oplus M_i$ and $N = \oplus N_i$ with each $M_i$, $N_i$ an $R_i$-module of finite constant rank.  Then $\Pol^d(M, N) = \oplus \Pol^d(M_i, N_i)$ and $\Sym^d(M^*) \otimes N = \oplus (\Sym^d(M_i^*) \otimes N_i)$ and the claim follows by the constant rank case.

\eqref{pol.Loos} follows trivially from \eqref{pol.S}.  For \eqref{pol.proj}, note that $M^*$ is finitely generated projective, so so is $\Sym^d(M^*)$ and also the tensor product $\Sym^d(M^*) \otimes N$.   Applying \eqref{pol.S} gives the claim.
\end{proof}

One can create new polynomial laws from old by twisting by a line bundle.

\begin{lem} \label{pol.twisting}
Let $M$ and $N$ be finitely generated projective $R$-modules.  Then for every $d \ge 0$ and 
every line bundle $L$, we have:
\begin{enumerate}
\item \label{pol.disc} There is a natural isomorphism $\Pol^d(M, N) \otimes (L^*)^{\otimes d} \to \Pol^d(M \otimes L, N)$.
\item \label{pol.twist} There is a natural isomorphism $\Pol^d(M, N) \cong \Pol^d(M \otimes L, N \otimes L^{\otimes d})$.
\end{enumerate}
\end{lem}

\begin{proof}
For \eqref{pol.disc}, since $L^*$ is a line bundle, the natural map $(L^*)^{\otimes d} \to \Sym^d(L^*)$ is an isomorphism because it is so after faithfully flat base change.  Since $\Sym^d(M^*) \otimes \Sym^d(L^*)$ is naturally identified with $\Sym^d((M \otimes L)^*)$, combining Lemma \ref{pol}\eqref{pol.S},\eqref{pol.Loos} then gives \eqref{pol.disc}.  

For \eqref{pol.twist}, there are isomorphisms $\Pol^d(M \otimes L, N \otimes L^{\otimes d}) \xrightarrow{\sim} \Pol^d(M, N) \otimes (L^*)^{\otimes d} \otimes L^{\otimes d}$ by  \eqref{pol.disc} and Lemma \ref{pol}\eqref{pol.Loos}.  Since $L^{\otimes d} \otimes (L^*)^{\otimes d} \cong R$, the claim follows.
\end{proof}

\begin{eg} \label{disc.module} (References: \cite[tag 03PK]{stacks-project}, \cite[\S2.4.3]{CalmesFasel}, \cite[\S{III.3}]{Knus:qf})
Suppose $L$ is a line bundle and there is an isomorphism $h \colon L^{\otimes d} \to R$ for some $d \ge 1$. We call such a pair $[L, h]$ a \emph{$d$-trivialized line bundle}.  (In the case $d = 2$ they are sometimes called \emph{discriminant modules}.)  Applying $h$ to identify $N \otimes L^{\otimes d} \xrightarrow{\sim} N$ in Lemma \ref{pol.twisting}\eqref{pol.twist} gives a construction that takes $f \in \Pol^d(M, N)$ and gives an element of $\Pol^d(M \otimes L, N)$, which we denote by $[L, h] \cdot (M, f)$.

For example, for each $\alpha \in R^\times$, define $\qform{\alpha}$ to be $[L, h]$ as in the preceding paragraph, where $L = R$ and $h$ is defined by $h(\ell_1 \otimes \cdots \otimes \ell_d) = \alpha \prod \ell_i$.  Clearly, $\qform{\alpha \beta^d} \cong \qform{\alpha}$ for all $\alpha, \beta \in \Rx$.  Applying the construction in the previous paragraph, we find $\qform{\alpha} \cdot (M, f) \cong (M, \alpha f)$.

Every $[L, h]$ with $L = R$ is necessarily isomorphic to $\qform{\alpha}$ for some $\alpha \in \Rx$.  In particular, if $\Pic(R)$ has no $d$-torsion elements other than zero --- e.g., if $R$ is a semilocal ring or a UFD \cite[tags 0BCH, 02M9]{stacks-project} --- then each $[L, h]$ is isomorphic to $\qform{\alpha}$ for some $\alpha$.

The group scheme $\mu_d$ of $d$-th roots of unity is the automorphism group of each $[L, h]$, where $\mu_d$ acts by multiplication on $L$.  The group $H^1(R, \mu_d)$ classifies pairs $[L, h]$ up to isomorphism.
\end{eg}

We say that homogeneous polynomial laws related by the isomorphism in Lemma \ref{pol.twisting}\eqref{pol.twist} are \emph{projectively similar}, imitating the language from \cite[\S1.2]{AuelBB} for the case of quadratic forms ($d = 2$).  (This relationship was called ``lax-similarity'' in \cite{BalmerCalmes}.)  We say that homogeneous degree $d$ laws $f$ and $[L, h] \cdot f$ for $[L, h] \in H^1(R, \mu_d)$ as in the preceding example are \emph{similar}.  If $\Pic(R)$ has no $d$-torsion elements other than zero, the two notions coincide.

For $f \in \Pol^d(M, N)$, we define $\Aut(f)$ to be the subgroup of $\GL(M)$ consisting of elements $g$ such that $fg = f$ as polynomial laws.  In case $M$ and $N$ are finitely generated projective, so is $\Pol^d(M, N)$, whence the functor $\bfAut(f)$ from $\Ralg$ to groups defined by $\bfAut(f)(T) = \Aut(f_T)$ is a closed sub-group-scheme of $\bfGL(M)$.

\begin{lem} \label{proj.sim}
Let $f$ and $f'$ be homogeneous polynomial laws on finitely generated projective modules.  If $f$ and $f'$ are projectively similar, then their automorphism groups are isomorphic.
\end{lem}

\begin{proof}
By hypothesis, $f \in \Pol^d(M, N)$ and $f' \in \Pol^d(M \otimes L, N \otimes L^{\otimes d})$ for some modules $M$ and $N$; line bundle $L$; and $d \ge 0$.  The group scheme $\bfAut(f)$ is the closed sub-group-scheme of $\bfGL(M)$ stabilizing the element $f$ in $\Sym^d(M^*) \otimes N$.  Now, any element of $\bfGL(M)$ acts on $\Sym^d((M \otimes L)^*) \otimes (N \otimes L^{\otimes d})$ by defining it to act as the identity on $L$.  In this way, we find a homomorphism $\bfAut(f) \to \bfAut(f')$.  Viewing $M$ as $(M \otimes L) \otimes L^*$ and $N$ as $(N \otimes L^{\otimes d}) \otimes (L^*)^{\otimes d}$, and repeating this construction, we find an inverse mapping $\bfAut(f') \to \bfAut(f)$.
\end{proof}

%%%%%%%%%%%%%%%%%%%%%%%%%%%%%%%%%%%%%%%%%%%%%%%%%%%%%%%%%%
\section{Background on composition algebras} \label{comp.sec}

A not-necessarily-associative $R$-algebra $C$ is an $R$-module with an $R$-linear map $C \otimes_R C \to C$, which we view as a multiplication and write as juxtaposition.  Such a $C$ is \emph{unital} if it has an element $1_C \in C$ such that $1_C c = c 1_C = c$ for all $c \in C$.  See e.g.~\cite{Schfr}.  A \emph{composition $R$-algebra} as in \cite{Ptr:compzero} is such a $C$ that is finitely generated projective as an $R$-module, is unital, and has a quadratic form  $n_C \!: C \to R$ that 
allows composition (that is, such that $n_C(xy) = n_C(x)n_C(y)$ for all $x, y \in C$), satisfies $n_C(1_C) = 1$, and whose bilinearization defined by $n_C(x, y) := n_C(x+y) - n_C(x) - n_C(y)$ gives an isomorphism $C \to C^*$ via $x \mapsto n_C(x, \cdot )$.  We say that a symmetric bilinear form with this property is \emph{regular}.  The quadratic form $n_C$ (which is unique by Proposition \ref{UNINOR} below) is called the \emph{norm} of $C$.

\begin{rmk}
In the definition above, one can swap the condition  $n_C(1_C) = 1$ with the requirement that the rank of $C$ is nowhere zero.
\end{rmk}

We put $\Tr_C(x) := n_C(x, 1_C)$, a linear map $C \to R$, called the \emph{trace} of $C$.  Trivially, $\Tr_C(1_C) = 2$.  Lemma \ref{law.uni} gives that $1_C$ is unimodular, so we may identify $R$ with $R1_C$, and  $C$ is a faithful $R$-module.
The unimodularity of $1_C$ is equivalent to the existence of some $\lambda \in C^*$ such that $\lambda(1_C) = 1$, i.e., some $x \in C$ such that $\Tr_C(x) = 1$, whence $\Tr_C \colon C \to R$ is surjective.

The class of composition algebras is stable under base change.  That is, if $C$ is a composition $R$-algebra with norm $n_C$, then for every $S \in \Ralg$, $C \otimes S$ is a composition $S$-algebra with norm $n_C \otimes S$. The following two results are essentially well known \cite[1.2$\--$1.4]{Ptr:compzero}. For convenience, we include their proof.

\begin{lem}[``Cayley-Hamilton'']  \label{PRECOAL}
Let $C$ be a composition algebra with norm $n_C$ and define $\Tr_C$ as above.  Then
\[
x^2 - \Tr_C(x) x + n_C(x) 1_C = 0
\]
for all $x \in C$.
\end{lem}

\begin{proof}
Linearizing the composition law $n_C(xy) = n_C(x) n_C(y)$, we find
\begin{gather} 
n_C(xy, x) = n_C(x) \Tr_C(y)  \quad \text{and} \label{PRECOAL.1} \\
 n_C(xy, wz) + n_C(wy, xz) = n_C(x,w) n_C(y,z) \label{PRECOAL.2}
\end{gather}
for all $x, y, z, w \in C$.  
%Setting $y = x$ and $w = z = 1_C$ in \eqref{PRECOAL.2}, we find:
%\[
%\Tr_C(x^2) + 2 n_C(x) = \Tr_C(x)^2.
%\]
Setting $z = x$ and $w = 1_C$ in \eqref{PRECOAL.2}, we find:
\[
n_C(xy, x) + n_C(y, x^2) = \Tr_C(x) n_C(x, y).
\]
Combining these with \eqref{PRECOAL.1}, we find:
\[
n_C(x^2 - \Tr_C(x) x + n_C(x) 1_C, y) = 0 \quad \text{for all $x, y \in C$.}
\]
Since the bilinear form $n_C$ is regular, the claim follows.
\end{proof}

A priori, a composition algebra is a unital algebra together with a quadratic form, the norm.  The next result shows that this data is redundant.

\begin{prop} \label{UNINOR}
If $C$ is a composition algebra, then the norm $n_C$ is uniquely determined by the algebra structure of $C$.
\end{prop}

\begin{proof}
Let $n' \colon C \to R$ be any quadratic form making $C$ a composition algebra and write $\Tr'$ for the corresponding trace $\Tr'(x) := n'(x + 1_C) - n'(x) - n'(1_C)$.  Then $\lambda := \Tr_C - \Tr'$ (resp., $q := n_C - n'$) is a linear (resp., quadratic) form on $C$ and the Cayley-Hamilton property yields
\begin{equation} \label{DITRNO}
\lambda(x) x = q(x) 1_C \quad \text{for all $x \in C$.}
\end{equation}
We aim to prove that $q = 0$.  Because $1_C$ is unimodular, it suffices to prove $\lambda = 0$.  This can be checked locally, so we may assume that $R$ is local and in particular $C = R1_C \oplus M$ for a free module $M$.  Now, $\Tr_C(1_C) = 2 = \Tr'(1_C)$, so $\lambda(1_C) = 0$.  For $m \in M$ a basis vector, 
 $\lambda(m)m$ belongs to $M \cap R1_C$ by \eqref{DITRNO}, so it is zero, whence $\lambda(m) = 0$, proving the claim.
\end{proof}

\begin{cor}
Let $C$ be a unital $R$-algebra.  If there is a faithfully flat $S \in \Ralg$ such that $C \otimes S$ is a composition $S$-algebra, then $C$ is a composition algebra over $R$.
\end{cor}

\begin{proof}
Because the norm $n_{C \otimes S}$ of $C \otimes S$ is uniquely determined by the algebra structure, one obtains by faithfully flat descent a quadratic form $n_C \colon C \to R$ such that $n_C \otimes S = n_{C \otimes S}$.  Because $n_{C \otimes S}$ satisfies the properties required to make $C \otimes S$ a composition algebra and $S$ is faithfully flat over $R$, it follows that the same properties hold for $n_C$.
\end{proof}

The following facts are standard, see for example \cite[\S{V.7}]{Knus:qf}:  Composition algebras are alternative algebras.  The map $\bar{\ } \!: C \to C$ defined by $\overline{x} := \Tr_C(x)1_C - x$ is an involution, i.e., an $R$-linear anti-automorphism of period 2.
 
\subsection*{Composition algebras of constant rank} 
In case $R$ is connected, a composition $R$-algebra has rank $2^e$ for $e \in \{ 0, 1, 2, 3 \}$ \cite[p.~206, Th.~V.7.1.6]{Knus:qf}.  Therefore,
specifying a composition $R$-algebra $C$ is equivalent to writing 
\begin{equation} \label{decomp}
R = \prod_{e=0}^3 R_e \quad \text{and} \quad C = \prod_{e=0}^3 C_e,
\end{equation}
where $C_e$ is a composition $R_e$-algebra of constant rank $2^e$. 

If $C$ is a composition algebra of rank 1, then since $1_C$ is unimodular, $C$ is equal to $R$.  The bilinear form $n_C(\cdot, \cdot)$ gives an isomorphism $C \to C^*$ and $n_C(1_C, \alpha 1_C) = 2\alpha$, we deduce that $2$ is invertible in $R$.  Conversely, if 2 is invertible, then $R$ is a composition algebra by setting $n_C(\alpha) = \alpha^2$.

A composition algebra whose rank is 2 is not just an associative and commutative ring, it is an \'etale algebra \cite[p.~43, Th.~I.7.3.6]{Knus:qf}.  Conversely, every rank 2 \'etale algebra is a composition algebra.  Among rank 2 \'etale algebras, there is a distinguished one, $R \times R$, which is said to be split.

A composition algebra whose rank is 4 is associative and is an Azumaya algebra, commonly known as a \emph{quaternion algebra}.   (Note that our notion of quaternion algebra is more restrictive than the one in the books \cite[see p.~43]{Knus:qf} and \cite{Voight}.)  Among quaternion $R$-algebras, there is a distinguished one, the 2-by-2 matrices $\Mat_2(R)$, which is said to be split.  

A composition algebras whose rank is 8 is known as an \emph{octonion algebra}.  Among octonion $R$-algebras, there is a distinguished one that is said to be split, called the Zorn vector matrices and denoted $\Zor(R)$, see \cite[4.2]{LoosPR}.  As a module, we view it as $\stbtmat{R}{R^3}{R^3}{R}$ with multiplication
\[
\stbtmat{\alpha_1}{u}{x}{\alpha_2} \stbtmat{\beta_1}{v}{y}{\beta_2} = \stbtmat{\alpha_1 \beta_1 - u^\trans y}%
{\alpha_1 v + \beta_2 u + x \times y}{\beta_1 x + \alpha_2 y + u \times v}{-x^\trans v + \alpha_2 \beta_2}
\]
where $\times$ is the ordinary cross product on $R^3$.  The quadratic form is 
\[
n_{\Zor(R)}\stbtmat{\alpha_1}{u}{x}{\alpha_2} = \alpha_1 \alpha_2 + u^\trans x.
\]

One says that a composition $R$-algebra $C$ is \emph{split} if, when we write $R$ and $C$ as in \eqref{decomp}, $C_e$ is isomorphic to the split composition $R_e$-algebra for $e \ge 1$.

\begin{eg} \label{octaves}
The real octonions $\OO$ are a composition $\RR$-algebra with  basis $1_{\OO}$, $e_1$, $e_2$, $\ldots$, $e_7$ which is orthonormal with respect to the quadratic form $n_\OO$ with multiplication table 
\[
e_r^2 = -1 \quad \text{and} \quad e_r e_{r+1} e_{r+3} = -1
\]
for all $r$ with subscripts taken modulo 7, and the displayed triple product is associative.

The $\ZZ$-sublattice $\oct$ of $\OO$ spanned by $1_\OO$, the $e_r$, and 
\begin{gather*}
h_1 = (1 + e_1 + e_2 + e_4)/2, \quad h_2 = (1 + e_1 + e_3 + e_7)/2, \\
h_3 = (1 + e_1 + e_5 + e_6)/2, \quad \text{and} \quad h_4 = (e_1 + e_2 + e_3 + e_5)/2
\end{gather*}
is a composition $\ZZ$-algebra.  It is a maximal order in $\oct \otimes \QQ$, and all such are conjugate under the automorphism group of $\oct \otimes \QQ$.  (As a consequence, there is some choice in the way one presents this algebra.  We have followed \cite{ElkiesGross}.)  As a subring of $\OO$, it has no zero divisors.  For more on this, see \cite[\S19]{Dickson:hyper}, \cite{Coxeter:cayley}, \cite[\S9]{ConwaySmith}, or \cite[\S5]{Conrad:Z}.  The non-uniqueness of this choice of maximal order and its relationship to other orders like $\ZZ \oplus \ZZ e_1 \oplus \cdots \oplus \ZZ e_7$ can be understood in terms of the Bruhat-Tits building of the group $\Aut(\Zor(\QQ_2))$ of type $\gee_2$ over the 2-adic numbers, compare \cite[\S9]{GanYu:G2}.
\end{eg}

%%%%%%%%%%%%%%%%%%%%%%%%%%%%%%%%%%%%%%%%%%%%%%%%%%%%%%%%%%
\section{Background on Jordan algebras} \label{jord.sec}

\subsection*{Para-quadratic and Jordan algebras} A (unital) \emph{para-quadratic algebra} over a ring $R$ is an $R$-module $J$ together with a quadratic map $U \colon J \to \End_R(J)$ --- i.e., $U$ is an element of $\Pol^2(J, \End_R(J))$ --- called the \emph{$U$-operator}, and a distinguished element $1_J \in J$, such that $U_{1_J} = \Id_J$.
As a notational convenience, we define a linear map $J \otimes J \otimes J \to J$ denoted $x \otimes y \otimes z \mapsto \{ xyz\}$ via 
\begin{equation} \label{brace.def}
\{ x y z \} := (U_{x+z} - U_x - U_z)y.
\end{equation}
Evidently, $\{ x y z \} = \{ z y x \}$ for all $x, y, z \in J$.

A \emph{homomorphism} $\phi \colon J \to J'$ of para-quadratic $R$-algebras is an $R$-linear map such that $\phi(1_J) = 1_{J'}$ and $U'_{\phi(x)} \phi(y) = \phi(U_x y)$ for all $x, y \in J$, where $U'$ denotes the $U$-operator in $J'$.

A para-quadratic $R$-algebra $J$ is a \emph{Jordan $R$-algebra}  if the identities
\begin{equation} \label{jord.id}
U_{U_x y} = U_x U_y U_x \quad \text{and}  \quad
U_x \{ yxz \} = \{ (U_xy) z x \}
\end{equation}
hold for all $x, y, z \in J\otimes S$ for all $S \in \alg{R}$.  (Alternatively, one can define a Jordan $R$-algebra entirely in terms of identities concerning elements of $J$, avoiding the ``for all $S \in \alg{R}$'', at the cost of requiring a longer list of identities, see \cite[\S1]{McC:NAS}.)   Note that if $J$ is a Jordan $R$-algebra, then $J \otimes T$ is a Jordan $T$-algebra for every $T \in \Ralg$ (``Jordan algebras are closed under base change'').  If $J$ is a para-quadratic algebra and $J \otimes T$ is Jordan for some faithfully flat $T \in \Ralg$, then $J$ is Jordan.

 For $x$ in a Jordan algebra $J$ and $n \ge 0$, we define the $n$-th power $x^n$ via
\begin{equation} \label{jord.pow}
x^0 := 1_J, \quad x^1 := x, \quad x^n = U_x x^{n-2} \ \text{for $n \ge 2$.}
\end{equation}
An element $x \in J$ is \emph{invertible with inverse $y$} if $U_x y = x$ and $U_x y^2 = 1$ \cite[\S5]{McC:NAS}.   It turns out that $x$ is invertible if and only if $U_x$ is invertible if and only if $1$ is in the image of $U_x$; when these hold, the inverse of $x$ is $y = U_x^{-1} x$, which we denote by $x^{-1}$. It follows from \eqref{jord.id} that $x,y \in J$ are both invertible if and only if $U_xy$ is invertible, and in this case $(U_xy)^{-1} = U_{x^{-1}}y^{-1}$.

\begin{eg} \label{special.jord}
Let $A$ be an associative and unital $R$-algebra.  Define $U_xy := xyx$ for $x, y \in A$.  Then $\{ xyz\} = xyz + zyx$ and $A$ endowed with this $U$-operator is a Jordan algebra denoted by $A^+$.  Note that for $x \in A$ and $n \ge 0$, the $n$-th powers of $x$ in $A$ and $A^+$ are the same.
\end{eg}

\subsection*{Relations with other kinds of algebras}
Suppose for this paragraph that 2 is invertible in $R$.  Given a para-quadratic algebra $J$ as in the preceding paragraph, one can define a commutative (bilinear) product $\jprod$ on $J$ via
\begin{equation} \label{jprod.def}
x \jprod y := \frac12 \{xy1_J \} \quad \text{for $x, y \in J$.}
\end{equation}
(In the case where $J$ is constructed from an associative algebra as in Example \ref{special.jord}, one finds that $x \bullet y = \frac12 (xy + yx)$.  If additionally the associative algebra is commutative, $\bullet$ equals the product in that associative algebra.)
If $J$ is Jordan, then $\jprod$ satisfies  
\begin{equation} \label{jid.old}
(x \bullet y) \bullet (x \bullet x) = x \bullet (y \bullet (x \bullet x)),
\end{equation}
which is the axiom classically called the ``Jordan identity''.  

In the opposite direction, given an $R$-module $J$ with a commutative product $\jprod$ with identity element $1_J$, we obtain a para-quadratic algebra by setting
\begin{equation} \label{Uop.def}
U_x y := 2x \jprod (x \jprod y) - (x \jprod x) \jprod y \quad \text{for $x, y \in J$}.
\end{equation}
If the original product satisfied the Jordan identity, then the para-quadratic algebra so obtained satisfies \eqref{jord.id}, i.e., is a Jordan algebra in our sense, see for example %\cite[p.~202]{McC} or 
\cite[\S1.4]{Jac:Tata}.

\begin{defn}[hermitian matrix algebras] \label{herm.lin}
Let $C$ be a  composition $R$-algebra and $\Gamma = \qform{\gamma_1, \gamma_2, \gamma_3} \in \GL_3(R)$.  We define $\Her_3(C, \Gamma)$ to be the $R$-submodule of $\Mat_3(C)$ consisting of elements fixed by the involution $x \mapsto \Gamma^{-1} \bar{x}^\trans \Gamma$ and with diagonal entries in $R$.
Note that, as an $R$-module, $\Her_3(C, \Gamma)$ is a sum of 3 copies of $C$ and 3 copies of $R$, so it is finitely generated projective. 

In the special case where 2 is invertible in $R$, one can define a multiplication $\jprod$ on $\Her_3(C, \Gamma)$ via $x \jprod y := \frac12 (xy + yx)$, where juxtaposition denotes the usual product of matrices in $\Mat_3(C)$.  It satisfies the Jordan identity \cite[p.~61, Cor.]{Jac:J}, and therefore the $U$-operator defined via \eqref{Uop.def} makes $\Her_3(C, \Gamma)$ into a Jordan algebra.
\end{defn}

%Writing out $U_x \eps_i$ and $U_x \delta_i^\Gamma(d)$ for $d \in C$ and $x$ as in \eqref{basJord}, one finds expressions that do not involve any denominators, and so make sense over an arbitrary ring $R$.  In this way, $\Her_3(C, \Gamma)$ is a para-quadratic $R$-algebra, whether or not 2 is invertible in $R$.  
%Note that, if we multiply $\Gamma$ by an element of $R^\times$ or any entry in $\Gamma$ by the square of an element of $R^\times$, we obtain an algebra isomorphic to the original.  

%%%%%%%%%%%%%%%%%%%%%%%%%%%%%%%%%%%%%%%%%%%%%%%%%%%%%%%%%%
\section{Cubic Jordan algebras}

In this section, we define cubic Jordan algebras and the closely related notion of cubic norm structure.  They provide a useful alternative language for computation. 

\begin{defn} \label{cubjord.def}
Following \cite{McC:FST} (see \cite[p.~212]{PR:jord3} for the terminology), we define a \emph{cubic norm $R$-structure} as a quadruple $\bfM = (M,1_\bfM,\sharp,N_\bfM)$ consisting of an $R$-module $M$, a distinguished element $1_\bfM \in M$ (the \emph{base point}), a quadratic map $\sharp\colon M \to M$, $x \mapsto x^\sharp$ (the \emph{adjoint}), with (symmetric bilinear) polarization $x \times y := (x + y)^\sharp - x^\sharp - y^\sharp$, a cubic form $N_\bfM\colon M \to R$ (the \emph{norm}) such that the following axioms are fulfilled. Defining a bilinear form $T_\bfM\colon M \times M \to R$ by
\begin{equation}
\label{bilt} T_\bfM(x,y) := (\grd_xN_\bfM)(1_\bfM)(\grd_yN_\bfM)(1_\bfM) - (\grd_x\grd_yN_\bfM)(1_\bfM)
\end{equation}
(the \emph{bilinear trace}), which is symmetric since the directional derivatives $\grd_x$, $\grd_y$ commute \cite[p.~241, Prop.~II.5]{Roby}, and a linear form $\Tr_\bfM\colon M \to R$ by 
\begin{equation}
\label{lit} \Tr_\bfM(x) := T_\bfM(x,1_\bfM)
\end{equation}
(the \emph{linear trace}), the identities
\begin{gather}
\label{cunun} 1_\bfM^\sharp = 1_\bfM, \quad N_\bfM(1_\bfM) = 1,\\
\label{cubad} 1_\bfM \times x = \Tr_\bfM(x)1_\bfM - x,\; (\grd_yN_\bfM)(x) = T_\bfM(x^\sharp,y),\:x^{\sharp\sharp} = N_\bfM(x)x
\end{gather}
hold in all scalar extensions $M \otimes S$, $S\in \alg{R}$. 

For such a cubic norm structure $\bfM$, we then define a $U$-operator by
\begin{equation}
\label{uop.cub} U_xy := T_\bfM(x,y)x - x^\sharp \times y,
\end{equation}
which together with $1_\bfM$ converts the $R$-module $M$ into a Jordan $R$-algebra $J = J(\bfM)$ \cite[Th.~1]{McC:FST}. In the sequel, we rarely distinguish carefully between the cubic norm structure $\bfM$ and the Jordan algebra $J(\bfM)$. By abuse of notation, we write $1_J = 1_\bfM$, $N_J = N_\bfM$, $T_J = T_\bfM$, and $\Tr_J:= \Tr_\bfM$ if there is no danger of confusion, even though, in general, $J$ does not determine $\bfM$ uniquely \cite[p.~216]{PR:jord3}.

A Jordan $R$-algebra $J$ is said to be \emph{cubic} if there exists a cubic norm $R$-structure $\bfM$ as above such that (i) $J = J(\bfM)$ and (ii) $J = M$ is a finitely generated projective $R$-module. With the quadratic form $S_J\colon M \to R$ defined by $S_J(x) := \Tr_J(x^\sharp)$ for $x \in J$ (the \emph{quadratic trace}), the cubic Jordan algebra $J$ satisfies the identities
\begin{gather}
\label{sharp.norm} (U_xy)^\sharp = U_{x^\sharp}y^\sharp, \quad N_J(U_xy)U_xy = N_J(x)^2N_J(y)U_xy, \\
\label{inv.cub} U_xx^\sharp = N_J(x)x, \quad U_x(x^\sharp)^2 = N_J(x)^21_J, \\
\label{sharp.square} x^\sharp = x^2 - \Tr_J(x)x + S_J(x)1_J, \quad \text{and} \\
\label{cub.eq} x^3 - \Tr_J(x)x^2 + S_J(x)x - N_J(x)1_J = 0 = x^4 - \Tr_J(x)x^3 + S_J(x)x^2 - N_J(x)x 
\end{gather} 
for all $x \in J$.
For \eqref{sharp.norm}$\--$\eqref{sharp.square} and the first equation of \eqref{cub.eq}, see \cite[p.~499]{McC:FST}, while the second equation of \eqref{cub.eq} follows from the first, \eqref{inv.cub}, and \eqref{sharp.square} via $x^4 = U_xx^2 = U_xx^\sharp + \Tr_J(x)U_xx - S_J(x)U_x1_J = \Tr_J(x)x^3 - S_J(x)x^2 + N_J(x)x$. 
\end{defn}

\begin{rmk} \label{loc.lin}
Note that the second equality of \eqref{cub.eq} derives from the first through formal multiplication by $x$. But, due to the para-quadratic character of Jordan algebras, this is not a legitimate operation unless $2$ is invertible in $R$. In fact, cubic Jordan algebras exist that contain elements $x$ satisfying $x^2 = 0 \neq x^3$ \cite[1.31--1.32]{Jac:Tata}.
\end{rmk}

\begin{eg}[3-by-3 matrices] \label{Mat3}
We claim that $\Mat_3(R)^+$ is a cubic Jordan algebra, in particular it is $J(\bfM)$ for $\bfM := (\Mat_3(R), \Id, \sharp, \det)$, where $\sharp$ denotes the classical adjoint.  We first verify that $\bfM$ is a cubic norm structure.  Computing directly from the definition \eqref{bilt}, we find that $T_\bfM(x,y) = \Tr_{\Mat_3(R)}(xy)$, where the juxtaposition on the right is usual matrix multiplication.  The formulas in \eqref{cunun} are obvious.  For \eqref{cubad}, the first two equations can be verified directly and the third equation is a standard property of the classical adjoint, completing the proof the $\bfM$ is a cubic norm structure.  Similarly, one can check directly that the $U$-operator defined from the cubic norm structure by \eqref{uop.cub} equals the $U$-operator defined from the usual matrix product in Example \ref{special.jord}, i.e., $J(\bfM) = \Mat_3(R)^+$.
\end{eg}

\begin{lem} \label{cub.comp}
Let $J$ be a cubic Jordan $R$-algebra and $x,y \in J$.
\begin{enumerate}
\item \label{inv.char} $x$ is invertible in $J$ if and only if $N_J(x)$ is invertible in $R$. In this case
\begin{equation}
x^{-1} = N_J(x)^{-1}x^\sharp \quad \text{and} \quad N_J(x^{-1}) = N_J(x)^{-1}. \notag
\end{equation}

\item \label{inv.uni} Invertible elements of $J$ are unimodular.

\item \label{nor.comp} $N_J(U_xy) = N_J(x)^2N_J(y)$ and $N_J(x^2) = N_J(x)^2 = N_J(x^\sharp)$.
\end{enumerate}
\end{lem}

\begin{proof}
(\ref{inv.char}): If $N_J(x)$ is invertible in $R$, then \eqref{inv.cub} shows that so is $x$, with inverse $x^{-1} = N_J(x)^{-1}x^\sharp$. Conversely, assume $x$ is invertible in $J$. Then $y := (x^{-1})^2$ satisfies $U_xy = 1_J$, and \eqref{sharp.norm} yields $1_J = N_J(U_xy)U_xy = N_J(x)^2N_J(y)1_J$, hence 
\[
N_J(x)^2N_J(y) = 1 
\]
since $1_J$ is unimodular by Lemma~\ref{law.uni} and \eqref{cunun}. Thus $N_J(x) \in R^\times$. Before proving the final formula of (\ref{inv.char}), we deal with (\ref{inv.uni}), (\ref{nor.comp}).

(\ref{inv.uni}) follows immediately from Lemma~\ref{law.uni} combined with the first part of (\ref{inv.char}).

\eqref{nor.comp}: Applying Lemma~\ref{law.dense} to the polynomial law $g\colon J \times J \to \End_R(J)$ defined by $g(x,y) := U_{U_xy}$ in all scalar extensions, we may assume that $U_xy$ is invertible. By \eqref{inv.uni}, therefore, $U_xy$ is unimodular, and the first equality follows from \eqref{sharp.norm}. The second equality follows from the first for $y = 1_J$, while in the third equality we may again assume that $x$ is invertible, hence unimodular. Then \eqref{inv.cub} combines with the first equality to imply 
$N_J(x)^4 = N_J(N_J(x)x) = N_J(U_xx^\sharp) = N_J(x)^2N_J(x^\sharp)$, as desired.

Now the second equality of \eqref{inv.char} follows from the first and \eqref{nor.comp} via
\[
N_J(x^{-1}) = N_J(x)^{-3}N_J(x^\sharp) = N_J(x)^{-1}. \qedhere
\]
\end{proof}

%\begin{rmk}
Without the assumption that $J$ is finitely generated projective as an $R$-module, Lemma~\ref{cub.comp} would be false \cite[Th.~10]{PR:rads}.
%\end{rmk}

\begin{eg}  \label{her.quad}
We endow the $R$-module $M := \Her_3(C,\Gamma)$ from Definition \ref{herm.lin} with a cubic norm $R$-structure $\bfM= (M,1_\bfM,\sharp,N_\bfM)$, where  $1_\bfM$ is the 3-by-3 identity matrix.
An element of $x \in \Her_3(C, \Gamma)$ may be written as
\[
x = \sbasGjord
\]
for $\alpha_i \in R$ and $c_i \in C$.  Because three of the entries are determined by symmetry, we may denote such an element by
\begin{equation} \label{basJord}
x  := \sum\nolimits_{i=1}^3 \left( \alpha_i \eps_i + \delta_i^\Gamma(c_i) \right),
\end{equation}
where $\eps_i$ has a 1 in the $(i, i)$ entry and zeros elsewhere, and $\delta_i^\Gamma(c)$ has $\gamma_{i+2}c$ in the $(i+1, i+2)$ entry --- where the symbols $i+1$ and $i+2$ are taken modulo 3 --- and zeros in the other entries not determined by symmetry.  In the literature on Jordan algebras, one finds the notation $c[(i+1)(i+2)]$ for what we denote $\delta_i(c)$.

We define the adjoint $\sharp$ by
\[
x^\sharp := \sum\nolimits_{i=1}^3\Big(\big(\alpha_{i+1}\alpha_{i+2} - \gamma_{i+1}\gamma_{i+2}n_C(c_i)\big)\varepsilon_i + \delta_i^\Gamma\big({-\alpha_ic_i} + \gamma_i\overline{c_{i+1}c_{i+2}}\big)\Big) 
\]
with indices mod 3, 
and the norm $N_\bfM$ by
\begin{equation}
N_\bfM(x) := \alpha_1\alpha_2\alpha_3 - \sum\nolimits_{i=1}^3\gamma_{i+1}\gamma_{i+2}\alpha_in_C(c_i) + \gamma_1\gamma_2\gamma_3\mathrm{Tr}_C(c_1c_2c_3)  \label{Frd.norm}
\end{equation}
in all scalar extensions, where the last summand on the right of \eqref{Frd.norm} is unambiguous since $\mathrm{Tr}_C((c_1c_2)c_3) = \mathrm{Tr}_C(c_1(c_2c_3))$ \cite[Th.~3.5]{McC:scalar}. By \cite[Th.~3]{McC:FST}, $\bfM$ is indeed a cubic norm structure. The corresponding cubic Jordan algebra will again be denoted by $J := \Her_3(C,\Gamma) := J(\bfM)$.  

(In case 2 is invertible in $R$, the commutative product $\jprod$ on $\Her_3(C, \Gamma)$ defined from the $U$-operator by \eqref{jprod.def} equals the product $x \jprod y := \frac12 (xy + yx)$ from Definition \ref{herm.lin}.  In order to see this, it suffices to note that the square of $x \in \Her_3(C, \Gamma)$ as defined in \eqref{jord.pow} is the same as the square of $x$ in the matrix algebra $\Mat_3(C)$.  This in turn follows immediately from \eqref{sharp.square}, \eqref{Frd.norm}, and the definition of the adjoint.)

For $x$ as above and $y = \sum(\beta_i\varepsilon_i + \delta_i^\Gamma(v_i))$, with $\beta_i \in R$, $v_i \in C$, evaluating the bilinear trace at $x,y$ yields
\begin{equation}
\label{bilt.mat} T_J(x,y) = \sum\nolimits_{i=1}^3 \big(\alpha_i\beta_i + \gamma_{i+1}\gamma_{i+2}n_C(u_i, v_i)\big).
\end{equation}
Since the bilinear trace $n_C$ is regular, so is $T_J$.

For the special case where $\Gamma = \Id$, we define $\Her_3(C) := \Her_3(C, \Id)$ and write $\delta_i$ for $\delta_i^\Gamma$.  It can be useful to write elements of $\Her_3(C)$ as
\[ 
\sbasjord
\]
where $\cdot$ denotes an entry that is omitted because it is determined by symmetry.  As an example of the triple product defined from \eqref{brace.def} and \eqref{uop.cub}, we mention that for $x = \sum \alpha_i \eps_i$  diagonal, we have
\begin{equation} \label{OFFTI}
\{ \delta_i(a) \delta_{i+1}(b) x \} =   \delta_{i+2}(\overline{ab}) \alpha_i \quad \text{and} 
\quad \{ \delta_{i+1}(b) \delta_i(a) x \} = \delta_{i+2}(\overline{ab}) \alpha_{i+1}
\end{equation}
for $i \in 1, 2, 3$ taken mod 3 and $a, b \in C$.
\end{eg}

Note that,  for the Jordan algebra $\Her_3(C, \Gamma)$ just defined, if we multiply $\Gamma$ by an element of $R^\times$ or any entry in $\Gamma$ by the square of an element of $R^\times$, we obtain an algebra isomorphic to the original.  Therefore, replacing $\Gamma$ by $\qform{ (\det \Gamma)^{-1} \gamma_1, (\det \Gamma) \gamma_2, (\det \Gamma) \gamma_3}$ does not change the isomorphism class of $\Her_3(C, \Gamma)$ and we may assume that $\gamma_1 \gamma_2 \gamma_3 = 1$.  

\begin{eg} \label{real.trace}
In case $R = \RR$, the preceding paragraph shows that it is sufficient to consider two choices for $\Gamma$, namely $\qform{1, 1, \pm 1}$.  We compute $T_{\Her_3(C, \Gamma)}$ for each choice of $C$ and $\Gamma$.  Regular symmetric bilinear forms over $\RR$ are classified by their dimension and signature (an integer), so it suffices to specify the signature.
If $C = \RR$, $\CC$, $\HH$, or $\OO$, the signature of $n_C$ is $2^r$ for $r = 0$, 1, 2, 3 respectively.  By \eqref{bilt.mat}, $T_J$ has signature $3(1 + 2^r)$ for $J = \Her_3(C)$ and $3 - 2^r$ for $J = \Her_3(C, \qform{1, 1, -1})$.  For $J$ the split Freudenthal algebra of rank $3(1 + 2^r)$ with $r = 1$, 2, or 3, the signature of $T_J$ is 3.
\end{eg}

\begin{rmk}
Alternatively, one could define the Jordan algebra structure on $\Her_3(C, \Gamma)$ for an arbitrary ring $R$ without referring to cubic norm structures as follows.  Writing out the formulas for the $U$-operator from Definition \ref{herm.lin} in case $R = \QQ$, one finds that the formulas do not involve any denominators other than $\gamma_i$ terms and therefore make sense for any $R$ regardless of whether 2 is invertible.  This makes $\Her_3(C, \Gamma)$ a para-quadratic algebra.  Because it is a Jordan algebra in case $R = \QQ$ as in Definition \ref{herm.lin}, we conclude that $\Her_3(C, \Gamma)$ is a Jordan algebra with no hypothesis on $R$ by extension of identities \cite[\S{IV.2.3}, Th.~2]{Bou:alg2}.  This alternative definition gives the same objects, but is much harder to work with.
\end{rmk}
%
%---------------------------------------------------------------
%
\section{Albert algebras are Freudenthal algebras are Jordan algebras}
%\begin{prop}
%Freudenthal algebras are cubic Jordan algebras.
%\end{prop}
%
%\begin{proof}
%\todo{Need to know that the cubic norm structure is unique in this case.  Refer to 2.4b or Prop.~2.7 in \cite{Loos:genalg}?}
%\end{proof}
%

\begin{defn} \label{Frd.def}
A \emph{split Freudenthal $R$-algebra} is a Jordan algebra $\Her_3(C)$ as in Example \ref{her.quad} for some split composition $R$-algebra $C$.  Because split composition algebras are determined up to isomorphism by their rank function, so are split Freudenthal algebras.

A para-quadratic $R$-algebra $J$ is a \emph{Freudenthal} algebra if $J \otimes S$ is a split Freudenthal $S$-algebra for some faithfully flat $S \in \Ralg$.  It is immediate that every Freudenthal algebra is a Jordan algebra.

Since every split Freudenthal $R$-algebra is finitely generated projective as an $R$-module for every $R$, the same is true for every Freudenthal $R$-algebra $J$ \cite[Tags 03C4, 05A9]{stacks-project}, and by the same reasoning we see that the identity element $1_J$ is unimodular.  Because the rank of a composition algebra takes values in $\{ 1, 2, 4, 8 \}$, the rank of a Freudenthal algebra takes values in $\{ 6, 9, 15, 27 \}$.  An \emph{Albert $R$-algebra} is a Freudenthal $R$-algebra of rank 27.
\end{defn}

%\begin{eg} \label{split.Frd}
%For rank 6, the split Freudenthal algebra is $\Her_3(R)$, i.e., the space of symmetric matrices.  For rank 9, the split Freudenthal algebra $\Her_3(R \times R)$ is isomorphic $\Mat_3(R)^+$ as defined in \ref{special.jord}.\todo{Provide explanation.}
%\end{eg}

\begin{prop} \label{red.is.Frd}
For every composition $R$-algebra $C$ and every $\Gamma \in \GL_3(R)$, $\Her_3(C, \Gamma)$ is a Freudenthal algebra.
\end{prop}

\begin{proof}
Replacing $R$ with $R_e$ as in \eqref{decomp}, we may assume that $C$ has constant rank.  There is a faithfully flat $S \in \Ralg$ such that $C \otimes S$ is a split composition algebra.  

Consider $T := S[t_1, t_2, t_3] / (t_1^2 - \gamma_1, t_2^2 - \gamma_2, t_3^2 - \gamma_3)$.  It is a free $S$-module, so faithfully flat.  Then $\Her_3(C, \Gamma) \otimes T$ is isomorphic to $\Her_3(C \otimes T)$ as Jordan algebras, and the latter is a split Freudenthal algebra.
\end{proof}
  
The Freudenthal algebras $\Her_3(C, \Gamma)$ are said to be \emph{reduced}.

\begin{eg} \label{Frd.squares}
Let $J$ be a Freudenthal $R$-algebra.  \emph{If $x \in J$ has $U_x = \Id_J$, then $x = \zeta 1_J$ for $\zeta \in R$ such that $\zeta^2 = 1$.}  To see this, first suppose that $J$ is $\Her_3(C)$ for some composition algebra $C$ and write $x = \sum (\alpha_i \eps_i + \delta_i(c_i))$ for $\alpha_i \in R$ and $c_i \in C$.  We find
\[
U_x \eps_i = \alpha_i^2 \eps_i + \delta_{i+2}(\alpha_i c_{i+2}) + \cdots
\]
for each $i$,
so $\alpha_i^2 = 1$ and $c_{i+2} = 0$ for all $i$.  Then
\[
U_x \delta_i(1_C) = \delta_i(\alpha_{i+1} \alpha_{i+2} 1_C).
\]
Since $1_C$ is unimodular,  $\alpha_{i+1} \alpha_{i+2} = 1$ for all $i$, proving the claim for this $J$.  

For general $J$, let $S \in \Ralg$ be faithfully flat such that $J \otimes S$ is split.  Then $x \in J$ maps to an element of $R1_J \otimes S \subseteq J \otimes S$ and so belongs to $R1_J \subseteq J$.  Since $U_{\zeta 1_J} = \zeta^2 \Id_J$ for $\zeta \in R$, the claim follows.
\end{eg}

The following result is well known when $R$ is a field or perhaps a local ring, see for example \cite[Prop.~20]{Ptr:surv}.  We impose no hypothesis on $R$.

\begin{prop} \label{split.iso}
Suppose $C$ is a split composition $R$-algebra of constant rank at least 2, i.e., $C$ is $R \times R$, $\Mat_2(R)$, or $\Zor(R)$.  Then $\Her_3(C, \Gamma) \cong \Her_3(C)$ for all $\Gamma$.
\end{prop}

\begin{proof}
Since $n_C$ is universal, there are invertible $p, q \in C$ such that $\gamma_2  = n_C(q^{-1})$ and $\gamma_3 = n_C(p^{-1})$, so $\gamma_1 = n_C(pq)$. 
We define $C^{(p,q)}$ to be a not-necessarily associative $R$-algebra with the same underlying $R$-module structure and with multiplication $\cdot_{(p,q)}$ defined by 
\[
x \cdot_{(p,q)} y := (xp)(qy),
\]
where the multiplication on the right is the multiplication in $C$.  Certainly $(pq)^{-1}$ is an identity element in $C^{(p,q)}$.  The algebra $C^{(p,q)}$ is called an isotope of $C$ and is studied in \cite{McC:homo}, where it is proved to be alternative.  One checks that it is a composition algebra with quadratic form $n_{C^{(p,q)}} = n_C(pq) n_C$, see \cite[Prop.~5]{McC:homo} for a more general statement in case $R$ is a field.

Define $\phi \!: \Her_3(\iso{C}{p,q}) \to \Her_3(C, \Gamma)$ via $\phi(\sum x_i \eps_i + \delta_i(c_i)) = \sum x_i \eps_i + \delta_i^\Gamma(c'_i)$, where
\[
c'_1 = (pq) c_1 (pq), \quad c'_2 =  c_2p, \quad \text{and} \quad c'_3 = qc_3.
\]
It is evidently an isomorphism of $R$-modules and one checks that it is an isomorphism of Jordan algebras, compare \cite[Th.~3]{McC:homo}.  Therefore, we are reduced to verifying that $\iso{C}{p,q}$ is split.

If $C$ is associative, then the $R$-linear map
\[
L_{pq} \colon \iso{C}{p,q} \to C \quad \text{such that} \quad L_{pq}(x) = pqx
\]
is an isomorphism of $R$-algebras.  So assume $C = \Zor(R)$.

At the beginning, when we chose $p$ and $q$, we were free to pick $\xi_i, \eta_i \in R^\times$ such that $p = \stbtmat{\xi_1}{0}{0}{\xi_2}$ and $q = \stbtmat{\eta_1}{0}{0}{\eta_2}$.  Let $A \in \Mat_3(R)$ be any matrix such that $\det A = (\xi_1 \xi_2^2 \eta_1^2 \eta)^{-1}$ and put $B := \xi_2 \eta_1 (A^\sharp)^{\trans}$, where $\sharp$ denotes the classical adjoint.  
With $\zeta_i := (\xi_i \eta_i)^{-1}$, one checks, using the formula $(Sx) \times (Sy) = (S^\sharp)^\trans(x \times y)$ for $\times$ the usual cross product in $R^3$, that the assignment
\[
\stbtmat{\alpha_1}{u_1}{u_2}{\alpha_2} \mapsto \stbtmat{\zeta_1 \alpha_1}{Au_1}{Bu_2}{\zeta_2 \alpha_2}
\]
defines an isomorphism $C \xrightarrow{\sim} \iso{C}{p,q}$.
\end{proof}

%
%----------------------------------------------------------------------------------------------------------------
%
\section{The ideal structure of Freudenthal algebras} \label{simple.sec}

%
% IDFREU p638   IDCUJM
%

It is a standard exercise to show that every (two-sided) ideal in the matrix algebra $\Mat_n(R)$ is of the form $\Mat_n(\idl)$ for some ideal $\idl$ in $R$.  More generally, every ideal in an Azumaya $R$-algebra $A$ is of the form $\idl A$ some ideal $\idl$ of $R$ \cite[p.~95, Cor.~III.5.2]{KO}.

A similar result holds for every octonion $R$-algebra $C$: \emph{Every one-sided ideal in $C$ is a two-sided ideal that is stable under the involution on $C$.  The maps $I \mapsto I \cap R$ and $\idl C \mapsfrom \idl$ are bijections between the set of ideals of $C$ and ideals in $R$.}  See \cite[\S4]{PeMS} for a proof in this generality and the references therein for earlier results of this type going back to \cite{Mahler}.

We now prove a similar result for Freudenthal algebras.

\begin{defn} \label{pq.ideal}
An \emph{ideal} in a para-quadratic $R$-algebra $J$ is the kernel of a homomorphism, i.e., an $R$-submodule $I$ such that 
\[
U_I J + U_J I + \{ J J I \} = I,
\]
where we have written $U_I J$ for the $R$-span of $U_x y$ with $x \in I$ and $y \in J$.  (This is sometimes written with a $\subseteq$ instead of $=$, but the two are equivalent since $U_J I \supseteq U_{1_J} I = I$.)  An $R$-submodule $I$ is an \emph{outer ideal} if
\begin{equation} \label{outer.ideal.def}
U_J I + \{ J J I \} = I.
\end{equation} 
\end{defn}

Here are some observations about outer ideals:
\begin{enumerate}
\item Every ideal is an outer ideal.  

\item If 2 is invertible in $R$, then for every $x \in I$ and $y \in J$, $U_x y = \frac12 \{ xyx \} \in \{ J J I \}$, so the notions of ideal and outer ideal coincide, and both agree with the notion of ideal for the commutative bilinear product $\jprod$ defined in \eqref{jprod.def}.
\item \label{outer.3} For every ideal $\idl$ in $R$, the $R$-submodule $\idl J$ is an ideal of $J$.

\item If $1_J$ is unimodular, then for every outer ideal $I$ of $J$, $I \cap R1_J$ is an ideal in $R$, for the trivial reason that $I$ is an $R$-module.

\item \label{outer.5}  If $\idl$ is an ideal in $R$ and $1_J$ is unimodular, then $\idl 1_J = (\idl J) \cap R 1_J$.  The containment $\subseteq$ is clear.  To see the opposite containment, suppose $\alpha 1_J \in \idl J \cap R1_J$ for some $\alpha \in R$ and write $\alpha1_J = \sum \alpha_i y_i$ with $\alpha_i \in \idl$ and $y_i \in J$.  There is some $R$-linear $\lambda \!: J \to R$ such that $\lambda(1_J) = 1$.  Then $\alpha = \lambda(\alpha 1_J) = \sum \alpha_i \lambda(y_i)$ is in $\idl$.
%\item \label{outer.sq} Write $\Sq(J)$ for the $R$-submodule of $J$ spanned by the $x^2$ for $x \in J$.  For every outer ideal $I$ in $J$, $(I \cap R) \Sq(J) \subseteq I$.  To see this, let $\alpha \in I \cap R1_J$ and $x = \sum U_{x_i} r_i \in \Sq(J)$ for $x_i \in J$ and $r_i \in R$.  Then $\alpha x = \sum U_{x_i} (\alpha r_i)$, where $\alpha r_i \in \idl$.  Since $I$ is an outer ideal, $\alpha x$ is in $I$.
\end{enumerate}

\begin{thm} \label{IDCUJM}
Let $J$ be a Freudenthal $R$-algebra.  Every outer ideal of $J$ is an ideal.  The maps $I \mapsto I \cap R1_J$ and $\idl J \mapsfrom \idl$ are bijections between the set of outer ideals of $J$ and the set of ideals of $R$.
\end{thm}

\begin{proof}
It suffices to show that the stated maps are bijections, because then observation \eqref{outer.3} implies that every outer ideal is of the form $\idl J$ and therefore an ideal.  In view of \eqref{outer.5} (noting that $1_J$ is unimodular), it suffices to verify that $(I \cap R1_J)J = I$ for every outer ideal $I$.  

\smallskip
First suppose that $J = \Her_3(C)$ for some composition $R$-algebra $C$ and write $\idl := I \cap R1_J$.  The Peirce projections relative to the diagonal frame of $J$, i.e., $U_{\eps_i}$ and $x \mapsto \{\eps_j x \eps_l \}$ for $i, j, l = 1, 2, 3$ \cite[p.~1074]{McC:NAS} stabilize $I$, and we find
\[
I = \sum_i (I \cap R\eps_i) + (I \cap \delta_i(C)).
\]

Set $B := \{ c \in C \mid \delta_1(c) \in I \}$.  We claim that $B$ is an ideal in $C$.  Note that $U_{\delta_1(1_C)} \delta_1(b) = \delta_1(\bar{b})$, so $B$ is stable under the involution.

We leverage \eqref{OFFTI}.
Repeatedly applying this with $a = 1_C$ and using that $B$ is stable under the involution, we conclude that $\delta_i(B) = I \cap \delta_i(C)$ for all $i$.  For $c \in C$ and $b \in B$, $I$ contains $\{ 1_J \delta_2(\bar{c}) \delta_1(\bar{b}) \} = \delta_3(c b)$, so $cB \subseteq B$, i.e., $B$ is an ideal in $C$ and therefore $B = \idl C$ for some ideal $\idl$ of $R$.

For $c \in C$, $I$ contains $\{ \delta_i(1_C) \eps_{i+1} \delta_i(\idl c) \} = \Tr_C(\idl c) \eps_{i+2}$.  Since $\Tr_C$ is surjective, $\idl \eps_j \subseteq I$ for all $j$.

In the other direction, if $\alpha_i \eps_i \in I$, then so is 
\[
\{ \delta_{i+1}(1_C) 1_J (\alpha_i \eps_i) \} = \delta_{i+1}(\alpha_i 1_C).
\]
It follows that $I \cap R\eps_i = \idl R$ for all $i$ and in particular, $I \cap R1_J = \idl R$ and $I = \idl J$.

\smallskip
We now treat the general case.
Suppose $I$ is an outer ideal in a Freudenthal $R$-algebra $J$.  There is a faithfully flat $S \in \alg{R}$ such that $J \otimes S$ is a split Freudenthal algebra.  We have
\[
((I \cap R1_J) J) \otimes S = (I \otimes S \cap S1_J) (J \otimes S) = I \otimes S
\]
where the first equality is because $S$ is faithfully flat and the second is by the previous case, since $I \otimes S$ is an outer ideal.  It follows that $I = (I \cap R1_J)J$ as desired.
\end{proof}

\begin{rmk*}
In the proof above, the inclusion $(I \cap R1_J) J \subseteq I$ could instead have been argued as follows.  Define $\Sq(J)$ as the $R$-submodule of $J$ generated by $x^2$ for $x \in J$.  Since $1_J$ is unimodular, one finds that $(I \cap R1_J) \Sq(J) \subseteq I$.  Then, one argues that $\Sq(J) = J$ for a split Freudenthal algebra, and that $\Sq(J \otimes S) = \Sq(J) \otimes S$ for all flat $S \in \Ralg$.
\end{rmk*}

\begin{cor} \label{Frd.noaug}
Let $\phi \colon J \to A$ be a homomorphism of para-quadratic $R$ algebras.  If $J$ is a Freudenthal algebra and $1_A$ is unimodular in $A$, then $\phi$ is injective.
\end{cor}

In particular, the corollary applies to every homomorphism between Freudenthal $R$-algebras.

\begin{proof}
The kernel of $\phi$ is an ideal of $J$ and therefore $\idl J$ for some ideal $\idl$ of $R$.  For $\alpha \in \idl$, we have
$0 = \phi(\alpha 1_J) = \alpha \phi(1_J) = \alpha 1_A$, so $\alpha = 0$ because $1_{A}$ is unimodular.
$0 = \phi(\alpha 1_J) = \alpha \phi(1_J) = \alpha 1_{A}$, so $\alpha = 0$ because $1_{A}$ is unimodular.
\end{proof}

Because a Freudenthal algebra $J$ is a projective $R$-module of rank $\ge 3$, the corollary says that there is no homomorphism of para-quadratic $R$-algebras $J \to R^+$.  This might be stated as $J$ \emph{is not augmented} or $J$ \emph{has no counit}.

A para-quadratic algebra $J$ is said to be \emph{simple} if the underlying module is not the zero module and if every ideal in $J$ equals 0 or $J$.  Theorem \ref{IDCUJM} immediately gives:

\begin{cor} \label{Frd.simple}
If $J$ is a Freudenthal $R$-algebra and $R$ is a field, then $J$ is simple. $\hfill\qed$
\end{cor}

\begin{rmk*}
There is also the notion of an \emph{inner} ideal in a Jordan algebra, see \cite[Th.~8]{McC:inn} for a description of them for $\Her_3(\Zor(R))$.  The inner ideals are related to the projective homogeneous varieties associated with the group of isometries described in \S\ref{E6.sec} and ``outer automorphisms'' relating these varieties, see \cite{Racine:point} and \cite{CG}.
\end{rmk*}

%----------------------------------------------------------------------------------------------------------------
%
\section{Groups of type \texorpdfstring{$\eff_4$}{F4} and \texorpdfstring{$\cee_3$}{C3}}

In the following, for a Jordan $R$-algebra $J$, we write $\Aut(J)$ for the ordinary group of $R$-linear automorphisms of $J$ and $\bfAut(J)$ for the functor from $\alg{R}$ to groups such that $S \mapsto \Aut(J \otimes S)$.  Recall that for every simple root datum, there is a unique simple group scheme over $\ZZ$ called a \emph{Chevalley group} \cite[Cor.~XXIII.5.4]{SGA3.3}, and every split simple algebraic group over a field is obtained from a unique Chevalley group by base change \cite[\S23g]{Milne:AG}.

\begin{lem} \label{skip.aut.alb}
Let $J$ be a Freudenthal algebra of rank 15 or 27 over a ring $k$.  Then
$\bfAut(J)$ is a semisimple $k$-group scheme that is adjoint (i.e., its center is the trivial group scheme).  Its root system has type $\cee_3$ if $J$ has rank 15 and type $\eff_4$ if $J$ has rank 27.  If $J$ is the split Freudenthal algebra, then the group scheme $\bfAut(J)$ is obtained from the Chevalley group over $\ZZ$ by base change.
\end{lem}

\begin{proof}
First suppose that $R = \ZZ$ and $J$ is split.   If $J$ has rank 15, then the proof of 14.19 in \cite{Sp:jord} shows that the automorphisms of $J \otimes F$ for every field $F$ are exactly the automorphisms of the algebra $\Mat_6(F)$ with the split symplectic involution, which is the split adjoint group $\PGSp_6$.  For $J$ of rank 27,  $\bfAut(J) \times F$ is split of type $\eff_4$ by 
 \cite[\S6]{Jac:ex} (written for Lie algebras), \cite[Satz 4.11]{Frd:OAO} (written for $\IR$),  \cite[Th.~7.2.1]{Sp:ex} (if $\car F \ne 2, 3$), or  \cite[14.24]{Sp:jord} in general.

Note that $\bfAut(J) \times F$ is connected and smooth as a group scheme over $F$, and $\bfAut(J)$ is finitely presented (because $\ZZ$ is noetherian and $J$ is a finitely generated module), so it follows by \cite[Prop.~6.1]{GanYu:G2} or \cite[Lemma B.1]{AlsaodyGille} that $\bfAut(J)$ is smooth as a scheme over the Dedekind domain $\ZZ$.  In summary, $\bfAut(J)$ is semisimple and adjoint of the specified type.  Moreover, because $\bfAut(J) \times \QQ$ is split, $\bfAut(J)$ is a Chevalley group \cite[Th.~1.4]{Conrad:Z}.

In the case of general $R$ and $J$, let $S \in \alg{R}$ be faithfully flat such that $J \otimes S$ is split.  Then $\bfAut(J) \times S$ is semisimple adjoint of the specified type.  Certainly, $\bfAut(J)$ is also smooth.  Moreover, 
for each $\mfp \in \Spec R$, there is a $\mfq \in \Spec S$ such that $\mfq \cap R = \mfp$.  Then the field of fractions $R(\mfp)$ of $R/\mfp$ embeds in the field $S(\mfq)$, so the algebraic closure $\overline{R(\mfp)}$ includes in the algebraic closure $\overline{S(\mfq)}$.  Because $\bfAut(J) \times \overline{S(\mfq)}$ is adjoint semisimple of the specified type and this property is unchanged by replacing one algebraically closed field by a smaller one, the same holds over $\overline{R(\mfp)}$.  Since this holds for every $\mfp$, the claim is verified.
\end{proof}

\begin{rmk}
In case $R$ is a field, the automorphism group of the split Freudenthal algebra of rank 6 or 9 can be deduced in a similar manner, referring to 14.17 and 14.16 in \cite{Sp:jord}.
The automorphism group of the split Freudenthal algebra of rank 9 is $\PGL_3 \rtimes \ZZ/2$.
The automorphism group of the split Freudenthal algebra of rank 6 is the special orthogonal group of the quadratic form $x^2 + y^2 + z^2$, i.e., the group commonly denoted $\mathrm{SO}(3)$.  In particular, it is not smooth when $R$ is a field of characteristic 2 and indeed one can give examples of Freudenthal algebras of rank 6 over a field of characteristic 2 that are not split by any \'etale cover.  
\end{rmk}

For $J$, $J_0$ Jordan $R$-algebras, we define $\Iso(J, J_0)$ to be the set of $R$-linear isomorphisms $J \to J_0$ and $\bfIso(J, J_0)$ to be the corresponding functor from $\alg{R}$ to sets defined by $S \mapsto \Iso(J \otimes S, J_0 \otimes S)$.  If $J$ and $J_0$ become isomorphic over a faithfully flat $S \in \alg{R}$, then $\bfIso(J, J_0)$ is naturally an $\bfAut(J_0)$-torsor in the fpqc topology. 

The statement of the following result is similar to statements over a field that can be found in \cite{SeCG}.  Its proof amounts to combining the lemma with the general machinery of descent.

\begin{thm} \label{skip.F4.alb}
Let $J_0$ be a Freudenthal $R$-algebra of rank $r = 15$ or 27.  In the diagram
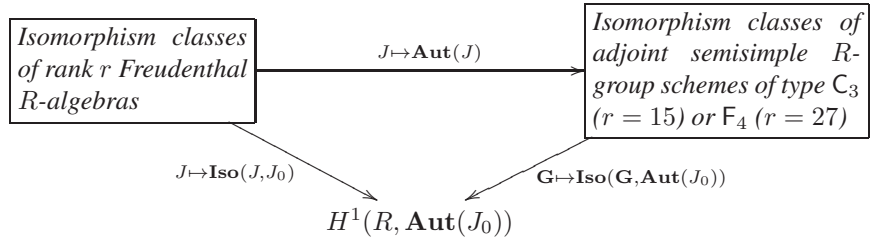
\[
\xymatrix{*+[F]{\parbox{1.2in}{Isomorphism classes of rank $r$ Freudenthal $R$-algebras}} \ar[rr]^{J \mapsto \bfAut(J)}
\ar[dr]_*!/_4pt/{\labelstyle {J \mapsto \bfIso(J, J_0)}} && *+[F]{\parbox{1.4in}{Isomorphism classes of adjoint semisimple $R$-group schemes of type $\cee_3$ ($r = 15$) or $\eff_4$ ($r = 27$)}} \ar[dl]^*!/^6pt/{\labelstyle {\bfG \mapsto \bfIso(\bfG, \bfAut(J_0))}} \\
&H^1(R, \bfAut(J_0))}
\]
all arrows are bijections that are functorial in $R$.
\end{thm}

\begin{proof}
The facts that the arrows are well defined, the diagram commutes, and the diagonal arrows are injective are general feature of the machinery of descent.  The lower left arrow is surjective because every Freudenthal algebra is split by some faithfully flat $R$-algebra by definition.  The lower right arrow is surjective because every semisimple group scheme is split by some faithfully flat $R$-algebra (even an \'etale cover) \cite[Cor.~XXIV.4.1.6]{SGA3.3}.  
\end{proof}

The machinery of descent shows a more refined statement, where each of the boxes in the theorem are replaced by groupoids and the arrows are equivalences of groupoids.  In that statement, the bottom box is replaced by the groupoid of $\bfAut(J_0)$-torsors and the diagonal arrows are provided by \cite[p.~151, Th.~III.2.5.1]{Giraud}.

In the theorem, the set $H^1(R, \bfAut(J_0))$ is naturally a pointed set and the 
bijections are actually of pointed sets, where the distinguished elements are $J_0$ in the upper left and $\bfAut(J_0)$ in the upper right.

In case $R$ is a field of characteristic different from 2, 3 and $r = 27$, the theorem goes back to \cite{Hij}.  Or see \cite[26.18]{KMRT}.

\begin{cor}
For each Freudenthal $R$-algebra $J$ of rank 15 or 27, there is an \'etale cover $S \in \alg{R}$ such that $J \otimes S$ is a split Freudenthal algebra.
\end{cor}

\begin{proof}
Let $J_0$ be the split Freudenthal $R$-algebra of the same rank as $J$.  The image $\bfIso(J, J_0)$ of $J$ in $H^1(R, \bfAut(J_0))$ is a $\bfAut(J_0)$-torsor.  Since $\bfAut(J_0)$ is smooth (Lemma \ref{skip.aut.alb}), there is an \'etale cover of $R$ that trivializes $\bfIso(J, J_0)$.
\end{proof}

Note that exactly the same kind of argument gives analogues of Lemma \ref{skip.aut.alb} and Theorem \ref{skip.F4.alb} for composition algebras, where $r = 4$ or 8, and the group is of type $\ay_1$ or $\gee_2$ respectively.

%----------------------------------------------------------------------------------------------------------------
%
\section{Generic minimal polynomial of a Freudenthal algebra} \label{genmin.sec}

\subsection*{Polynomials with polynomial-law coefficients}
Let $J$ be a Jordan $R$-algebra, $\Pol(J, R)$ the $R$-algebra of polynomial laws from $J$ to $R$, and $t$ a variable. Consider a polynomial $\bfp(t) = \sum_{i=0}^n f_it^i$ with $f_i \in \Pol(J, R)$ for $0 \leq i \leq n$. For $S \in \alg{R}$, $x \in J \otimes S$, we have $\bfp(t,x) := \sum_{i=0}^nf_{iS}(x)t^i \in S[t]$, and we define
\[
\bfp(x,x) := \sum_{i=0}^nf_{iS}(x)x^i \in J \otimes S.
\]
The algebra $J$ is said to \emph{satisfy $\bfp$} if $\bfp(x,x) = 0 = (t\bfp)(x,x)$ for all $x \in J \otimes S$, $S \in \alg{R}$. Note that the second equation follows from the first if $2$ is invertible in $R$ but not in general, see Remark~\ref{loc.lin}.

\subsection*{The generic minimal polynomial} 
Let $J := \Her_3(C, \Gamma)$ as in Example \ref{her.quad}.
With a variable $t$ we recall from \eqref{cub.eq} that $J$ satisfies the monic polynomial
\begin{equation}
\label{gen.min} \bfm_J = t^3 - \Tr_J\cdot t^2 + S_J\cdot t - N_J \in \Pol(J, R)[t].
\end{equation}
More precisely, by \cite[2.4(b)]{Loos:genalg}, $J$ is generically algebraic of degree $3$ in the sense of \cite[2.2]{Loos:genalg} and $\bfm_J$ is the \emph{generic minimal polynomial} of $J$, i.e., the unique monic polynomial in $\Pol(J,R)[t]$ of minimal degree satisfied by $J$ \cite[2.7]{Loos:genalg}. It follows that the Jordan algebra $J$ determines the polynomial $\bfm_J$ uniquely. In particular, the \emph{generic norm} $N_J$, the \emph{generic trace} $T_J$ (or $\Tr_J$) and, in fact, the cubic norm structure underlying $J$ in the sense of Definition~\ref{cubjord.def} are uniquely determined by $J$ as a Jordan algebra.

By faithfully flat descent, every Freudenthal algebra $J$ has a uniquely determined generic minimal polynomial of the form \eqref{gen.min}, and a uniquely determined underlying cubic norm structure.  We conclude:

\begin{prop} \label{Frd.cubic}
Every Freudenthal algebra $J$ is a cubic Jordan algebra and the bilinear trace form $T_J$ is regular. $\hfill\qed$
\end{prop}

The preceding discussion shows that, for a Freudenthal $R$-algebra $J$, the Jordan algebra structure of $J$ alone (ignoring that $J$ is a \emph{cubic} Jordan algebra) determines the bilinear form $T_J$.  (For example, the 11 Freudenthal $\RR$-algebras discussed in Example \ref{real.trace} have distinct trace forms and therefore are distinct.)  When $R$ is a field of characteristic $\ne 2, 3$ and $J$ and $J'$ are reduced Freudenthal algebras, Springer proved that the converse also holds, i.e., $J \cong J'$ if and only if $T_J \cong T_{J'}$ \cite[Th.~5.8.1]{Sp:ex}.  We do not use Springer's result in this paper. 

The following result can also be found in \cite[Cor.~18(b)]{Ptr:surv}, based on the different definition of Albert algebra appearing there.

\begin{lem} \label{HOCUJO}
Let $J$ and $J'$ be Freudenthal $R$-algebras.  An $R$-linear map $\phi \!: J \to J'$ is an isomorphism of para-quadratic algebras if and only if $\phi$ is surjective, $\phi(1_J) = 1_{J'}$, and $N_{J'} = N_J \phi$ as polynomial laws.
\end{lem}

\begin{proof}  The ``only if'' direction follows from the uniqueness of the generic minimal polynomial as in \eqref{gen.min}, so we show ``if''.  The equality $N_{J'} = N_J \phi$ of polynomial laws and the definition of the directional derivative in \S\ref{directional.sec} gives formulas such as 
\[
\grd_y N_J(x) = \grd_{\phi(y)} N_{J'}(\phi(x)).
\]
Since $\phi(1_J) = 1_{J'}$, the definition of the bilinear forms $T_J$ and $T_{J'}$ in \eqref{bilt} give:
\[
T_{J'}(\phi (x), \phi (y)) = T_J(x, y)
\]
for all $x, y$.  Therefore, on the one hand we have
\[
\grd_y N_J(x) = T_J(x^\sharp, y) = T_{J'}(\phi(x^\sharp), \phi(y)).
\]
On the other hand, we have
\[
\grd_y N_J(x) = \grd_{\phi(y)} N_{J'}(\phi(x)) = T_{J'}((\phi(x))^\sharp, \phi(y)).
\]
Therefore, $\phi(x^\sharp) = \phi(x)^\sharp$ for all $x$.
In summary, $\phi$ commutes with $\sharp$ and preserves $T_J$.  Therefore, by \eqref{uop.cub}, $\phi$ is a homomorphism of Jordan algebras.

Suppose that $x$ is in $\ker \phi$.  Then for all $y \in J$, $T_J(x, y) = T_{J'}(\phi(x), \phi(y)) = 0$, so $x = 0$ since the bilinear form $T_J$ is regular.  Since $\phi$ is both surjective and injective, it is an isomorphism.
\end{proof}

\begin{eg} \label{Mat3.9}
We claim that $\Her_3(R \times R)$, the split Freudenthal algebra of rank 9,  is isomorphic to $\Mat_3(R)^+$.  To see this, define $\pi_i \!: R \times R \to R$ to be the projection on the $i$-th coordinate and define $\phi \colon \Her_3(R \times R) \to \Mat_3(R)^+$ by sending 
\[
\sbasjord \mapsto \left( \begin{smallmatrix}
\alpha_1 & \pi_1(c_3) & \pi_2(c_2)\\
\pi_2(c_3) & \alpha_2 & \pi_1(c_1) \\
\pi_1(c_2) & \pi_2(c_2) & \alpha_3 
\end{smallmatrix} \right).
\]
This map is obviously $R$-linear and surjective and sends the identity to the identity.  One checks directly that $\phi$ preserves norms, i.e., that $\det(\phi(x))$ equals $N(x)$ according to \eqref{Frd.norm}.  Because $\Her_3(R \times R)$ and $\Mat_3(R)^+$ are both cubic Jordan algebras with regular trace bilinear forms, the proof of the ``if'' direction of Lemma \ref{HOCUJO} shows that $\varphi$ is an isomorphism of Jordan algebras.
\end{eg}

%%%%%%%%%%%%%%%%%%%%%%%%%%%%%%%%%%%%%%%%%%%%%%%%%%%%%%%%%%
\section{Basic classification results for Albert algebras}

In the case where $R$ is a field such as the real numbers, a finite field, a local field, or a global field, one can find in many places in the literature classifications of Albert algebras proved using techniques involving algebras as in \cite[\S5.8]{Sp:ex}.  For such an $R$, groups of type $\eff_4$ can be classified using techniques from algebraic groups, such as in \cite[Ch.~6]{PlatRap} or \cite{Gille:sc2}.  The two approaches are equivalent by Theorem \ref{skip.F4.alb}.

\begin{eg}[Albert algebras over $\RR$] \label{alb.RR}
Up to isomorphism, there are three Albert $\RR$-algebras, namely the split one $\Her_3(\Zor(\RR))$, $\Her_3(\OO, \qform{1, 1, -1})$, and $\Her_3(\OO)$.  Rather than proving this in the language of Jordan algebras as in \cite[Th.~10]{AlbJac}, one may leverage Theorem \ref{skip.F4.alb} as follows.  The three algebras are pairwise non-isomorphic because their trace forms are (Example \ref{real.trace}).  At the same time, a computation in the Weyl group of $\eff_4$ 
as in \cite[\S{III.4.5}]{SeCG}, \cite[14.1]{BorovoiEvenor}, or \cite[Table 3]{AdamsTaibi} shows that $H^1(\RR, \bfAut(\Her_3(\Zor(\RR))))$ has three elements.  That is, there are exactly three isomorphism classes of simple affine group schemes over $\RR$ of type $\eff_4$, so we have found all of them.
\end{eg}

\begin{eg}[Albert algebras over global fields] \label{alb.global}
Let $A$ be an Albert $F$-algebra for $F$ a global field.  Put $\Omega$ for the (finite) set of inequivalent embeddings $\omega \colon F \hookrightarrow \RR$.  Since $\bfAut(A)$ is a simple and simply connected affine group scheme, the natural map 
\[
H^1(F, \bfAut(A)) \xrightarrow{\prod \omega} \prod_{\omega \in \Omega} H^1(\RR, \bfAut(A))
\]
is a bijection by \cite{Hrdr:2} or \cite[p.~286, Th.~6.6]{PlatRap}.  Because $H^1(\RR, \bfAut(A))$ has three elements by the preceding example, there are exactly $3^{|\Omega|}$ isomorphism classes of Albert $F$-algebras.  This result, for number fields and with a proof in the language of Albert algebras, dates back to \cite[p.~417, Cor.~of Th.~12]{AlbJac}.

The same argument goes through for octonion algebras, and one finds that there are $2^{|\Omega|}$ isomorphism classes of octonion $F$-algebras.  This result, for number fields and with a proof in the language of octonion algebras, dates back to \cite[p.~400]{Zorn}.

In the special case where $F$ has a unique real embedding (e.g., $F = \QQ$), the two isomorphism classes of octonion algebras are $\Zor(F)$ and $\oct \otimes F$, and the three isomorphism classes of Albert algebras are  $\Her_3(\Zor(F))$ and $\Her_3(\oct \otimes F, \qform{1, 1, \pm 1})$.
\end{eg}

Below, we will focus our attention on classification results in the case where $R$ is not a field.  We translate known results about cohomology of affine group schemes into the language of Albert algebras.

\begin{prop} \label{G2F4.dvr}
If $R$ is (1) a complete discrete valuation ring whose residue field is finite or (2) a finite ring, then every Freudenthal $R$-algebra of rank 15 or 27 and every quaternion or octonion $R$-algebra is split.
\end{prop}

\begin{proof}
In view of Theorem \ref{skip.F4.alb} and its analogue for composition algebras, it suffices to prove that $H^1(R, \bfG) = 0$ for $\bfG$ a simple $R$-group scheme of type $\eff_4$ or $\cee_3$ obtained by base change from a Chevalley group over $\ZZ$.  In case (1), this is \cite[Prop.~3.10]{Conrad:Z}.  In case (2), we apply the following lemma.
\end{proof}

\begin{lem} \label{skip.pr.finite}
If $R$ is a finite ring and $\bfG$ is a smooth connected $R$-group scheme, then $H^1(R, \bfG) = 0$.
\end{lem}

\begin{proof}
If $R$ is not connected, then it is a finite product $R = \prod R_i$ where each ring $R_i$ is finite, so $H^1(R, \bfG) = \prod H^1(R_i, \bfG \times R_i)$.   Therefore it suffices to assume that $R$ is connected.

Suppose $\bfX$ is a $\bfG$-torsor.  Our aim is to show that $\bfX$ is the trivial torsor, i.e., $\bfX(R)$ is nonempty.  Put $\idl$ for the nil radical $\Nil(R)$ of $R$.  Because $R$ is finite, there is some minimal $m \ge 1$ such that $\idl^m = 0$.  We proceed by induction on $m$.  If $m = 1$, then $R$ is reduced and connected, so it is a finite field and $H^1(R, \bfG) = 0$ by Lang's Theorem.  For the case $m \ge 2$, put $I := \idl^{m-1}$.  The ring $R/I$ has $\Nil(R/I)^{m-1} = (\Nil(R)/I)^{m-1} = 0$, so by induction $\bfX(R/I)$ is nonempty.  On the other hand, $I^2 = \idl^{2m-2} = \idl^m \cdot \idl^{m-2} = 0$ and $\bfX$ is smooth, so the natural map $\bfX(R) \to \bfX(R/I)$ is surjective.
\end{proof}

\begin{eg} \label{dedekind.eg1}
Suppose $R$ is a Dedekind ring and write $F$ for its field of fractions.  For $\bfG$ a Chevalley group of type $\gee_2$, $\eff_4$, or $\ee_8$, the map $H^1(R, \bfG) \to H^1(F, \bfG)$ has zero kernel \cite[Satz 3.3]{Hrdr:Ded}.  Consequently,  \emph{if $A$ is an Albert or octonion $R$-algebra and $A \otimes F$ is split, then the $R$-algebra $A$ is split.}

In particular, if $F$ is a global field with no real embeddings, then every Albert or octonion $F$-algebra is split, so every Albert or octonion $R$-algebra is split.
\end{eg}

In the case where $F$ is a number field with a real embedding, we provide the following partial result, which relies on Example \ref{alb.RR}.  %An Albert (resp., octonion) $\RR$-algebra $A$ has no nonzero nilpotent elements if and only if $\bfAut(A)$ is the compact real form, if and only if $\bfAut(A)$ is an anisotropic group, the last equivalence being \cite[\S24.6]{Borel}.

\begin{prop} \label{skip.numb.indef}
Suppose $F$ is a number field and $R$ is a localization of its ring of integers at finitely many primes.  If $A$ is an Albert (resp., octonion) $F$-algebra such that $A \otimes \RR$ is not isomorphic to $\Her_3(\OO)$ (resp., $\OO$) for every embedding $F \hookrightarrow \RR$, then there is an Albert (resp., octonion) $R$-algebra $B$ such that $B \otimes F \cong A$ and $B$ is uniquely determined up to $R$-isomorphism.
\end{prop}

\begin{proof}
Write $\bfG$ for the automorphism group of the split Albert (resp., octonion) $F$-algebra.  Write $\Hind(R, \bfG) \subseteq H^1(R, \bfG)$ for the isomorphism classes of $R$-algebras $B$ such that $B \otimes F_v$ is not $\Her_3(\OO)$ (resp., $\OO$), i.e., such that $\bfAut(B) \times F_v$ is not compact, for all real places $v$ of $F$.  Since $\bfG$ is simply connected, Strong Approximation gives that the natural map $\Hind(R, \bfG) \to \Hind(F, \bfG)$ is an isomorphism \cite[Satz 4.2.4]{Hrdr:Ded}, which is what is claimed.
\end{proof}

%
%----------------------------------------------------------------------------------------------------------------
%

\section{The number of generators of an Albert algebra}

The goal of this section is to prove Proposition \ref{generators}, which is inspired by the work of First-Reichstein \cite{FirstRei} generalizing the Forster-Swan Theorem.  Let $J$ be a para-quadratic $R$-algebra. By a \emph{(para-quadratic) subalgebra} of $J$ we mean a submodule $J' \subseteq J$ containing $1_J$ and closed under the $U$-operator, i.e., such that $U_x y \in J'$ for all $x,y \in J'$.
For any subset $S \subseteq J$, the smallest subalgebra of $J$ containing $S$ is called the subalgebra 
\emph{generated by} $S$; if this subalgebra is all of $J$, we say that $S$ \emph{generates
$J$}.

%
%We say that an $R$-submodule $B$ of a para-quadratic $R$-algebra $A$ is \emph{closed under $U$} if $B$ contains $U_x y$ for all $x, y \in B$.  We say that a subset $S$ of $A$ \emph{generates $A$} if $A$ is the smallest $R$-submodule of $A$ containing $S$ and closed under $U$.  
%
%Since a para-quadratic algebra $A$ contains an identity element $1_A$, one might ask whether $S \cup \{ 1_A \}$ generates $A$ (``$S$ generates $A$ as a unital algebra'') rather than merely $S$.  However, we are interested here in the case where $A$ is an Albert algebra, and Albert algebras are not augmented (\S\ref{simple.sec}).  Therefore, at least in the case where $R$ contains a field of characteristic $\ne 2$, $S$ generates $A$ if and only if $S \cup \{ 1_A \}$ generates $A$ by \cite[Lemma 2.2]{FirstReiW}.

\begin{prop} \label{generators}
For every noetherian ring $R$, every Albert $R$-algebra can be generated in the sense of the preceding paragraph by $3 + \dim \Max R$ elements.
\end{prop}

In the statement $\Max R$ is the topological space whose points are the maximal ideals of $R$, endowed with the subspace topology inherited from $\Spec R$.  It is evident that $\dim \Max R \le \dim \Spec R$, also known as the Krull dimension.  Beyond this inequality, the two numbers may be quite different (e.g., for a local ring, $\dim \Max R = 0$ and $\dim \Spec R$ can be any number).  If 2 is invertible in $R$, then this bound is Corollary 4.2c in \cite{FirstRei}; the contribution here is to remove the hypothesis on 2.  In the special case where $R$ contains an infinite field of characteristic $\ne 2$, a different (and possibly smaller) upper bound is given in \cite[\S13]{FirstReiW}.

The results in \cite{FirstRei} reduce the proof of the proposition to the case where $R$ is a field.  For a field of characteristic not $2$, a proof may be read off from \cite[p.~112]{McC}; we give a characteristic-free proof using the first Tits construction of cubic Jordan algebras \cite[pp.~507--509]{McC:FST}.   We briefly recall its details.

\subsection*{The first Tits construction} Let $A$ be a (finite-dimensional) separable associative algebra of degree $3$ over a field $F$, so its generic norm $N_A$ is a cubic form on $A$ and its  trace $T_A\colon A \times A \to F$ defined as in \eqref{bilt} is a non-degenerate symmetric bilinear form. Given any non-zero scalar $\mu \in F$, we obtain a cubic norm structure
\[
\bfM:= (A \times A \times A,1_\bfM,\sharp,N_\bfM)
\]
by defining
\begin{gather}
\label{fitone} 1_\bfM:= (1_A,0,0), \\
\label{fitsha} x^\sharp := (x_0^\sharp - x_1x_2,,\mu^{-1}x_2^\sharp - x_0x_1,\mu x_1^\sharp - x_2x_0),  \quad \text{and} \\
\label{fitnor} N_{\bfM}(x) := N_A(x_0) + \mu N_A(x_1) + \mu^{-1}N_A(x_2) - T_A(x_0x_1x_2)
\end{gather}
for all $x = (x_0,x_1,x_2)$ in all scalar extensions of $A \times A \times A$. By \cite[Th.~6]{McC:FST}, $\bfM$ is a cubic norm structure, so $J(\bfM)$ is a cubic Jordan $F$-algebra which we denote by $J(A, \mu)$.
% and $J := J(A,\mu) := J(\bfM)$ is a Freudenthal\todo{What is evidently proved is that it is Jordan.} algebra over $F$. Moreover, $J$ is a division algebra if and only if ($A$ is an associative division algebra) and $\mu$ is not a generic norm of $A$.  Note that if $A$ is central simple of dimension $9$, then $J(A,\mu)$ has dimension 27, hence is an Albert algebra.

\begin{eg} \label{e.fitspli}
(a): Let $E$ be a cubic \'etale $F$-algebra.  If $E$ is either split --- i.e., isomorphic to $F \times F \times F$ --- or a cyclic cubic field extension of $F$, then $J(E,1) \cong \Her_3(F \times F, \qform{1, -1, 1})$ \cite[Th.~3]{PR:toral}, i.e., is the split Freudenthal algebra of rank 9 (Proposition \ref{split.iso}), which is $\Mat_3(F)^+$ by Example \ref{Mat3.9}.

(b): Let $A$ be a central simple associative $F$-algebra of degree $3$. Then $J(A,1)$ is a split Albert algebra \cite[Cor.~4.2]{PR:Sp}. 
\end{eg}

\begin{lem} \label{l.genfit}
Let $A$ be a separable associative $F$-algebra of degree $3$ and $\mu \in F^\times$. Then the first Tits construction $J(A,\mu)$ is generated by $A^+$ $($identified in $J(A,\mu)$ through the initial summand$)$ and $(0,1_A,0)$. 	
\end{lem}

\begin{proof}
Let $J'$ be the subalgebra of $J(A, \mu)$ generated by $A^+$ and $w := (0,1_A,0)$.
As a subalgebra, it is closed under $\sharp$ by \eqref{sharp.square}, i.e., $x^\sharp \in J'$ and $x \times y \in J'$ for all $x, y \in J'$.
Since $w^\sharp = \mu(0,0,1_A)$ and $x_0 \times (0,x_1,x_2) = (0,-x_0x_1,-x_2x_0)$ for all $x_i \in A$ by \eqref{fitsha}, it follows that $J'$ must be all of $J(A, \mu)$.
\end{proof}

\begin{proof}[Proof of Proposition \ref{generators}] Theorem 1.2 in \cite{FirstRei} reduces the proof to the case where $R$ is a field $F$, in which case $\dim \Max R = 0$, so the task is to prove that 3 elements suffice to generate an Albert algebra $J$.  If $F$ is infinite, then Proposition 4.1 in \cite{FirstRei} reduces us to considering the case where $J$ is split.  If $F$ is finite, then $J$ is split by Proposition \ref{G2F4.dvr}.  

If $F \ne \FF_2$, then the split cubic \'etale $F$-algebra $E := F \times F \times F$ can be generated by a single element $x$ as an associative algebra.  On the other hand, if $F = \FF_2$, let $E := \FF_8$ be the cyclic cubic extension of $F$, which is again generated by one element, call it $x$.  In either case, $x$ also generates the Jordan algebra $E^+$, because the powers of $x$ in $E^+$ and $E$ are the same.

Hence Example~\ref{e.fitspli}~(a) and Lemma~\ref{l.genfit} show that $\Mat_3(F)^+$ is generated by two elements.   Lemma~\ref{l.genfit} combined with Example~\ref{e.fitspli}~(b) shows that the split Albert algebra $J(\Mat_3(F),1)$ is generated by three elements. 
\end{proof}

\begin{rmk}
That the Jordan algebra $\Mat_3(F)^+$, for any field $F$, is generated by two elements doesn't seem too surprising, but one should keep in mind that \emph{the analogous result for 2-by-2 matrices is false}: the minimal number of generators for the Jordan algebra $\Mat_2(F)^+$ is $3$.
\end{rmk}

\begin{rmk}[dichotomy of fields and the Tits construction] \label{Tits.const}
The classification of Albert algebras over a field $F$ of characteristic $\ne 3$ has a fundamentally different flavor depending on whether or not $H^3(F, \ZZ/3)$ is zero, as indicated by \cite{Rost:CR}, \cite{PR:el}, or \cite[\S8]{G:lens}.  If $H^3(F, \ZZ/3) = 0$ --- as is the case for global fields, $p$-adic fields, and the real numbers --- every Albert $F$-algebra is reduced, i.e., of the form $\Her_3(C, \Gamma)$ for some $C$ and $\Gamma$, and is not a division algebra.  (It is natural to speculate that this is the reason it took many years after Albert algebras were defined --- all the way until 1958 --- for the first Albert division algebra to be exhibited in \cite{Albert:div}.)  In the other case, when $H^3(F, \ZZ/3) \ne 0$, as happens when $F = \QQ(t)$ for example, one can construct an Albert division algebra via the first Tits construction described above as $J(A \otimes \QQ(t), t)$ for $A$ an associative division algebra of dimension 9 over $\QQ$.

It is known that every Albert algebra over a field is obtained by the first Tits construction or second Tits construction (which we have not described here), see \cite[Th.~10]{McC:FSTrev} or \cite[Th.~3.1(i)]{PR:class}.  Both constructions have been extended from the case of algebras over a field to an arbitrary base ring \cite{PR:jord3}.  However, in this more general setting, the Tits constructions do not produce all Albert algebras \cite{PST:Ti}.
\end{rmk}

%----------------------------------------------------------------------------------------------------------------
%
\section{Isotopy} \label{isotopy.sec}

The aim of this section is to discuss the notion of isotopy of Jordan algebras, which will pay off later in the paper when we discuss groups of type $\ee_6$ in \S\ref{E6.sec} and $\ee_7$ in \S\ref{E7.sec}.  We include this material at this point in the paper because Corollary \ref{iso.Frd} is needed in the following section.

\begin{defn}
Let $J$ be a Jordan $R$-algebra and suppose $u \in J$ is invertible.  We define a Jordan algebra $\uiso{J}$ with the same underlying $R$-module, with $U$-operator $\uiso{U}_x := U_x U_u$ (where the unadorned $U$ on the right denotes the $U$-operator in $J$), and with identity element $\uiso{1} := u^{-1}$.  One checks that $\uiso{J}$ is indeed a Jordan algebra and for $u, v$ invertible, we have $\iso{(\uiso{J})}{v} = \iso{J}{U_u v}$.  A Jordan $R$-algebra $J'$ is an \emph{isotope} of $J$ if it is isomorphic to $\uiso{J}$ for some invertible $u \in J$; equivalently one says that $J$ and $J'$ are \emph{isotopic}.  This defines an equivalence relation on Jordan algebras, which is a priori weaker than isomorphism.
\end{defn}

We have presented the notion of isotopy here for Jordan algebras.  However, there are analogous notions for other classes of algebras, which go back at least to \cite{Alb:na1}.  For associative algebras, isotopy is the same as isomorphism.  For octonion algebras, isotopy amounts to norm equivalence \cite[Cor.~6.7]{AlsaodyGille}, which is a weaker condition than isomorphism, see \cite{Gille:oct} and \cite{AsokHW}.

\subsection*{Isotopes of cubic Jordan algebras} if $J$ is a cubic Jordan $R$-algebra and $u \in J$ is invertible, then \cite[Th.~2]{McC:FST} and its proof show that the isotope $J^{(u)}$ is a cubic Jordan algebra as well whose identity element, adjoint and norm are given by
\begin{equation}
\label{cuno.isotope} 1_{J^{(u)}} = u^{-1}, \quad x^{\sharp(u)} = N_J(u)U_u^{-1}x^\sharp, \quad N_{J^{(u)}}(x) = N_J(u)N_J(x).
\end{equation}
Moreover, the (bi-)linear and quadratic trace of $J^{(u)}$ have the form
\begin{equation}
\label{tr.isotope} T_{J^{(u)}}(x,y) = T_J(U_ux,y), \quad \Tr_{J^{(u)}}(x) = T_J(u,x), \quad S_{J^{(u)}}(x) = T_J(u^\sharp,x^\sharp).
\end{equation}
The first equation of \eqref{tr.isotope} is in \cite[p.~500]{McC:FST} while the second one follow from \eqref{cuno.isotope}, the first, and Lemma~\ref{cub.comp}~(\ref{inv.char}) via  $\Tr_{J^{(u)}}(x) = T_{J^{(u)}}(u^{-1},x) = T_J(U_uu^{-1},x) = T_J(u,x)$. Similarly,
\[
S_{J^{(u)}}(x) = \Tr_{J^{(u)}}(x^{\sharp(u)}) = T_J(u,N_J(u)U_u^{-1}x^\sharp) = T_J(N_J(u)U_u^{-1}u,x^\sharp) = T_J(u^\sharp,x^\sharp).
\]

\begin{eg} \label{DIISJO}
 \emph{$\Her_3(C, \Gamma)$ is isotopic to $\Her_3(C)$ for every $\Gamma$}.  Indeed, for 
\[
u := \left( \begin{smallmatrix} \gamma_1 & 0 & 0 \\
0 & \gamma_2 & 0 \\
0 & 0 & \gamma_3 
\end{smallmatrix} \right) \quad  \in \Her_3(C, \Gamma),
\]
the map $\phi \!: \uiso{\Her_3(C, \Gamma)} \to \Her_3(C)$ defined by 
\[
\phi \sbasGjord = \sjordmat{\gamma_1 \alpha_1}{\gamma_2 \alpha_2}{\gamma_3 \alpha_3}{\gamma_2 \gamma_3 c_1}{\gamma_1 \gamma_3 c_2}{\gamma_1 \gamma_2 c_3} 
\]
is an isomorphism of Jordan algebras.  One can also turn this around:
\[
\Her_3(C,\Gamma) = (\Her_3(C,\Gamma)^{(u)})^{(u^{-2})} \cong \Her_3(C)^{(\phi(u^{-2}))} = \iso{\Her_3(C)}{u^{-1}}.
\]
\end{eg}

\subsection*{Jordan algebras isotopic to a reduced Freudenthal algebra}
In the special case where $R$ is a field, a Jordan algebra that is isotopic to the split Albert algebra $\Her_3(\Zor(R))$ is necessarily isomorphic to it, see for example \cite[p.~53, Th.~9]{Jac:ex}.  Some hypothesis on $R$ is necessary for the conclusion to hold.  Alsaody has shown in \cite[Th.~2.7]{Alsaody} that there exists a ring $R$ finitely generated over $\CC$ and an Albert $R$-algebra that is isotopic to the split Albert $R$-algebra but is not isomorphic to it.  Here we show that it is sufficient to assume that $R$ is a semi-local ring (Corollary \ref{Frd.iso.split}), as a consequence of a more general result (Theorem \ref{local.iso}) from which we also obtain the key Corollary \ref{iso.Frd}.

%In the case where $R$ is a field, it is also standard that a Freudenthal algebra isotopic to a reduced Freudenthal algebra is itself a reduced Freudenthal algebra; we now prove this in case $R$ is a (semi-)local ring, which provides the key Corollary \ref{iso.Frd}.

We work in a slightly more general context than semi-local rings.  For the following statements, see \cite{EstesGur} and \cite{McDWa}. We say that $R$ is \emph{an LG ring} if whenever a polynomial $f \in R[x_1, \ldots, x_n]$ represents a unit over $R_\mfm$ for every maximal ideal $\mfm$ of $R$, then $f$ represents a unit over $R$.  Every semi-local ring is an LG ring.  It is easy to see that rings $R_1$, $R_2$ are both LG if and only if their product $R_1 \times R_2$ is LG.  The ring of all algebraic integers and the ring of all real algebraic integers are LG rings.  Every integral extension of an LG ring is LG \cite[Cor.~2.3]{EstesGur}.  

\begin{thm} \label{local.iso}
Suppose $J$ is a Jordan $R$-algebra that is isotopic to $\Her_3(C, \Gamma)$ for some composition $R$-algebra $C$ and some $\Gamma$.  If
$R$ is an LG ring,
then $J$ is isomorphic to $\Her_3(C, \Gamma')$ for some $\Gamma'$.
\end{thm}

\begin{proof} 
As in previous proofs, we reduce to the case where $C$ has constant rank.  

In view of Example \ref{DIISJO}, $J$ is isotopic to $\Her_3(C)$, i.e.,  $J \cong \iso{\Her_3(C)}{u^{-1}}$ for some invertible $u \in \Her_3(C)$.  The same example shows we are done if $u$ is diagonal.

Write $N$ for the cubic form on $\Her_3(C)$.
In case $u$ is not diagonal, we will apply successive elements $\eta \in \GL(\Her_3(C))$ such that $N \eta = N$ as polynomial laws.  (In the notation of \S\ref{E6.sec} below, $\eta \in \bfIsom(\Her_3(C))(R)$.)  Note that each such $\eta$ defines an isomorphism of $R$-modules 
\begin{equation} \label{local.iso.1}
\eta \colon \iso{\Her_3(C)}{u^{-1}} \to \iso{\Her_3(C)}{\eta(u)^{-1}}.
\end{equation}
We have
\[
N(\eta(u)^{-1}) = N(\eta(u))^{-1} = N(u)^{-1} = N(u^{-1}),
\]
so we have by \eqref{cuno.isotope} that
\[
N_{\iso{\Her_3(C)}{u^{-1}}} = N(u)^{-1} N = N(\eta(u)^{-1}) N \eta = N_{\iso{\Her_3(C)}{\eta(u)^{-1}}}\eta.
\]
Since $\eta$ is a norm isometry that  maps the identity element $u^{-1}$ in the domain of \eqref{local.iso.1} to the identity element in the codomain, it is an isomorphism of algebras by Lemma \ref{HOCUJO}.  Thus, if successive elements $\eta$ transform $u$ into a diagonal element, the proof will be complete.

We employ the transformation $\tau_{st}(q)$ for $1 \le s \ne t \le 3$ and $q \in C$ defined by
\[
\tau_{st}(q) \: A \mapsto (I_3 + q E_{st}) A (I_3 + \bar{q} E_{ts}),
\]
where $I_3$ is the identity matrix, $E_{st}$ is the 3-by-3 matrix with a 1 in the $(s, t)$-entry and 0 elsewhere, and juxtaposition defines naive multiplication of 3-by-3 matrices with entries in $C$.  For example,
\[
\tau_{12}(q) \sbasjord = \sjordmat{\alpha_1+n_C(q, c_3) + \alpha_2 N_C(q)}{\alpha_2}{\alpha_3}{c_1}{c_2 + \bar{c}_1\bar{q}}{c_3 + \alpha_2 q}.
\]
These transformations appear in \cite[\S5]{Jac:J3} and \cite[\S2]{Krut:E6}; the argument in either reference shows that $\tau_{st}(q)$ preserves $N$ for all choices of $s$, $t$, and $q$.  For $e = 2, 3$, define polynomial functions $\nu_e$ from $C^e$ to the group scheme $\bfG$ of linear transformations stabilizing the norm $N$ via 
\[
\nu_3(q_1, q_2, q_3) = \tau_{31}(q_3) \tau_{21}(q_2) \tau_{12}(q_1) 
\quad \text{and} \quad
\nu_2(q_1, q_2) = \tau_{32}(q_2) \tau_{23}(q_1).
\]

Additionally, for every permutation $\pi$ of $\{ 1, 2, 3 \}$, there is a linear transformation that preserves $N$ (actually, an automorphism of the algebra) that maps
\begin{equation} \label{isom.permute}
\sbasjord \mapsto \sjordmat{\alpha_{\pi(1)}}{\alpha_{\pi(2)}}{\alpha_{\pi(3)}}{c'_{\pi(1)}}{c'_{\pi(2)}}{c'_{\pi(3)}},
\end{equation}
where $c'_{\pi(i)}$ is a linear function of $c_i$ for each $i$.  We abuse notation and denote the transformations also by $\pi$.

\medskip
\underline{\emph{Case: $R$ is local}:} We collect some observations in the case where $R$ is a local ring.  Write 
\[
u = \sbasjord.
\]
By hypothesis $N(u)$ is invertible, i.e., does not lie in the maximal ideal $\mfm$ of $R$.

If $\alpha_1 \in R^\times$ or $\alpha_2, \alpha_3 \notin R^\times$, then after modifying $u$ by a transformation in the image of $\nu_3$, we may assume that $\alpha_1 \in R^\times$ and $c_2 = c_3 = 0$.  Indeed, if $\alpha_1 \notin R^\times$, then by \eqref{Frd.norm} we have 
\[
N(u) \equiv \Tr_C(c_1 c_2 c_3) \bmod \mfm,
\]
whence $c_3 \notin \mfm C$. Since $n_C$ continues to be regular when changing scalars to $R/\mfm$, some $q \in C$ has $n_C(q,c_3) \notin \mfm$. Applying $\tau_{12}(q)$, we may arrange $\alpha_1 \in R^\times$.   Then note that
$\tau_{21}(q) \sbasjord$
has top row entries $\alpha_1$, $c_3 + \alpha_1 \bar{q}$, $\bar{c_2}$.  Taking $q = -\bar{c}_3
\alpha_1^{-1}$ shows that we may assume $c_3 = 0$.  The argument that we may assume $c_2 = 0$ is similar, with the role of $\tau_{21}$ replaced by $\tau_{31}$.

Now suppose that $\alpha \in R^\times$ and $c_2 = c_3 = 0$.  If $\alpha_2 \in R^\times$ or $\alpha_3 \notin R^\times$, then after modifying $u$ by a transformation in the image of $\nu_2$, we may assume that $u$ is diagonal.  Indeed, since $u$ has norm 
$\alpha_1 (\alpha_2 \alpha_3 - N_C(c_1)) \not\in \mfm$, at least one of $\alpha_2$, $\alpha_3$, or $N_C(c_1)$ is not in $\mfm$.  The same argument as in the preceding paragraph, with $\tau_{12}$ replaced by $\tau_{23}$, shows that we may assume that $\alpha_2 \not\in \mfm$.  The same argument as in the preceding paragraph, with $\tau_{21}$ replaced by $\tau_{32}$, shows that we may assume that $c_1 = 0$.  Thus, we have transformed $u$ into a diagonal element, as required.

\medskip
\underline{\emph{General case}:} Return to the setting of $R$ as in the statement of the theorem.  We combine the transformations $\nu_3$, $\nu_2$, and permutations together into a polynomial function  $C^{21} \to \bfG$, namely 
\begin{equation} \label{semi.prod}
\left( \nu_2 \, (2\,3)\, \nu_2 \nu_3 \,(1\,3\,2) \right)\left( \nu_2 \, (2\,3)\, \nu_2 \nu_3 \,(1\,2\,3) \right) \left( \nu_2 \, (2\,3)\, \nu_2 \nu_3 \right)
\end{equation}
where the arguments to the various $\nu_2$, $\nu_3$ are assigned independently.  Combining this with the polynomial function on $\bfG$ that sends $g \in \bfG(R)$ to the product of the diagonal entries of $gu$, we obtain a polynomial law in $\Pol(C^{21}, R)$.  But more is true.  Because $R$ is LG and $C$ is projective of constant rank, $C$ is a free module, see \cite[Th.~2.10]{EstesGur} or \cite[p.~457]{McDWa}.  Choosing a basis for $C$ expresses this polynomial law as a polynomial with coefficients in $R$.

We claim that this polynomial represents a unit over $R_{\mfm}$ for every maximal ideal $\mfm$ of $R$.  For a given $\mfm$, here is how to pick the element of $C^{21}$ that produces a unit.  If $\alpha_1 \in R_{\mfm}^\times$ or $\alpha_2, \alpha_3 \notin R_{\mfm}^\times$, applying $\nu_3$ to $u$, with arguments chosen as in the second paragraph of the local case, we obtain an element with $\alpha_1 \in R_{\mfm}^\times$ and $c_3 = c_2 = 0$.  We take this to be the rightmost term in \eqref{semi.prod}.  If that element has $\alpha_2 \in R_{\mfm}^\times$ or $\alpha_3 \notin R_{\mfm}^\times$, the next $\nu_2$ term can be chosen to produce a diagonal $u$; one takes the remaining $\nu$ terms in \eqref{semi.prod} to have argument 0.  Otherwise, $\alpha_3$ is invertible in $R_\mfm$, and we plug 0 into the rightmost $\nu_2$, pick the argument for the next $\nu_2$ as in the proof of the local case, and plug 0 into the remaining $\nu$ terms to the left in \eqref{semi.prod}.  The claim is verified if $\alpha_1 \in R_\mfm^\times$ or $\alpha_2, \alpha_3 \notin R_\mfm^\times$.

The next case of the claim is where $\alpha_2 \in R_\mfm^\times$.  In that case, we plug 0 into the rightmost three $\nu$ terms in \eqref{semi.prod}.  After applying the permutation $(2\,3)$ and then $(1\,2\,3)$, we obtain an element of $\Her_3(C)$ with $\alpha_1$ invertible and a well-chosen argument for the next $\nu_3$ term will assure that $c_3 = c_2 = 0$.  As in the preceding paragraph, choosing the arguments for the leftmost two $\nu_2$ terms in the middle product in \eqref{semi.prod} suffices to transform $u$ into a diagonal element, verifying the claim in this case.

The last case of the claim is when $\alpha_3 \in R_\mfm^\times$.  Plug 0 in the $\nu$ terms in the middle and right parenthetical expressions in \eqref{semi.prod}.  After applying all permutations in \eqref{semi.prod} besides the leftmost transposition to $u$, we obtain an element of $\Her_3(C)$ with $\alpha_1 \in R_{\mfm}^\times$ and the argument in the preceding paragraph again transforms $u$ into a diagonal element, completing the proof of the claim.

Since $R$ is an LG ring, the claim provides an element $g \in \bfG(R)$ such that $gu$ has $(1,1)$-entry a unit.  That is, we may assume that in the element $u$, $\alpha_1$ is invertible.  Applying now $\tau_{21}(q)$ and $\tau_{31}(q)$ to $u$ for $q$'s chosen as in the local case, we may assume that $c_3 = c_2 = 0$.  

Applying now an argument as in the preceding five paragraphs, with the function 
\[
\nu_2 \, (2\,3) \, \nu_2 \colon C^4 \to \bfG, 
\] 
we conclude that we may transform $u$ to further assume that $\alpha_2$ is invertible, and therefore apply a transformation $\tau_{32}(q)$ to transform it into a diagonal element, as required.
\end{proof}

\begin{cor} \label{Frd.iso.split}
Suppose $J$ is a Jordan $R$-algebra over an LG ring $R$.  If $J$ is isotopic to a split Freudenthal algebra whose rank does not take the value 6, then $J$ is itself a split Freudenthal algebra.
\end{cor}

\begin{proof}
As in previous proofs, one is reduced to the case where $J$ has constant rank, which is not 6.  The theorem and Proposition \ref{split.iso} give the claim.
\end{proof}

The hypothesis that the rank is not 6 is necessary, because $\Her_3(\RR, \qform{1, 1, -1})$ is isotopic to the split Freudenthal algebra $\Her_3(\RR)$ (Example \ref{DIISJO}) but is not isomorphic to it (Example \ref{real.trace}).

\begin{eg}[isotopy over global fields] \label{isotopy.global}
For $F = \RR$ or a global field, there is a bijection between the isomorphism classes of octonion algebras and isotopy classes of Albert algebras given by $C \leftrightarrow \Her_3(C)$.  Indeed, every Albert $F$-algebra is reduced (Example \ref{alb.global}), so $C \mapsto \Her_3(C)$ touches every isotopy class.  For injectivity, if $C, C'$ are distinct octonion algebras, there is a real embedding $F \hookrightarrow \RR$ such that $C \otimes \RR \not\cong C' \otimes \RR$, and Corollary \ref{Frd.iso.split} shows that $\Her_3(C) \otimes \RR$ and $\Her_3(C') \otimes \RR$ are not isotopic.
\end{eg}

\begin{cor} \label{iso.Frd}
Every isotope of a Freudenthal algebra is itself a Freudenthal algebra. 
\end{cor}

\begin{proof}
Suppose $J$ is an isotope of a Freudenthal algebra.  After base change to a faithfully flat extension, $J$ is an isotope of a split Freudenthal algebra.

The $R$-algebra $S := \prod_\mfm R_\mfm$, where $\mfm$ ranges over maximal ideals of $R$, is faithfully flat.  For each $\mfm$, $J \otimes R_\mfm$ is $\Her_3(C, \Gamma)$ for $C$ a split composition $R_\mfm$-algebra and some $\Gamma$ by Theorem \ref{local.iso}.  By Proposition \ref{red.is.Frd}, there is a faithfully flat $R_\mfm$-algebra $T$ such that $J \otimes T$ is a split Freudenthal algebra.  The product of these $T$'s is a faithfully flat $R$-algebra over which $J$ is the split Freudenthal algebra.
\end{proof}

We close this section by making explicit the relationship between isotopy and norm similarity between Freudenthal algebras, extending Lemma \ref{HOCUJO}.

\begin{prop} \label{isotopy.similar}
Let $J$ and $J'$ be Freudenthal $R$-algebras.  For an $R$-linear map $\phi \!: J \to J'$, the following are equivalent:
\begin{enumerate}
\item \label{isotopy.similar.1} $\phi$ is an isomorphism $J \to \uiso{(J')}$ for some invertible $u \in J'$ (``$\phi$ is an isotopy'').
\item \label{isotopy.similar.2} $N_{J'} \phi = \alpha N_J$ as polynomial laws for some $\alpha \in R^\times$, and $\phi$ is surjective (``$\phi$ is a norm similarity'').
\end{enumerate}
\end{prop}

\begin{proof}
Since $(J^\prime)^{(u)}$ is a Freudenthal algebra by Corollary~\ref{iso.Frd}, condition (\ref{isotopy.similar.2}) follows from (\ref{isotopy.similar.1}) by Lemma~\ref{HOCUJO} and \eqref{cuno.isotope}. Conversely, we assume \eqref{isotopy.similar.2} and prove \eqref{isotopy.similar.1}.  Because $N_{J'}(\phi(1_J)) = \alpha$, the element $\phi(1_J)$ is invertible in $J'$.  We set $u := \phi(1_J)^{-1}$ and $J'' := \uiso{(J')}$.  We have 
\[
\phi(1_J) = u^{-1} = 1_{J''}.
\]
Also, $N_{J'}(u) = N_{J'}(\phi(1_J))^{-1} = \alpha^{-1}$.  Then
\[
N_{J''} \phi = N_{J'}(u) N_{J'} \phi = N_J
\]
as polynomial laws.  Lemma \ref{HOCUJO} implies that $\phi$ is an isomorphism $J \xrightarrow{\sim} J''$, as desired.
\end{proof}

%%%%%%%%%%%%%%%%%%%%%%%%%%%%%%%%%%%%%%%%%%%%%%%%%%%%%%%%%%
%\section{An Albert algebra constructed from the Elkies-Gross lattice} \label{EG.sec}

%----------------------------------------------------------------------------------------------------------------
%
\section{Classification of Albert algebras over \texorpdfstring{$\ZZ$}{Z}}

In this section, we study Albert algebras over the integers.

\begin{defn} \label{EG.def}
In the notation of Example \ref{octaves},
consider the element 
\[
\beta := (-1 + e_1 + e_2 + \cdots + e_7)/2 = h_1 + h_2 + h_3 - (2 + e_1) \quad \in \oct,
\]
as was done in \cite[(5.2)]{ElkiesGross}.
That element has 
\[
\Tr_\oct(\beta) = -1, \quad n_\oct(\beta) = 2, \quad \text{and} \quad \beta^2 + \beta + 2 = 0.
\]
Put 
\[
v := \sjordmat{2}{2}{2}{\beta}{\beta}{\beta} \quad \in \Her_3(\oct).
\]
Since $\Tr_\oct(\beta^3) = 5$, we find that  $N_{\Her_3(\oct)}(v) = 1$.  In particular, $v$ is invertible with inverse $v^\sharp$.  We define $\Lambda := \iso{\Her_3(\oct)}{v}$; it is an Albert algebra by Corollary \ref{iso.Frd}.
\end{defn}

\begin{prop} \label{La.def}
$\Her_3(\oct) \not\cong \Lambda$ as Jordan $\ZZ$-algebras, but $\Her_3(\oct) \otimes \QQ \cong \Lambda \otimes \QQ$ as Jordan $\QQ$-algebras.
\end{prop}

\begin{proof}
We first prove the claim over $\ZZ$, which amounts to a computation from \cite{ElkiesGross}.  The isomorphism class of a Freudenthal algebra determines its cubic norm form and also its trace linear form.  %Writing $\times$ for the bilinearization of $\nat$, i.e., $y \times z := (y + z)^\nat - y^\nat - z^\nat$, we can express the right side of the displayed equation as $\frac12 \Tr_{\Her_3(\oct)}(v \times v, x)$,  which was denoted $T_v(x)$ in \cite[(2.5)]{ElkiesGross}\todo{H: I looked at this reference and couldn't locate the expression $T_v(x)$ there. Besides, this is immaterial for the subsequent consideratiosnns. I would delete the entire sentence after the displayed equation.} 
From \eqref{cuno.isotope} we deduce for $x \in \Her_3(\oct)$ that $x^{\sharp(v)} = 0$ if and only if $x^\sharp = 0$. Hence \cite[Prop.~5.5]{ElkiesGross} says that $\Her_3(\oct)$ contains exactly 3 elements $x$ such that $x^\nat = 0$ and $\Tr_{\Her_3(\oct)}(x) = 1$, whereas $\Lambda$ has no elements $x$ such that $x^{\sharp(v)} = 0$ and 
\[
T_{\Her_3(\oct)}(v, x) = 1,
\]
where the left side is $\Tr_\Lambda(x)$ by \eqref{tr.isotope}.
This proves that $\Her_3(\oct) \not\cong \Lambda$.

% CUEREV CHARPOS MINPOS
Now consider $\Her_3(\oct) \otimes \RR$.  It is called a ``euclidean'' Jordan algebra or, in older references, a ``formally real'' Jordan algebra, because every sum of nonzero squares is not zero \cite[p.~331]{BK}.  The element $v$ has generic minimal polynomial, in the sense of \eqref{gen.min},
%\[ x^3 - 6x^2 + 6x - 1 = 
$(t-1)(t^2 - 5t + 1)$,
%\]
which has three positive real roots.  Therefore, there is some $u \in \Her_3(\oct) \otimes \RR$ such that $u^2 = v$ \cite[\S{XI.3}, S.~3.6 and 3.7]{BK}.
From this, it is trivial to see that 
\[
\La \otimes \RR \cong \iso{(\Her_3(\oct) \otimes \RR)}{v} \cong \Her_3(\oct) \otimes \RR,
\]
and Example \ref{alb.global} gives that $\La \otimes \QQ \cong \Her_3(\oct) \otimes \QQ$.
\end{proof}

\begin{thm} \label{skip.AlbZ}
Over $\IZ$:
\begin{enumerate}
\renewcommand{\theenumi}{\alph{enumi}}
\item \label{skip.octZ} There are exactly two isomorphism classes of octonion algebras: $\Zor(\IZ)$ and $\oct$.
\item \label{skip.AlbZ.2} There are exactly four isomorphism classes of Albert algebras: $\Her_3(\Zor(\IZ))$, $\Her_3(\oct, \qform{1, -1, 1})$, $\Her_3(\oct)$, and $\La$.
\item \label{skip.AlbZ.isotopy} There are exactly two isotopy classes of Albert algebras: $\Her_3(\Zor(\IZ))$ and $\Her_3(\oct)$.
\end{enumerate}
\end{thm}

\begin{proof}
We first prove \eqref{skip.octZ} and \eqref{skip.AlbZ.2}.
No pair of the algebras listed are isomorphic to each other.  For $\Her_3(\oct)$ and $\La$, this is Prop.~\ref{La.def}.  For any other pair, base change to $\IQ$ yields non-isomorphic $\IQ$-algebras.  To complete the proof, it suffices to show that every octonion or Albert $\ZZ$-algebra $B$ is isomorphic to one of the ones listed.

If $B$ is indefinite --- i.e., $B \otimes \RR$ is not isomorphic to $\OO$ nor $\Her_3(\OO)$ --- then the isomorphism class of $B$ is determined by $B \otimes \QQ$ as a $\QQ$-algebra (Prop.~\ref{skip.numb.indef}).   Since the indefinite octonion or Albert $\IQ$-algebras are $\Zor(\IQ)$, $\Her_3(\Zor(\IQ))$, and $\Her_3(\oct \otimes \QQ, \qform{1, -1, 1})$ by Example \ref{alb.global},  $B$ is isomorphic to one of the algebras listed in the statement.

If $B$ is definite, then $\bfAut(B)$ is a $\ZZ$-form of the compact real group of type $\gee_2$ or $\eff_4$.  Gross's mass formula \cite[Prop.~5.3]{Gross:Z} shows that, up to $\ZZ$-isomorphism, there is only one group of type $\gee_2$ and two groups of type $\eff_4$ with this property.  Using the equivalence between these groups and octonion or Albert algebras (Th.~\ref{skip.F4.alb}), we conclude that up to isomorphism $\oct$ is the unique definite octonion $\ZZ$-algebra and $\Her_3(\oct)$ and $\La$ are the two isomorphism classes of Albert $\ZZ$-algebras, completing the proof of \eqref{skip.octZ} and \eqref{skip.AlbZ.2}.

For \eqref{skip.AlbZ.isotopy}, note that the three algebras in \eqref{skip.AlbZ.2} that are not $\Her_3(\Zor(\IZ))$ are all isotopic, see Example \ref{DIISJO}, so the two algebras listed in \eqref{skip.AlbZ.isotopy} represent all of the isotopy classes of Albert $\IZ$-algebras.  The base change of these two algebras to $\IQ$ are not isotopic (Example \ref{isotopy.global}), so they are not isotopic as $\IZ$-algebras.
\end{proof}

Note that part \eqref{skip.octZ} of the theorem can be proved entirely in the language of octonion algebras, see \cite{vdBSp:G2}.  

In view of Theorem \ref{skip.F4.alb}, part \eqref{skip.AlbZ.2} is equivalent to a classification of the group schemes of type $\eff_4$ over $\ZZ$, which was done in Sections 6 and 7 of \cite{Conrad:Z}, especially Examples 6.7 and 7.4.  The innovation here is that we can use the language of Albert algebras also in the case of $\ZZ$ where $2$ is not invertible.  Because of this extra flexibility, we can substitute results from the literature over algebraically closed fields (including characteristic 2) for some of the computations over $\ZZ$ done in \cite{Conrad:Z}.

Part \eqref{skip.AlbZ.isotopy} corresponds to the classification of groups of type $\ee_6$ over $\ZZ$ up to isogeny, see \S\ref{E6E7.sec}.

%%%%%%%%%%%%%%%%%%%%%%%%%%%%%%%%%%%%%%%%%%%%%%%%%%%%%%%%%%

%----------------------------------------------------------------------------------------------------------------
%
\section{Groups of type \texorpdfstring{$\ee_6$}{E6}} \label{E6.sec}

\subsection*{Roundness of the norm}
We note that the cubic norm of a Freudenthal algebra has the following special property.
 A quadratic form with this property is called ``round'', see \cite[\S9.A]{EKM}.
 
\begin{lem}[roundness] \label{round}
For every Freudenthal $R$-algebra $J$,
\[
\{ \alpha \in R^\times \mid \alpha N_J \cong N_J \} = \{ N_J(x) \in R^\times \mid \text{$x$ invertible in $J$} \}.
\]
\end{lem}

\begin{proof}
If $\alpha \in R^\times$ and $\phi \in \GL(J)$ are such that $\alpha N_J = N_J \phi$, then for $x := \phi(1_J)$ we have $N_J(x) = \alpha$.  Conversely, if $x$ is invertible in $J$, put $\alpha := N_J(x)$ and define $\phi := \alpha U_{x^{-1}}$.  Then $N_J \phi = \alpha^3 N_J(x^{-1})^2 N_J$ by Lemma~\ref{cub.comp}(\ref{nor.comp}), so $N_J \phi = \alpha N_J$.
\end{proof}

\begin{eg} \label{norm.surj}
For $J = \Her_3(C, \Gamma)$, the sets displayed in Lemma \ref{round} equal $R^\times$.  
To see this for the right side, take $\alpha \in R^\times$ and note that $N_J(\alpha \eps_1 + \eps_2 + \eps_3) = \alpha$.  For the left side,
consider $\phi \in \GL(J)$ defined by
\begin{gather*}
\phi(\eps_i) = \alpha \eps_i \quad \text{and} \quad \phi(\delta_i(c)) = \delta_i(c) \quad \text{for $i = 1, 2$}, \\
\phi(\eps_3) = \alpha^{-1} \eps_3 \quad \text{and} \quad \phi(\delta_3(c)) = \delta_3(\alpha c).
\end{gather*} 
Then $N_J \phi = \alpha N_J$ as polynomial laws.
\end{eg}

\begin{eg}
In contrast to the preceding example, we now show that the sets displayed in Lemma \ref{round} may be properly contained in $R^\times$.
Suppose $F$ is a field and $J$ is a Freudenthal $F$-algebra such that $N_J$ is anisotropic, i.e., $N_J(x) = 0$ if and only if $x = 0$.  (For example, such a $J$ exists if $F$ is Laurent series or rational functions in one variable over a global field, see Remark \ref{Tits.const}.)  We claim that, for $t$ an indeterminate, every nonzero element in the image of $N_{J \otimes F((t))}$ has lowest term of degree divisible by 3.  Because the norm is a homogeneous form, it suffices to prove this claim for $J \otimes F[[t]]$.

Let $x \in J \otimes F[[t]]$ be nonzero, so $x = \sum_{j \ge j_0} x_j t^j$ for some $j_0 \ge 0$ with $x_{j_0} \ne 0$.  Since $N_J$ is anisotropic, $N_J(x_0) \ne 0$.  
 If $j_0 = 0$, then the homomorphism $F[[t]] \to F$ such that $t \mapsto 0$ sends $x \mapsto x_0$ and $N_{J\otimes F[[t]]}(x) \mapsto N_J(x_0) \ne 0$, therefore $N_{J \otimes F[[t]]}(x)$ has lowest degree term $N_J(x_0)t^0$.
If $j_0 > 0$, then
\[
N_{J \otimes F[[t]]}(x) = N_{J \otimes F[[t]]}(t^{j_0} (xt^{-j_0})) = t^{3j_0} (N_J(x_{j_0}) t^0 + \cdots),
\]
proving the claim.
\end{eg}

\begin{cor} \label{iso.sim}
For Freudenthal $R$-algebras $J$ and $J'$, the following are equivalent:
\begin{enumerate}
\item \label{iso.sim.isotopic} $J$ and $J'$ are isotopic.
\item \label{iso.sim.sim} $N_J \cong \alpha N_{J'}$ for some $\alpha \in R^\times$.
\item \label{iso.sim.iso} $N_J \cong N_{J'}$.
\end{enumerate}
\end{cor}

\begin{proof}
The equivalence of \eqref{iso.sim.isotopic} and \eqref{iso.sim.sim} is Proposition \ref{isotopy.similar}.
 
Supposing \eqref{iso.sim.sim}, let $\phi \colon J' \to J$ be an $R$-module isomorphism such that $\alpha N_{J'} = N_J \phi$.  Take $x := \phi(1_{J'})$.  Since $N_J(x) = \alpha$, Lemma \ref{round} gives that $\alpha N_J \cong N_J$.  As $N_J$ is also isomorphic to $\alpha N_{J'}$, we conclude \eqref{iso.sim.iso}.  The converse is trivial.
\end{proof}

In the corollary, the inclusion of \eqref{iso.sim.iso} seems to be new, even in the case where $R$ is a field.  Omitting that, in the special case where $R$ is a field of characteristic $\ne 2, 3$, the equivalence of \eqref{iso.sim.isotopic} and \eqref{iso.sim.sim} and Proposition \ref{E6.iso} below can be found as Theorems 7 and 10 in \cite{Jac:ex}.

\subsection*{Albert algebras and groups of type $\ee_6$}
The stabilizer of the cubic form $N_J$ in $\bfGL(J)$ is a closed sub-group-scheme denoted $\bfIsom(J)$.  It contains $\bfAut(J)$ as a natural sub-group-scheme.
Arguing as in the proof of Lemma \ref{skip.aut.alb}, one finds that $\bfIsom(J)$ is a simple affine group scheme that is simply connected of type $\ee_6$.  (In the case where $R$ is an algebraically closed field, this claim is verified in \cite[11.20, 12.4]{Sp:jord}, or see \cite[Th.~7.3.2]{Sp:ex} for the case where $R$ is a field of characteristic different from 2, 3.)  Compare \cite[Lemma 2.3]{Alsaody} or \cite[App.~C]{Conrad:Z}.  
Moreover, $\bfIsom(J)$ is a ``pure inner form'' in the sense of \cite[\S3]{Conrad:Z}, resp.~``strongly inner'' in \cite[Def.~2.2.4.9]{CalmesFasel}, meaning that it is obtained by twisting the group scheme $\bfIsom(J_0)$ for the split Albert algebra $J_0$ by a class with values in $\bfIsom(J)$.  We note that the center of $\bfIsom(J)$ is the group scheme $\mu_3$ of cube roots of unity operating on $J$ by scalar multiplication 

Faithfully flat descent shows that the set $H^1(R, \bfIsom(J))$ is in bijection with isomorphism classes of pairs $(M, f)$, where $M$ is a projective module of the same rank as $J$ and $f$ is a cubic form on $M$ --- i.e., an element of $\Sym^3(M^*)$ --- such that $f \otimes S$ is isomorphic to the norm on $\Her_3(\Zor(S))$ for some faithfully flat $S \in \Ralg$.  For every Albert $R$-algebra $J$ and every $\alpha \in R^\times$, $(J, \alpha N_J)$ is such a pair by Example \ref{norm.surj}.  
 In the special case where $R$ is a field, every such pair $(M, f)$ --- i.e., every element of $H^1(R, \bfIsom(J))$ --- is of the form $(J, \alpha N_J)$ for some $J$ and $\alpha \in \Rx$, see \cite[9.12]{G:lens} in general or \cite{Sp:cubic} for the case of characteristic $\ne 2, 3$. 
 
\subsection*{Outer automorphism of $\bfIsom(J)$}
Suppose $J$ and $J'$ are Freudenthal $R$-algebras and $\phi \colon J \to J'$ is an isomorphism of $R$-modules.  Since the bilinear form $T_{J'}$ is regular, there is a unique $R$-linear map $\phi^\dagger \colon J \to J'$ such that $T_{J'}(\phi x, \phi^\dagger y) = T_J(x, y)$ for all $x, y \in J$.  Because $T_J$ and $T_{J'}$ are symmetric, we have $(\phi^\dagger)^\dagger = \phi$ for all $\phi$.  If $J''$ is another Freudenthal $R$-algebra and $\psi \colon J' \to J''$ is an $R$-linear bijection, then $(\phi \psi)^\dagger = \phi^\dagger \psi^\dagger$.
 
\begin{prop} \label{STRUG}
Let $J$ be a Freudenthal $R$-algebra.
\begin{enumerate}
\item \label{STRUG.1} If $\phi \in \GL(J)$ is such that $N_J \phi = \alpha N_J$ for some $\alpha \in R^\times$, then $N_J \phi^{\dagger} = \alpha^{-1} N_J$.
\item \label{STRUG.2} The map $\phi \mapsto \phi^\dagger$ is an automorphism of $\bfIsom(J)$ of order 2 that is not an inner automorphism.
\item \label{STRUG.3} For $\phi$ as in \eqref{STRUG.1} or in $\bfIsom(J)(R)$, $\phi^\dagger = \phi$ if and only if $\phi$ is an automorphism of $J$.
\end{enumerate}
\end{prop}

\begin{proof}
\eqref{STRUG.1}: Put $u := \phi(1_J)^{-1}$.  On the one hand,  
\[
T_J(x,y) = T_{\uiso{J}}(\phi(x), \phi(y))
\]
for all $x, y \in J$,  because $\phi$ is an isomorphism $J \to \uiso{J}$ by Proposition \ref{isotopy.similar}.  On the other hand, \eqref{tr.isotope} yields
\[
T_{\uiso{J}}(\phi(x), \phi(y)) = T_J(U_u \phi(x), \phi(y)).
\]
Therefore,
\begin{equation} \label{STRUG.4}
\phi^{\dagger} = U_{\phi(1_J)}^{-1} \phi.
\end{equation}
To complete the proof of \eqref{STRUG.1}, we note by Lemma~\ref{cub.comp}(\ref{nor.comp}) that 
\[
N_J \phi^\dagger = N_J U_u \phi = N_J(u)^2 N_J \phi = \alpha^{-1} N_J.
\]

For \eqref{STRUG.2}, we only have to check that the map is not an inner automorphism.  Let $S \in \Ralg$ be such that there exists $\zeta \in \mu_3(S)$ such that $\zeta \ne 1$.  Then $\zeta^\dagger = \zeta^{-1} \ne \zeta$ and $\zeta$ is in the center of $\bfIso(J)(R)$, proving that the automorphism is not inner (and not the identity).

For \eqref{STRUG.3}, suppose $\phi^\dagger = \phi$.  Then $N_J \phi = N_J$.  By \eqref{STRUG.4}, $U_{\phi(1_J)} = \Id_J$, so $\phi(1_J) = \zeta 1_J$ for some $\zeta \in R$ with $\zeta^2 = 1$ (Example \ref{Frd.squares}).  Yet $1 = N_J(1_J) = N_J \phi(1_J)$, so $\zeta^3$ also equals 1, whence $\phi(1_J) = 1_J$.  Lemma \ref{HOCUJO} shows that $\phi$ is an automorphism of $J$.  Conversely, if $\phi$ is an automorphism of $J$, then $u = 1_J$, so $\phi^\dagger = \phi$ by \eqref{STRUG.4}. 
\end{proof}

%We say that Albert $R$-algebras $J$ and $J'$ are \emph{isometric} if there is an $R$-linear bijection $\phi \colon J \to J'$ such that $N_{J'} \phi = N_J$ as polynomial laws.

\begin{prop} \label{E6.iso}
Let $J$ and $J'$ be Albert $R$-algebras.  Among the statements
\begin{enumerate}
\item \label{E6.giso} $\bfIsom(J) \cong \bfIsom(J')$.
\item \label{E6.psim} There is a line bundle $L$ and isomorphism $h \colon L^{\otimes 3} \to R$ such that $(J', N_{J'}) \cong [L, h] \cdot (J, N_J)$ for $\cdot$ as defined in \S\ref{directional.sec}.
\item \label{E6.isotopic} $J$ and $J'$ are isotopic.
\end{enumerate}
we have the implications \eqref{E6.giso} $\Leftrightarrow$ \eqref{E6.psim} $\Leftarrow$ \eqref{E6.isotopic}.  If $\Pic R$ has no 3-torsion other than zero, then all three statements are equivalent.
 \end{prop}
 
 \begin{proof}
Suppose \eqref{E6.giso}; we prove \eqref{E6.psim}.  We may assume $R$ is connected.
 
 The conjugation action gives a homomorphism $\bfIsom(J) \to \bfAut(\bfIsom(J))$, which gives a map of pointed sets
 \begin{equation} \label{isom.1}
 H^1(R, \bfIsom(J)) \to H^1(R, \bfAut(\bfIsom(J))),
 \end{equation}
 where the second set is in bijection with isomorphism classes of $R$-group schemes that become isomorphic to $\bfIsom(J)$ after base change to an fppf $R$-algebra.  By hypothesis, the class of $N_{J'} \in H^1(R, \bfIsom(J))$ is in the kernel of \eqref{isom.1}.

There is an exact sequence
\[
1 \to \bfIsom(J) / \mu_3 \to \bfAut(\bfIsom(J)) \to \ZZ/2 \to 1
\]
of fppf sheaves by \cite[Th.~XXIV.1.3]{SGA3.3}.  Since $R$ is connected, $(\ZZ/2)(R)$ has one non-identity element, and it is the image of the map $\dagger$ from Lemma \ref{STRUG}.  That is, in the exact sequence
\[
\bfAut(\bfIsom(J))(R) \to (\ZZ/2)(R) \to H^1(R, \bfIsom(J)/\mu_3) \to H^1(R, \bfAut(\bfIsom(J))),
\]
the first map is surjective, so the third map has zero kernel and we deduce that the image of $N_{J'}$ in $H^1(R, \bfIsom(J)/\mu_3)$ is the zero class.  It follows that $N_{J'}$ is in the image of the map 
\[
H^1(R, \mu_3) \to H^1(R, \bfIsom(J)),
\]
which is the orbit of the zero class $N_J$ under the action of the group $H^1(R, \mu_3)$, which is \eqref{E6.psim}. 

That \eqref{E6.psim} implies \eqref{E6.giso} is Lemma \ref{proj.sim}.
The claimed implications between \eqref{E6.isotopic} and \eqref{E6.psim} are Corollary \ref{iso.sim}.
 \end{proof}

\section{Freudenthal triple systems} \label{FT.sec}

In this section, we define Freudenthal triple systems, also known as FT systems.  We will see in Theorem \ref{E7.thm} in the next section that they play the same role relative to groups of type $\ee_7$ that forms of the norm on an Albert algebra play for groups of type $\ee_6$.

For any Albert $R$-algebra $J$, define  $\Frd(J)$ to be the rank 56 projective $R$-module $R \oplus R \oplus J \oplus J$ endowed with a 4-linear form $\Psi$ and an alternating bilinear form $b$, defined as follows.   

We write an element of $\Frd(J)$ as $\stbtmat{\alpha}{x}{x'}{\alpha'}$ for $\alpha, \alpha'  \in R$ and $x, x' \in J$.  Define
\begin{equation} \label{Frd.b.def}
b_J \left( \stbtmat{\alpha}{x}{x'}{\alpha'}, \stbtmat{\beta}{y}{y'}{\beta'} \right) := \alpha \beta' - \alpha' \beta  + T_J(x, y') - T_J(x', y).
\end{equation}
As an intermediate step to defining $\Psi$, define a quartic form
\begin{equation} \label{Frd.q.def}
q_J \stbtmat{\alpha}{x}{x'}{\alpha'} = -4T_J(x^\nat, x'^\nat)+ 4\alpha N_J(x) + 
  4 \alpha' N_J(x') + (T_J(x, x') - \alpha \alpha')^2,
  \end{equation}
compare \cite[p.~87]{Brown:E7} or \cite[p.~940]{Krut:E7}.  

To define the 4-linear form, consider first the case $R = \ZZ$ and $J := \Her_3(\Zor(\ZZ))$.  (The following definitions are inspired by \cite[\S6]{Lurie}.)
Putting $X_i$ for an element of $\Frd(J)$ and $t_i$ for an indeterminate, the coefficient of $t_1 t_2 t_3 t_4$ in $q(\sum t_i X_i)$, equivalently, the 4-linear form 
\[
(X_1, X_2, X_3, X_4) \mapsto  \grd_{X_1} \grd_{X_2} \grd_{X_3} q(X_4)
\]
on $Q(J)$, equals $2\Theta$ for a symmetric 4-linear form $\Theta$.  Define 4-linear forms $\Phi_i$  via
\begin{align}
\Phi_1(X_1, X_2, X_3, X_4) &= b(X_1,  X_2)\, b(X_3, X_4) \\
\Phi_2(X_1, X_2, X_3, X_4) &= b(X_1,  X_3)\, b(X_4, X_2) \label{FT.Phi.def} \\
\Phi_3(X_1, X_2, X_3, X_4) &= b(X_1,  X_4)\, b(X_2, X_3).
\end{align}
Then $\Theta + \sum \Phi_i$ is divisible by 2 as a 4-linear function on $Q(\Zor(\ZZ))$ and we set 
\begin{equation} \label{Psi.def}
\Psi_{\Her_3(\Zor(\ZZ))} := \frac12 (\Theta + \sum \Phi_i).
\end{equation}
As $\Theta$ is symmetric, $\Psi$ is evidently stable under even permutations of its arguments, and we have:
\[
\Psi(X_1, X_2, X_3, X_4)  - \Psi(X_2, X_1, X_3, X_4) = \sum \Phi_i.
\]
% (The analogous formula in \cite[p.~572]{Lurie} differs by a sign, a typo.)

For any ring $R$, we define $\Psi_{\Her_3(\Zor(R))} := \Psi_{\Her_3(\Zor(\ZZ))} \otimes R$, and we define $\Psi_J$ for an Albert $R$-algebra $J$ by descent.

\begin{defn} \label{FTS.def}
A \emph{Freudenthal triple system}\footnote{See p.~273 of \cite{Sp:e7} for remarks on the history of this term.} or \emph{FT system} $(M, \Psi, b)$ is an $R$-module $M$ endowed with a 4-linear form $\Psi$ and an alternating bilinear form $b$, such that $(M, \Psi, b) \otimes S$ is isomorphic (in an obvious sense) to $\Frd(J)$ for some faithfully flat $S \in \alg{R}$ and some Albert $S$-algebra $J$. 
\end{defn}

\subsection*{Comparison with other definitions}
Suppose for this paragraph that 6 is invertible in $R$.  
Given an FT system $(M, \Psi, b)$, we may define 4-linear forms $\Phi_i$ on $M$ via \eqref{FT.Phi.def} and recover $\Theta$ and $q$ via
\begin{equation} \label{q.recover}
\Theta := 2 \Psi - \sum \Phi_i \quad \text{and} \quad  \Theta(X,X,X,X) = 12q(X)
\end{equation}
 as polynomial laws in $X$.   (This last is a special case of the general fact that going from a homogeneous form of degree $d$ to a $d$-linear form and back to a homogeneous form of degree $d$ equals multiplication by $d!$ \cite[\S{IV.5.8}, Prop.~12(i)]{Bou:alg2}.)  Since the form $b$ is regular and $\Theta$ is symmetric, the equation
\[
\Theta(X_1, X_2, X_3, X_4) = b(X_1, t(X_2, X_3, X_4))
\]
implicitly defines a symmetric 3-linear form $t \colon M \times M \times M \to M$, and $\bfAut(M, \Psi, b)$ equals  $\bfAut(M, t, b)$.  That is, under the hypothesis that 6 is invertible in $R$, we would obtain an equivalent class of objects if we replaced the asymmetric 4-linear form $\Psi$ in the definition of FT systems with the quartic form $q$ (the version studied in \cite{Brown:E7}) or with the trilinear form $t$ (the version studied in \cite{Meyb:FT}).

\subsection*{Similarity of FT systems}
For a $d$-linear form $f$ on an $R$-module $M$, i.e., an $R$-linear map $f \colon M^{\otimes d} \to R$, and a $d$-trivialized line bundle $[L, h] \in H^1(R, \mu_d)$, we define a $d$-linear form $[L, h] \cdot f$ on $M \otimes L$ via the composition
\[
(M \otimes L)^{\otimes d} \xrightarrow{\sim} M^{\otimes d} \otimes L^{\otimes d} \xrightarrow{f \otimes h} R.
\]
For $\Frd := (M, \Psi, b)$ an FT system and a discriminant module $[L, h] \in H^1(R, \mu_2)$, we define $[L, h] \cdot \Frd$ to be the triple consisting of the module $M \otimes L$, the 4-linear form $[L, h^{\otimes 2}] \cdot \Psi$ for $[L, h^{\otimes 2}] \in H^1(R, \mu_4)$, and the bilinear form $[L, h] \cdot b$.  Since $\qform{1} \cdot \Frd$ is $\Frd$ itself, we deduce that $[L, h] \cdot \Frd$ is also an FT system.  We say that FT systems $\Frd$, $\Frd'$ are \emph{similar} if $\Frd' \cong [L, h] \cdot \Frd$ for some $[L, h] \in H^1(R, \mu_2)$.  For example, for any FT system $(M, \Psi, b)$ and any $\alpha \in R^\times$, $(M, \Psi, b)$ and $(M, \alpha^2 \Psi, \alpha b)$ are similar.

\begin{eg}
Suppose $(M, \Psi, b) = \Frd(J)$ for some Albert $R$-algebra $J$.  Then for every $\mu \in R^\times$, the map
\[
\stbtmat{\alpha}{x}{x'}{\alpha'} \mapsto \stbtmat{\alpha/\mu}{\mu x}{x'}{\mu^2 \alpha'}
\]
is an isomorphism $\qform{\mu} \cdot \Frd(J) \xrightarrow{\sim} \Frd(J)$.  One checks this for $R = \ZZ$ and $J = \Her_3(\Zor(\ZZ))$ using \eqref{Frd.b.def} and \eqref{Frd.q.def}.  It follows for general $R$ and $J$ by base change and twisting.
\end{eg}

%----------------------------------------------------------------------------------------------------------------
%

\section{Groups of type \texorpdfstring{$\ee_7$}{E7}} \label{E7.sec}

We will now relate FT systems as defined in the previous section to affine group schemes of type $\ee_7$.  Here is a tool that allows us to work with the quartic form $q$ as in \eqref{Frd.q.def} rather than the less-convenient 4-linear form $\Psi$, while still getting results that hold when 6 is not invertible.

\begin{lem} \label{perm.lem}
Let $(M, \Psi, b)$ be an FT system over $\ZZ$, let $\bfG$ be a closed subgroup of $\bfGL(M)$, and let $F$ be a field of characteristic zero.  If $\bfG(F)$ is dense in $\bfG$ (which holds if $\bfG$ is connected) and $\bfG(F)$ preserves $b \otimes F$ and the quartic form $q$ over $F$ defined by \eqref{q.recover}, then $\bfG$ is a closed sub-group-scheme of $\bfAut(M, \Psi, b)$.
\end{lem} 

\begin{proof}
Since $\bfG(F)$ is dense in $\bfG$, the group scheme $\bfG \times F$ preserves $b \otimes F$ and $q$, whence also $\Psi \otimes F$.  Viewing $b$ and $\Psi$ as elements of the representation $V := (M^*)^{\otimes d}$ of $\bfG$ for $d = 2$ or 4, the natural map $V^\bfG \otimes F \to (V \otimes F)^{\bfG \times F}$ is an isomorphism because $F$ is flat over $\ZZ$ \cite[Lemma 2]{Sesh:GR}, so $\bfG$ preserves $b$ and $\Psi$.
\end{proof}

\begin{cor} \label{E6inE7}
For every Freudenthal $R$-algebra $J$, there is an inclusion $f \colon \bfIsom(J) \hookrightarrow \bfAut(\Frd(J))$ via
\[
f(\phi) \stbtmat{\alpha}{x}{x'}{\alpha'} = \stbtmat{\alpha}{\phi(x)}{\phi^\dagger(x')}{\alpha'}.
\]
\end{cor}

\begin{proof}
Consider the case $J = \Her_3(\Zor(\ZZ))$.  For $\phi \in \bfIsom(J)(\QQ)$, it follows from the definition of $\phi^\dagger$ and Proposition \ref{STRUG}\eqref{STRUG.1} that $f(\phi)$ is an isomorphism of the bilinear and quartic forms $b \otimes \QQ$ and $q$ defined by \eqref{Frd.q.def} for $J \otimes \QQ$.  The lemma gives the claim in this case.  Base change and twisting gives the claim for every $R$ and every Albert $R$-algebra $J$.
\end{proof}

\begin{cor} \label{Frd.iso}
Suppose $J$ and $J'$ are Albert $R$-algebras.  If $J$ and $J'$ are isotopic, then $\bfAut(\Frd(J)) \cong \bfAut(\Frd(J'))$.
\end{cor}

\begin{proof}
The inclusions $\bfAut(J) \hookrightarrow \bfIsom(J) \hookrightarrow \bfAut(\Frd(J))$ induce maps
\[
H^1(R, \bfAut(J)) \to H^1(R, \bfIsom(J)) \to H^1(R, \bfAut(\Frd(J))),
\]
where the last set classifies FT systems over $R$.  The class of $J'$ in $H^1(R, \bfAut(J))$ maps to the class of $N_{J'}$ in $H^1(R, \bfIsom(J))$, and by hypothesis and by Proposition \ref{E6.iso} this is the trivial class.  Therefore, the image of $J'$ in $H^1(R, \bfAut(\Frd(J)))$, which is $\Frd(J')$, is the trivial class.
\end{proof}

In case $R$ is a field of characteristic $\ne 2, 3$, the converse of Corollary \ref{Frd.iso} is true by \cite[Cor.~6.9]{Ferr:strict}.  That is, if $\Frd(J) \cong \Frd(J')$, then $J$ and $J'$ are isotopic.

\begin{thm}  \label{E7.thm}
The group scheme $\bfAut(\Frd(\Her_3(\Zor(R))))$ over $R$ is obtained from the simply connected Chevalley group of type $\ee_7$ over $\ZZ$ by base change.  Every strongly inner and simply connected simple $R$-group scheme of type $\ee_7$ over $R$ is of the form $\bfAut(\Frd)$ for some FT system $\Frd$.  For FT systems $\Frd$ and $\Frd'$, $\bfAut(\Frd) \cong \bfAut(\Frd')$ if and only if $\Frd$ and $\Frd'$ are similar.
\end{thm}

\begin{proof}
Put $J_R := \Her_3(\Zor(R))$ and $\Frd_R := \Frd(J_R)$.  We will show that $\bfAut(Q_R)$ is isomorphic to the base change to $R$ of the simply connected Chevalley group $E_7$ over $\ZZ$.  

In addition to the sub-group-scheme $\bfIsom(J_R)$ of $\bfAut(\Frd_R)$ provided by Corollary \ref{E6inE7}, we consider a rank 1 torus $\Gm$ defined by 
\[
\beta \stbtmat{\alpha}{x}{x'}{\alpha'} = \stbtmat{\beta^{-3} \alpha}{\beta x}{\beta^{-1} x'}{\beta^3 \alpha'} \quad \text{for $\beta \in \Rx$}
\]
and two copies of $J_R$ (as group schemes under addition) through which an element $y \in J_R$ acts via
\[
y \stbtmat{\alpha}{x}{x'}{\alpha'} = \stbtmat{\alpha + b(x', y)}{x+\alpha' y}{x' + x \times y}{\alpha'} \quad \text{or} \quad \stbtmat{\alpha}{x+x' \times y}{x' + \alpha y}{\alpha' + b(x', y)}.
\]
These preserve $b$ and $q$, see for example \cite[p.~95]{Brown:E7} or \cite[p.~942]{Krut:E7}, and so by Lemma \ref{perm.lem} do belong to $\bfAut(\Frd_R)$.  Considering the Lie algebras of $\bfIsom(J_R)$, $\Gm$, and the two copies of $J$, as subalgebras of $\Lie(\bfGL(\Frd_R))$, one can identify the subalgebra $L_R$ they generate with the Lie algebra of $E_7 \times R$ by picking out specific root subalgebras and so on as in \cite{Frd:E7.1} or \cite{Sel:E7}, or see \cite[\S7]{G:struct} for partial information.  Note that $\Lie(\bfAut(\Frd_R)) \supseteq L_R$.  For $F$ any algebraically closed field, we may identify the smooth closed subgroup of $\bfAut(\Frd_F)$ generated by $\bfIsom(J_F)$, $\Gm$, and the two copies of $J_F$ with $E_7 \times F$.

In \cite{Lurie}, Lurie begins with $L_\ZZ$ and defines $L_\ZZ$-invariant 4-linear forms $\Theta^L$, $\Phi_i^L$, and $\Psi^L$ and alternating bilinear form $b^L$ on the 56-dimensional Weyl module of $L_\ZZ$.  Over $\CC$, $\bfAut(\Frd_\CC)$ is simply connected of type $\ee_7$ by the references in the previous paragraph, so it preserves the base change of Lurie's forms $\Theta^L \otimes \CC$, etc.  Because $\bfAut(\Frd_\CC)(\CC)$ is dense in $\bfAut(\Frd_\CC)$, Lemma \ref{perm.lem} shows that $\bfAut(\Frd_\ZZ)$ preserves $\Theta^L$, the $\Phi_i^L$, $\Psi^L$, and $b^L$.  By the uniqueness of $E_7$ invariant bilinear and symmetric 4-linear forms on $M$ (which follows from the uniqueness over $\CC$ as in the proof of Lemma \ref{perm.lem}), we find that $b^L = \pm b$ and $\Theta^L = \pm \Theta$.  Note that regardless of the sign on $b$ in the preceding sentence, we find $\Phi^L_i = \Phi_i$ for all $i$ and $\bfAut(\Frd_F)$ preserves $b^L$.
Now let $F$ be an algebraically closed field.  If $F$ has characteristic different from 2, then $\bfAut(\Frd_F)$ preserves $2 \Psi = \Theta + \sum \Phi_i$ and the $\Phi_i$, so it preserves $\Theta$, hence $\Theta^L$, hence $\Psi^L$.  If $F$ has characteristic 2, then although $\Psi^L = \frac12 (\pm \Theta + \sum \Phi_i)$ for some choice of sign as polynomials over $\ZZ$, we have $\Psi^L \otimes F = \Psi \otimes F$.  In either case, $\bfAut(\Frd_F)$ preserves $b^L \otimes F$ and $\Psi^L \otimes F$, whence so does its Lie algebra, so $\dim \Lie \bfAut(\Frd_F) \le \dim L_F$ by \cite[Th.~6.2.3]{Lurie}.  Putting this together with the previous paragraph, we see that $\bfAut(\Frd_F)$, an affine group scheme over the field $F$, is smooth with identity component $E_7 \times F$.

We claim that $\bfAut(\Frd_F)$ is connected.  Since its identity component $E_7$ has no outer automorphisms, every element of $\bfAut(\Frd_F)(F)$ is a product of an element of $E_7(F)$ and a linear transformation centralizing $E_7$.  
The action of $E_7 \times F$ on $\Frd_F$ is irreducible (it is the 56-dimensional minuscule representation), so the centralizer of $E_7$ consists of scalar transformations.  Finally, we note that the intersection of $\bfAut(\Frd_F)$ and the scalar transformations is the group scheme $\mu_2$ of square roots of unity, which is contained in $E_7$.  In summary, $\bfAut(\Frd_F) = E_7 \times F$ for every algebraically closed field $F$.

As in the proof of Lemma \ref{skip.aut.alb}, it follows that $\bfAut(\Frd_\ZZ)$ is a simple affine group scheme that is simply connected of type $\ee_7$, and we deduce from the fact that $\bfAut(\Frd_\RR)$ is split that $\bfAut(\Frd_\ZZ)$ is in fact the Chevalley group.

\smallskip
The second claim now follows by descent.  %, compare, for example, \cite[7.2.0.48]{CalmesFasel}. 

The third claim is proved in the same manner as Proposition \ref{E6.iso}, although the current situation is somewhat easier due to the absence of nontrivial automorphisms of the Dynkin diagram of $\ee_7$ and therefore the absence of outer automorphisms for semisimple groups of that type.  The sequence
\begin{equation} \label{E7.1}
H^1(R, \mu_2) \to H^1(R, \bfAut(\Frd)) \to H^1(R, \bfAut(\bfAut(\Frd)))
\end{equation}
is exact, where $\mu_2$ is the center of $\bfAut(\Frd)$ and $\bfAut(\bfAut(\Frd)) \cong \bfAut(\Frd_R)/\mu_2$ is the adjoint group. We have $\bfAut(\Frd') \cong \bfAut(\Frd)$ if and only if the element $\Frd'$ in $H^1(R, \bfAut(\Frd))$ is in the kernel of the second map in \eqref{E7.1}, if and only if $\Frd'$ is in the image of the first map.  To complete the proof, it suffices to calculate by descent that the action of $H^1(R, \mu_2)$ on $H^1(R, \bfAut(\Frd))$ is exactly by the similarity action defined in \S\ref{FT.sec}.
\end{proof}

\begin{cor}
If $R$ is (1) a complete discrete valuation ring whose residue field is finite; (2) a finite ring; or (3) a Dedekind domain whose field of fractions $F$ is a global field with no real embeddings, then the the split FT system is the only one over $R$, up to isomorphism.
\end{cor}

\begin{proof}
Imitate the arguments in Proposition \ref{G2F4.dvr} or Example \ref{dedekind.eg1}, where $\bfG$ is the base change to $R$ of the simply connected Chevalley group $\bfAut(\Frd(\Zor(\ZZ)))$.
\end{proof}

\begin{rmks*}
A previous work that considered groups of type $\ee_7$ over rings is \cite{Luzgarev}.   Aschbacher \cite{Asch:iso} studied the 4-linear form in the case where $R$ is a field of characteristic 2.  The paper \cite{MWeiss:FTS} studied the case of fields of any characteristic, organized around a polynomial law $\Theta \in \Pol(\Frd, R)$ that is not homogeneous.

For a field $F$ of characteristic $\ne 2, 3$, FT systems have been studied in this century in \cite{Clerc}, \cite{Helenius}, \cite{Krut:E7}, \cite{Sp:e7}, and \cite{Borsten:explicit} to name a few.  They arise naturally in the context of the bottom row of the magic triangle from \cite[Table 2]{DeligneGross}, in connection with the existence of extraspecial parabolic subgroups as in \cite{Roe:extra} or \cite[\S12]{G:lens}, or from groups with a $\bee\cee_1$ grading \cite[p.~995]{GrossG}.  For every Albert $F$-algebra $J$, the group scheme $\bfAut(\Frd(J))$ is isotropic, see for example \cite[Lemma 5.6(i)]{Sp:e7}.  Yet there exist strongly inner groups of type $\ee_7$ that are anisotropic, see \cite[3.1]{Ti:si} or \cite[App.~A]{G:lens}, and therefore there exist FT systems $\Frd$ that are not isomorphic to $\Frd(J)$ for any $J$.  
A construction that produces all FT systems can be obtained by considering a subgroup $\Isom(J) \rtimes \mu_4$ of $\bfAut(\Frd(J))$, which leads to a surjection $H^1(F, \Isom(J) \rtimes \mu_4) \to H^1(F, \bfAut(\Frd(J)))$, see \cite[12.13]{G:lens}, \cite[Lemma 4.15]{G:struct}, or \cite[\S4]{Sp:e7}.  
%
%The embedding $\mu_4 \hookrightarrow \bfAut(\Frd(J))$ from the preceding paragraph has an analogue that makes sense over an arbitrary ring $R$.  Consider the group scheme $\bfV \subset \bfGL_2$ whose $R$-points are the elements
%\[
%\left\{ \stbtmat{s}{0}{0}{s} \mid s \in \mu_2(R) \} \cup \left\{ \stbtmat{0}{\zeta}{\zeta}{0} \right)  \mid \zeta \in \mu_4(R) \setminus \mu_2(R) \right\}.
%\]
\end{rmks*}

%----------------------------------------------------------------------------------------------------------------
%
\section{Exceptional groups over \texorpdfstring{$\ZZ$}{Z}} \label{E6E7.sec}

We now record explicit descriptions of the isomorphism classes of semisimple affine group schemes over $\ZZ$ of types $\eff_4$, $\gee_2$, $\ee_6$, and $\ee_7$.

There are four such group schemes of type $\eff_4$, namely $\bfAut(J)$ for each of the four Albert $\ZZ$-algebras listed in Theorem \ref{skip.AlbZ}\eqref{skip.AlbZ.2}.  The proof of this fact is intertwined with the proof of that theorem.  Similarly, there are two such group schemes of type $\gee_2$, namely $\bfAut(C)$ for $C = \Zor(\ZZ)$ or $\oct$.

\begin{prop} \label{E6E7}
Regarding isomorphism classes of semisimple and simply connected affine group schemes over $\ZZ$:
\begin{enumerate}
\item there are two of type $\ee_6$, namely $\bfIsom(\Her_3(C))$ and
\item there are two of type $\ee_7$, namely $\bfAut(\Frd(\Her_3(C)))$
\end{enumerate}
for $C = \Zor(\ZZ)$ or $\oct$.
\end{prop}

\begin{proof}
Put $\bfG$ for the simply connected Chevalley group scheme over $\ZZ$ of type $\ee_n$, for $n =  6$ or 7.  By \cite[Remark 4.8]{Conrad:Z}, $\ZZ$ forms of absolutely simple and simply connected $\QQ$-group schemes are purely inner forms, i.e., are obtained by twisting $\bfG$ by a class $\xi \in H^1(\ZZ, \bfG)$.

For the split Albert algebra $J  = \Her_3(\Zor(\ZZ))$, the natural inclusions $\bfAut(J) \subset \bfIsom(J) \subset \bfAut(\Frd(J))$ give maps $H^1(R, \bfAut(J)) \to H^1(R, \bfG)$ for every ring $R$, where the domain is in bijection with the isomorphism classes of Albert $R$-algebras.  The groups in the statement we are aiming to prove are the image of $\Her_3(C)$ for $C = \Zor(\ZZ)$ or $\oct$.  The two choices for $C$, i.e., the two groups in the statement, give non-isomorphic groups over $\RR$, so it suffices to show that there are no others defined over $\ZZ$.

Now the compact real form of type $\ee_n$ is not a pure inner form, so $G_\xi \times \RR$ is not compact for all $\xi \in H^1(\ZZ, \bfG)$.  Therefore, the natural map $H^1(\ZZ, \bfG) \to H^1(\QQ, \bfG)$ is a bijection by \cite[Satz 4.2.4]{Hrdr:Ded}.  The natural map $H^1(\QQ, \bfG) \to H^1(\RR, \bfG)$ is also a bijection, a fact we have already used in Example \ref{alb.global}.  Since $H^1(\RR, \bfG)$ has two elements --- see \cite{BorovoiEvenor}, \cite[esp.~\S15]{BorovoiTim}, or \cite[Table 3]{AdamsTaibi} --- we have produced both elements of $H^1(\ZZ, \bfG)$.
\end{proof}

The proof provides the following corollary.

\begin{cor}
There are two isomorphism classes of FT systems over $\ZZ$, namely $\Frd(\Her_3(C))$ for $C = \Zor(\ZZ)$ or $\oct$.  $\hfill\qed$
\end{cor}

We have addressed now all the simple types that are usually called ``exceptional'', apart from $\ee_8$.  A classification of $\ZZ$-groups of type $\ee_8$ like Proposition \ref{E6E7} appears currently out of reach, because among those group schemes $\bfG$ over $\ZZ$ such that $\bfG \times \RR$ is the compact group of type $\ee_8$, there are at least 13,935 distinct isomorphism classes \cite[Prop.~5.3]{Gross:Z}.   Among those $\bfG$ over $\ZZ$ of type $\ee_8$ such that $\bfG \times \RR$ is not compact, the same argument as in the proof of Proposition \ref{E6E7} shows that there are two isomorphism classes.

%\bibliographystyle{amsalpha}
%\bibliography{skip_master}

\begin{thebibliography}{KMRT98}

\bibitem[ABB14]{AuelBB}
A.~Auel, M.~Bernardara, and M.~Bolognesi, \emph{Fibrations in complete
  intersections of quadrics, clifford algebras, derived categories, and
  rationality problems}, J. Math. Pures Appl. \textbf{102} (2014), 249--291.

\bibitem[AG19]{AlsaodyGille}
S.~Alsaody and P.~Gille, \emph{Isotopes of octonion algebras,
  {${\bf{G}}_2$}-torsors and triality}, Adv. Math. \textbf{343} (2019),
  864--909.

\bibitem[AHW19]{AsokHW}
A.~Asok, M.~Hoyois, and M.~Wendt, \emph{Generically split octonion algebras and
  $\mathbb{A}^1$-homotopy theory}, Algebra \& Number Theory \textbf{13} (2019),
  no.~3, 695--747.

\bibitem[AJ57]{AlbJac}
A.A. Albert and N.~Jacobson, \emph{On reduced exceptional simple {J}ordan
  algebras}, Ann. of Math. (2) \textbf{66} (1957), 400--417, (= Jacobson,
  Oe.~58).

\bibitem[Alb42]{Alb:na1}
A.A. Albert, \emph{Non-associative algebras. {I}. {F}undamental concepts and
  isotopy}, Annals Math. (2) \textbf{43} (1942), 687--707.

\bibitem[Alb58]{Albert:div}
\bysame, \emph{A construction of exceptional {J}ordan division algebras},
  Annals of Mathematics \textbf{67} (1958), 1--28.

\bibitem[Als21]{Alsaody}
S.~Alsaody, \emph{Albert algebras over rings and related torsors}, Canad. J.
  Math. \textbf{73} (2021), no.~3, 875--898.

\bibitem[Asc88]{Asch:iso}
M.~Aschbacher, \emph{Some multilinear forms with large isometry groups},
  Geometriae Dedicata \textbf{25} (1988), 417--465.

\bibitem[AT18]{AdamsTaibi}
J.~Adams and O.~Ta\"{\i}bi, \emph{Galois and {C}artan cohomology of real
  groups}, Duke Math. J. \textbf{167} (2018), no.~6, 1057--1097.

\bibitem[BC12]{BalmerCalmes}
P.~Balmer and B.~Calm{\`e}s, \emph{Bases of total {W}itt groups and
  lax-similitude}, J. Algebra Appl. \textbf{11} (2012), no.~3, 1250045.

\bibitem[BDF{\etalchar{+}}14]{Borsten:explicit}
L.~Borsten, M.J. Duff, S.~Ferrara, A.~Marrani, and W.~Rubens, \emph{Explicit
  orbit classification of reducible {J}ordan algebras and {F}reudenthal triple
  systems}, Comm. Math. Phys. \textbf{325} (2014), 17--39.

\bibitem[BE16]{BorovoiEvenor}
M.~Borovoi and Z.~Evenor, \emph{Real homogeneous spaces, {G}alois cohomology,
  and {R}eeder puzzles}, J. Algebra \textbf{467} (2016), 307--365.

\bibitem[BK66]{BK}
H.~Braun and M.~Koecher, \emph{Jordan-{A}lgebren}, Springer-Verlag, Berlin,
  1966.

\bibitem[Bou88]{Bou:alg2}
N.~Bourbaki, \emph{Algebra {II}: Chapters 4--7}, Springer-Verlag, 1988.

\bibitem[Bro69]{Brown:E7}
R.B. Brown, \emph{Groups of type ${E}_7$}, J. Reine Angew. Math. \textbf{236}
  (1969), 79--102.

\bibitem[BT21]{BorovoiTim}
M.~Borovoi and D.A. Timashev, \emph{Galois cohomology of real semisimple groups
  via {K}ac labelings}, Transf. Groups \textbf{26} (2021), no.~2, 443--477.

\bibitem[CF15]{CalmesFasel}
B.~Calm{\`e}s and J.~Fasel, \emph{Groupes classiques}, Panoramas \&
  Synth{\`e}ses \textbf{46} (2015), 1--133.

\bibitem[CG06]{CG}
M.~Carr and S.~Garibaldi, \emph{Geometries, the principle of duality, and
  algebraic groups}, Expo. Math. \textbf{24} (2006), 195--234.

\bibitem[CG21]{ChayetG}
M.~Chayet and S.~Garibaldi, \emph{A class of continuous non-associative
  algebras arising from algebraic groups including ${E}_8$}, Forum of
  Mathematics: Sigma \textbf{9} (2021), no.~e6.

\bibitem[Cle03]{Clerc}
J.-L. Clerc, \emph{Special prehomogeneous vector spaces associated to ${F}_4,
  {E}_6, {E}_7, {E_8}$ and simple {J}ordan algebras of rank 3}, J. Algebra
  \textbf{264} (2003), 98--128.

\bibitem[Con14]{Conrad:Z}
B.~Conrad, \emph{Non-split reductive groups over $\mathbf{Z}$}, Autour des
  Sch\'emas en Groupes II, Panoramas et synth{\`{e}}ses, vol.~46, Soci\'et\'e
  Math\'ematique de France, 2014, pp.~193--253.

\bibitem[Cox46]{Coxeter:cayley}
H.S.M. Coxeter, \emph{Integral {C}ayley numbers}, Duke Math. J. \textbf{13}
  (1946), 561--578.

\bibitem[CS03]{ConwaySmith}
J.H. Conway and D.A. Smith, \emph{On quaternions and octonions: their geometry,
  arithmetic, and symmetry}, A K Peters, Ltd., Natick, MA, 2003.

\bibitem[DG70]{SGA3.3}
M.~Demazure and A.~Grothendieck, \emph{Sch{\'{e}}mas en groupes {III}:
  {S}tructure des schemas en groupes reductifs}, Lecture Notes in Mathematics,
  vol. 153, Springer, 1970.

\bibitem[DG02]{DeligneGross}
P.~Deligne and B.H. Gross, \emph{On the exceptional series, and its
  descendants}, C. R. Math. Acad. Sci. Paris \textbf{335} (2002), no.~11,
  877--881.

\bibitem[Dic23]{Dickson:hyper}
L.E. Dickson, \emph{A new simple theory of hypercomplex integers}, Journal de
  Math{\'e}matiques Pures et Appliqu{\'e}es \textbf{29} (1923), 281--326.

\bibitem[EG82]{EstesGur}
D.R. Estes and R.M. Guralnick, \emph{Module equivalences: local to global when
  primitive polynomials represent units}, J. Algebra \textbf{77} (1982),
  138--157.

\bibitem[EG96]{ElkiesGross}
N.D. Elkies and B.H. Gross, \emph{The exceptional cone and the {L}eech
  lattice}, Internat. Math. Res. Notices (1996), no.~14, 665--698.

\bibitem[EKM08]{EKM}
R.S. Elman, N.~Karpenko, and A.~Merkurjev, \emph{The algebraic and geometric
  theory of quadratic forms}, Colloquium Publications, vol.~56, Amer. Math.
  Soc., 2008.

\bibitem[Fer72]{Ferr:strict}
J.C. Ferrar, \emph{Strictly regular elements in {F}reudenthal triple systems},
  Trans. Amer. Math. Soc. \textbf{174} (1972), 313--331.

\bibitem[FR17]{FirstRei}
U.A. First and Z.~Reichstein, \emph{On the number of generators of an algebra},
  Comptes Rendus Mathematique \textbf{355} (2017), no.~1, 5--9.

\bibitem[Fre54]{Frd:E7.1}
H.~Freudenthal, \emph{Beziehungen der ${E}_7$ und ${E}_8$ zur {O}ktavenebene.
  {I}}, Nederl. Akad. Wetensch. Proc. Ser. A \textbf{57} (1954), 218--230.

\bibitem[Fre85]{Frd:OAO}
\bysame, \emph{{O}ktaven, {A}usnahmegruppen und {O}ktavengeometrie}, Geom.
  Dedicata \textbf{19} (1985), 7--63, (Reprint of Utrecht Lecture Notes, 1951).

\bibitem[FRW22]{FirstReiW}
U.A. First, Z.~Reichstein, and B.~Williams, \emph{On the number of generators
  of an algebra over a commutative ring}, to appear in \emph{Trans. Amer. Math.
  Soc.}, 2022.

\bibitem[Gar01]{G:struct}
S.~Garibaldi, \emph{Structurable algebras and groups of type ${E}_6$ and
  ${E}_7$}, J. Algebra \textbf{236} (2001), no.~2, 651--691.

\bibitem[Gar09]{G:lens}
\bysame, \emph{Cohomological invariants: exceptional groups and spin groups},
  Memoirs Amer. Math. Soc., no. 937, Amer. Math. Soc., Providence, RI, 2009,
  with an appendix by Detlev W. Hoffmann.

\bibitem[GG21]{GrossG}
B.H. Gross and S.~Garibaldi, \emph{Minuscule embeddings}, Indag. Math.
  \textbf{32} (2021), no.~5, 987--1004.

\bibitem[Gil14]{Gille:oct}
P.~Gille, \emph{Octonion algebras over rings are not determined by their
  norms}, Canad. Math. Bull. \textbf{57} (2014), no.~2, 303--309.

\bibitem[Gil19]{Gille:sc2}
\bysame, \emph{Groupes alg\'ebriques semi-simples en dimension cohomologique
  $\le 2$}, Lecture Notes in Math., no. 2238, Springer, 2019.

\bibitem[Gir71]{Giraud}
J.~Giraud, \emph{Cohomologie non ab\'elienne}, Die Grundlehren der
  mathematischen Wissenschaften, vol. 179, Springer-Verlag, Berlin, 1971.

\bibitem[Gro96]{Gross:Z}
B.H. Gross, \emph{Groups over {${\bf Z}$}}, Invent. Math. \textbf{124} (1996),
  no.~1-3, 263--279.

\bibitem[GY03]{GanYu:G2}
W.T. Gan and J.-K. Yu, \emph{Sch{\'e}mas en groupes et immeubles des groupes
  exceptionnels sur un corps local. {P}remi{\`e}re partie: le groupe ${G}_2$},
  Bull. Soc. Math. France \textbf{131} (2003), no.~3, 307--358.

\bibitem[Har66]{Hrdr:2}
G.~Harder, \emph{\"{U}ber die {G}aloiskohomologie halbeinfacher
  {M}atrizengruppen. {II}}, Math. Z. \textbf{92} (1966), 396--415.

\bibitem[Har67]{Hrdr:Ded}
\bysame, \emph{Halbeinfache {G}ruppenschemata {\"u}ber {D}edekindringen},
  Invent. Math. \textbf{4} (1967), 165--191.

\bibitem[Hel12]{Helenius}
F.~Helenius, \emph{Freudenthal triple systems by root system methods}, J.
  Algebra \textbf{357} (2012), 116--137.

\bibitem[Hij63]{Hij}
H.~Hijikata, \emph{A remark on the groups of type ${G}_2$ and ${F}_4$}, J.
  Math. Soc. Japan \textbf{15} (1963), 159--164.

\bibitem[Jac61]{Jac:J3}
N.~Jacobson, \emph{Some groups of transformations defined by {J}ordan algebras.
  {III}}, J. Reine Angew. Math. \textbf{207} (1961), 61--85, (= Coll.\ Math.\
  Papers 65).

\bibitem[Jac68]{Jac:J}
\bysame, \emph{Structure and representations of {J}ordan algebras}, Coll.\
  Pub., vol.~39, Amer. Math. Soc., Providence, RI, 1968.

\bibitem[Jac69]{Jac:Tata}
\bysame, \emph{Lectures on quadratic {J}ordan algebras}, Tata Institute of
  Fundamental Research, Bombay, 1969, Tata Institute of Fundamental Research
  Lectures on Mathematics, No. 45.

\bibitem[Jac71]{Jac:ex}
\bysame, \emph{Exceptional {L}ie algebras}, Lecture notes in pure and applied
  mathematics, vol.~1, Marcel-Dekker, New York, 1971.

\bibitem[KMRT98]{KMRT}
M.-A. Knus, A.S. Merkurjev, M.~Rost, and J.-P. Tignol, \emph{The book of
  involutions}, Colloquium Publications, vol.~44, Amer.\ Math.\ Soc., 1998.

\bibitem[Knu91]{Knus:qf}
M.-A. Knus, \emph{Quadratic and {H}ermitian forms over rings}, Springer-Verlag,
  Berlin, 1991.

\bibitem[KO74]{KO}
M.-A. Knus and M.~Ojanguren, \emph{Th\'eorie de la descente et alg\`ebras
  d'{A}zumaya}, Lecture Notes in Mathematics, vol. 389, Springer-Verlag, 1974.

\bibitem[Kru02]{Krut:E6}
S.~Krutelevich, \emph{On a canonical form of a $3 \times 3$ {H}ermitian matrix
  over the ring of integral split octonions}, J. Algebra \textbf{253} (2002),
  276--295.

\bibitem[Kru07]{Krut:E7}
\bysame, \emph{Jordan algebras, exceptional groups, and {B}hargava
  composition}, J. Algebra \textbf{314} (2007), no.~2, 924--977.

\bibitem[Loo06]{Loos:genalg}
O.~Loos, \emph{Generically algebraic {J}ordan algebras over commutative rings},
  J. Algebra \textbf{297} (2006), 474--529.

\bibitem[LPR08]{LoosPR}
O.~Loos, H.P. Petersson, and M.L. Racine, \emph{Inner derivations of
  alternative algebras over commutative rings}, Algebra Number Theory
  \textbf{2} (2008), no.~8, 927--968.

\bibitem[Lur01]{Lurie}
J.~Lurie, \emph{On simply laced {L}ie algebras and their minuscule
  representations}, Comment. Math. Helv. \textbf{76} (2001), 515--575.

\bibitem[Luz13]{Luzgarev}
A.Yu. Luzgarev, \emph{Fourth-degree invariants for ${G}({E}_7, {R})$ not
  depending on the characteristic}, Vestnik St. Petersburg University.
  Mathematics \textbf{46} (2013), no.~1, 29--34.

\bibitem[Mah42]{Mahler}
K.~Mahler, \emph{On ideals in the {C}ayley-{D}ickson algebra}, Proc. Roy. Irish
  Acad. Sect. A. \textbf{48} (1942), 123--133.

\bibitem[McC66]{McC:NAS}
K.~McCrimmon, \emph{A general theory of {J}ordan rings}, Proc. Nat. Acad. Sci.
  U.S.A. \textbf{56} (1966), 1072--1079.

\bibitem[McC69]{McC:FST}
\bysame, \emph{The {F}reudenthal-{S}pringer-{T}its constructions of exceptional
  {J}ordan algebras}, Trans. Amer. Math. Soc. \textbf{139} (1969), 495--510.

\bibitem[McC70]{McC:FSTrev}
\bysame, \emph{The {F}reudenthal-{S}pringer-{T}its constructions revisited},
  Trans. Amer. Math. Soc. \textbf{148} (1970), 293--314.

\bibitem[McC71a]{McC:homo}
\bysame, \emph{Homotopes of alternative algebras}, Math. Ann. \textbf{191}
  (1971), 253--262.

\bibitem[McC71b]{McC:inn}
\bysame, \emph{Inner ideals in quadratic {J}ordan algebras}, Trans. Amer. Math.
  Soc. \textbf{159} (1971), 445--468.

\bibitem[McC85]{McC:scalar}
\bysame, \emph{Nonassociative algebras with scalar involution}, Pacific J.
  Math. \textbf{116} (1985), no.~1, 85--109.

\bibitem[McC04]{McC}
\bysame, \emph{A taste of {J}ordan algebras}, Universitext, Springer-Verlag,
  New York, 2004.

\bibitem[Mey68]{Meyb:FT}
K.~Meyberg, \emph{Eine {T}heorie der {F}reudenthalschen {T}ripelsysteme. {I},
  {II}}, Nederl. Akad. Wetensch. Proc. Ser. A 71 = Indag. Math. \textbf{30}
  (1968), 162--190.

\bibitem[Mil17]{Milne:AG}
J.S. Milne, \emph{Algebraic groups: the theory of group schemes of finite type
  over a field}, Cambridge studies in Advanced Math., vol. 170, Cambridge
  University Press, 2017.

\bibitem[MW81]{McDWa}
B.R. McDonald and W.C. Waterhouse, \emph{Projective modules over rings with
  many units}, Proc. Amer. Math. Soc. \textbf{83} (1981), no.~3, 455--458.

\bibitem[MW19]{MWeiss:FTS}
B.~M\"uhlherr and R.M. Weiss, \emph{Freudenthal triple systems in arbitrary
  characteristic}, J. Algebra \textbf{520} (2019), 237--275.

\bibitem[MZ88]{McCZ}
K.~McCrimmon and E.~Zel'manov, \emph{The structure of strongly prime quadratic
  {J}ordan algebras}, Adv. in Math. \textbf{69} (1988), no.~2, 133--222.

\bibitem[Pet93]{Ptr:compzero}
H.P. Petersson, \emph{Composition algebras over algebraic curves of genus
  zero}, Trans. Amer. Math. Soc. \textbf{337} (1993), no.~1, 473--493.

\bibitem[Pet19]{Ptr:surv}
\bysame, \emph{A survey on {A}lbert algebras}, Transf. Groups \textbf{24}
  (2019), no.~1, 219--278.

\bibitem[Pet21]{PeMS}
\bysame, \emph{Norm equivalences and ideals of composition algebras},
  M{\"u}nster {J}. of {M}ath. \textbf{14} (2021), 283--293.

\bibitem[PR84a]{PR:Sp}
H.P. Petersson and M.L. Racine, \emph{Springer forms and the first {T}its
  construction of exceptional {J}ordan division algebras}, Manuscripta Math.
  \textbf{45} (1984), 249--272.

\bibitem[PR84b]{PR:toral}
\bysame, \emph{The toral {T}its process of {J}ordan algebras}, Abh. Math. Sem.
  Univ. Hamburg \textbf{54} (1984), 251--256.

\bibitem[PR85]{PR:rads}
\bysame, \emph{Radicals of {J}ordan algebras of degree {$3$}}, Radical theory
  (Eger, 1982), Colloq. Math. Soc. J\'anos Bolyai, vol.~38, North-Holland,
  Amsterdam, 1985, pp.~349--377.

\bibitem[PR86a]{PR:class}
\bysame, \emph{Classification of algebras arising from the {T}its process}, J.
  Algebra \textbf{98} (1986), no.~1, 244--279.

\bibitem[PR86b]{PR:jord3}
\bysame, \emph{Jordan algebras of degree {$3$} and the {T}its process}, J.
  Algebra \textbf{98} (1986), no.~1, 211--243.

\bibitem[PR94]{PlatRap}
V.P. Platonov and A.~Rapinchuk, \emph{Algebraic groups and number theory},
  Academic Press, Boston, 1994.

\bibitem[PR96]{PR:el}
H.P. Petersson and M.L. Racine, \emph{An elementary approach to the
  {S}erre-{R}ost invariant of {A}lbert algebras}, Indag. Math. (N.S.)
  \textbf{7} (1996), no.~3, 343--365.

\bibitem[PST99]{PST:Ti}
R.~Pariamala, R.~Sridharan, and M.L. Thakur, \emph{Tits' constructions of
  {J}ordan algebras and ${F}_4$ bundles on the plane}, Compositio Math.
  \textbf{119} (1999), 13--40.

\bibitem[Rac77]{Racine:point}
M.L. Racine, \emph{Point spaces in exceptional quadratic {J}ordan algebras}, J.
  Algebra \textbf{46} (1977), 22--36.

\bibitem[Rob63]{Roby}
N.~Roby, \emph{Lois polynomes et lois formelles in th\'eorie des modules}, Ann.
  Sci. \'Ecole Norm. Sup. (3) \textbf{80} (1963), no.~3, 213--348.

\bibitem[R{\"o}h93]{Roe:extra}
G.~R{\"o}hrle, \emph{On extraspecial parabolic subgroups}, Linear algebraic
  groups and their representations (Los Angeles, CA, 1992), Contemp. Math.,
  vol. 153, Amer. Math. Soc., Providence, RI, 1993, pp.~143--155.

\bibitem[Ros91]{Rost:CR}
M.~Rost, \emph{A (mod 3) invariant for exceptional {J}ordan algebras}, C. R.
  Acad. Sci. Paris S{\'e}r. I Math. \textbf{313} (1991), 823--827.

\bibitem[Sch94]{Schfr}
R.~Schafer, \emph{An introduction to nonassociative algebras}, Dover, 1994.

\bibitem[Sel63]{Sel:E7}
G.B. Seligman, \emph{On the split exceptional {L}ie algebra ${E}_7$}, dittoed
  notes, 28 pages, 1963.

\bibitem[Ser02]{SeCG}
J-P. Serre, \emph{Galois cohomology}, Springer-Verlag, 2002, originally
  published as \emph{Cohomologie galoisienne} (1965).

\bibitem[Ses77]{Sesh:GR}
C.S. Seshadri, \emph{Geometric reductivity over arbitrary base}, Adv. Math.
  \textbf{26} (1977), 225--274.

\bibitem[Spr62]{Sp:cubic}
T.A. Springer, \emph{Characterization of a class of cubic forms}, Nederl.\
  Akad.\ Wetensch. \textbf{65} (1962), 259--265.

\bibitem[Spr73]{Sp:jord}
\bysame, \emph{Jordan algebras and algebraic groups}, Ergebnisse der Mathematik
  und ihrer Grenzgebiete, vol.~75, Springer-Verlag, 1973.

\bibitem[Spr06]{Sp:e7}
\bysame, \emph{Some groups of type ${E}_7$}, Nagoya Math. J. \textbf{182}
  (2006), 259--284.

\bibitem[{Sta}18]{stacks-project}
The {Stacks Project Authors}, \emph{Stacks project},
  \url{https://stacks.math.columbia.edu}, 2018.

\bibitem[SV00]{Sp:ex}
T.A. Springer and F.D. Veldkamp, \emph{Octonions, {J}ordan algebras and
  exceptional groups}, Springer-Verlag, Berlin, 2000.

\bibitem[Tit90]{Ti:si}
J.~Tits, \emph{Strongly inner anisotropic forms of simple algebraic groups}, J.
  Algebra \textbf{131} (1990), 648--677.

\bibitem[vdBS59]{vdBSp:G2}
F.~van~der {B}lij and T.A. Springer, \emph{The arithmetics of octaves and of
  the group ${G}_2$}, Nederl.\ Akad.\ Wetensch. \textbf{62} (1959), 406--418.

\bibitem[Voi21]{Voight}
J.~Voight, \emph{Quaternion algebras}, Springer, 2021.

\bibitem[Wat79]{Wa}
W.C. Waterhouse, \emph{Introduction to affine group schemes}, Graduate Texts in
  Mathematics, vol.~66, Springer, 1979.

\bibitem[Wei60]{Weil}
A.~Weil, \emph{Algebras with involution and the classical groups}, J. Indian
  Math. Soc. \textbf{24} (1960), 589--623.

\bibitem[Zor33]{Zorn}
M.~Zorn, \emph{Alternativk\"{o}rper und quadratische {S}ysteme}, Abh. Math.
  Semin. Hamburg. Univ. \textbf{9} (1933), 395--402.

\end{thebibliography}

\newcommand{\etalchar}[1]{$^{#1}$}
\providecommand{\bysame}{\leavevmode\hbox to3em{\hrulefill}\thinspace}
\providecommand{\MR}{\relax\ifhmode\unskip\space\fi MR }
% \MRhref is called by the amsart/book/proc definition of \MR.
\providecommand{\MRhref}[2]{%
  \href{http://www.ams.org/mathscinet-getitem?mr=#1}{#2}
}
\providecommand{\href}[2]{#2}

\end{document}